\documentclass[12pt]{amsart}
\usepackage{amsfonts}
\usepackage{mathrsfs}
\usepackage{bbm}%
\usepackage{bbding}
\usepackage[mathcal]{euscript}
\usepackage{graphicx}
\usepackage{amsthm,amscd}
\usepackage{calc}
\usepackage{mathrsfs,dsfont}
\usepackage{CJK,fancyhdr,amscd}
\usepackage{anysize}
\usepackage{amsmath,amssymb,amsfonts}
\usepackage{epsfig,enumerate}
\usepackage{latexsym}
\usepackage{indentfirst, latexsym}
\usepackage{graphics}
\usepackage[all,poly,knot]{xy}
\usepackage[pagebackref]{hyperref}

\providecommand{\U}[1]{\protect\rule{.1in}{.1in}}

\numberwithin{equation}{section}

\newtheorem{theorem} {Theorem} [section]
\newtheorem{proposition}[theorem]{Proposition}
\newtheorem{corollary}  [theorem]     {Corollary}
\newtheorem{lemma}  [theorem]     {Lemma}
\newtheorem{example}  [theorem]     {Example}
\newtheorem{question}  [theorem]     {Question}
\newtheorem{prob}[theorem]{Problem}
\newtheorem{remark}  [theorem]     {Remark}
\newtheorem{definition}  [theorem]     {Definition}
\newtheorem{notation}  [theorem]     {Notation}

\newtheorem{conjecture}  [theorem]     {Conjecture}

\newcommand{\dz}{d \bar{z}}
\newcommand{\pz}{\partial\bar{z}}
\newcommand{\w}{\wedge}
\newcommand{\im}{\mathrm{im}}
\renewcommand{\1}{\mathds{1}}
\newcommand{\G}{\mathbb{G}}
\newcommand{\db}{\overline{\partial}}
\newcommand{\dbs}{\overline{\partial}^*}

\newcommand{\lc}{\lrcorner}

\newcommand{\lk}{\left(}
\newcommand{\rk}{\right)}

\newcommand{\btheorem}{\begin{theorem}}
\newcommand{\etheorem}{\end{theorem}}
\newcommand{\bproposition}{\begin{proposition}}
\newcommand{\eproposition}{\end{proposition}}
\newcommand{\bdefinition}{\begin{definition}}
\newcommand{\edefinition}{\end{definition}}
\newcommand{\bcorollary}{\begin{corollary}}
\newcommand{\ecorollary}{\end{corollary}}
\newcommand{\bproof}{\begin{proof}}
\newcommand{\eproof}{\end{proof}}
\newcommand{\beq}{\begin{equation}}
\newcommand{\eeq}{\end{equation}}
\newcommand{\ee}{\end{eqnarray*}}
\newcommand{\be}{\begin{eqnarray*}}

\newcommand{\elemma}{\end{lemma}}
\newcommand{\blemma}{\begin{lemma}}
\newcommand{\ts}{\otimes}
\newcommand{\om}{\omega}
\renewcommand{\>}{\rightarrow}
\newcommand{\sL}{{\mathcal L}}
\newcommand{\p}{\partial}
\newcommand{\bp}{\overline\partial}

\newcommand{\bd}{\begin{enumerate} }
\newcommand{\ed}{\end{enumerate}}

\def\p{\partial}
\def\o{\overline}
\def\b{\bar}

\def\mc{\mathcal}

\def\w{\wedge}
\def\om{\omega}

\def\l{\lrcorner}

\begin{document}
\title{Several special complex structures and their deformation properties}

\author{Sheng Rao}
\address{School of Mathematics and statistics, Wuhan  University,
Wuhan 430072, China; Department of Mathematics, University of
California at Los Angeles, CA 90095-1555, USA}
\email{likeanyone@whu.edu.cn}
\thanks{Rao is partially supported by the National Natural Science Foundations of China No. 11301477, 11671305
and China Scholarship Council/University of California, Los Angeles Joint Scholarship
Program.  The corresponding author Zhao is partially supported by the
Fundamental Research Funds for the Central Universities No.
CCNU16A05013 and China Postdoctoral Science
Foundation No. 2016M592356}

\author{Quanting Zhao}
\address{School of Mathematics and Statistics \&
Hubei Key Laboratory of Mathematical Sciences,
Central China Normal University, Wuhan, 430079, P.R.China.}
\email{zhaoquanting@126.com; zhaoquanting@mail.ccnu.edu.cn}

\date{\today}


\begin{abstract}
We introduce a natural map from the space of pure-type complex
differential forms on a complex manifold to the corresponding one on
the infinitesimal deformations of this complex manifold. By use of
this map, we generalize an extension formula in a recent work of K.
Liu, X. Yang and the first author. As direct corollaries, we prove
several deformation invariance theorems for Hodge numbers. Moreover,
we also study the Gauduchon cone and its relation with the balanced
cone in the K\"ahler case, and show that the limit of the Gauduchon
cone in the sense of D. Popovici for a generic fiber in a
K\"ahlerian family is contained in the closure of the Gauduchon cone
for this fiber.
\end{abstract}
\maketitle

\tableofcontents

\section{Introduction}

We introduce an extension map from the space of complex differential
forms on a complex manifold to the corresponding one on the
infinitesimal deformations of the complex manifold and generalize an
extension formula in \cite{lry} with more complete deformation
significance. As direct corollaries, we prove several deformation
invariance theorems for Hodge numbers in sufficiently general situations
by a power series approach, which is analogously used to reprove the classical Kodaira-Spencer's local stability
of K\"ahler structures in a recent paper \cite{RwZ}.
We will also study the Gauduchon cone and its relation with the balanced one
in the K\"ahler case, to explore the deformation properties on the Gauduchon cone of
an \textbf{sGG} manifold introduced by D. Popovici \cite{P1}. We are
much motivated by Popovici's remarkable work on \cite[Conjecture
1.1]{P3}, which confirms that if the central fiber $X_0$ of a
holomorphic family of complex manifolds admits the deformation
invariance of $(0,1)$-type Hodge numbers or a so-called strongly
Gauduchon metric and the generic fiber $X_t$ ($t\neq 0$) of this
family is projective, then $X_0$ is Moishezon.

We will mostly follow the notations in \cite{lry}.  All manifolds in
this paper are assumed to be $n$-dimensional compact complex
manifolds. A {\emph{Beltrami differential}} is an element in $A^{0,1}(X,
T^{1,0}_X)$, where $T^{1,0}_X$ denotes the holomorphic tangent
bundle of $X$. Then $i_\phi$ or $\phi\lrcorner$ denotes the
contraction operator with $\phi\in A^{0,1}(X,T^{1,0}_X)$
alternatively if there is no confusion. We also follow the
convention
\begin{equation}\label{0e-convention}
e^{\spadesuit}=\sum_{k=0}^\infty \frac{1}{k!} \spadesuit^{k},
\end{equation}
where $\spadesuit^{k}$ denotes $k$-time action of the operator
$\spadesuit$. Since the dimension of $X$ is finite, the summation in
the above formulation is always finite.

Consider the smooth family $\pi: \mathcal{X} \rightarrow
B$ of $n$-dimensional complex manifolds
over a small domain $B$ in $\mathbb{R}^k$ as in Definition \ref{smf}, with the
central fiber $X_0:= \pi^{-1}(0)$ and the general fibers $X_t:=
\pi^{-1}(t).$ Set $k=1$ for simplicity. Denote by
$\zeta:=(\zeta^\alpha_j(z,t))_{\alpha=1}^n$ the holomorphic
coordinates of $X_t$ induced by the family with the holomorphic
coordinates $z:=(z^i)_{i=1}^n$ of $X_0$, under a coordinate covering
$\{\mathcal{U}_j\}$ of $\mathcal{X}$, when $t$ is assumed to be
fixed. Suppose that this family induces the integrable Beltrami
differential $\varphi(z,t)$, which is denoted by $\varphi(t)$ and
$\varphi$ interchangeably. These are reviewed at the beginning of
Section \ref{ext-formula}. Then we have the following crucial
calculation: \blemma[=Lemma \ref{inverse}]
$$\begin{pmatrix} \frac{\p z}{\p \zeta} & \frac{\p z}{\p \bar{\zeta}}  \\
\frac{\p \bar{z}}{\p \zeta} & \frac{\p \bar{z}}{\p \bar{\zeta}}  \\
\end{pmatrix} = \begin{pmatrix}
\lk \1 - \varphi \overline{\varphi} \rk^{-1} \lk \frac{\p \zeta}{\p
z}
\rk^{-1} & - \varphi \lk \1- \overline{\varphi} \varphi \rk^{-1} \lk \overline{ \frac{\p \zeta}{\p z} } \rk ^{-1}  \\
- \lk \1- \overline{\varphi} \varphi \rk^{-1} \overline{ \varphi }
\lk \frac{\p \zeta}{\p z} \rk^{-1} & \lk \overline{ \1 - \varphi
\overline{\varphi} } \rk ^{-1} \lk \overline{ \frac{\p \zeta}{\p z} } \rk^{-1}   \\
\end{pmatrix},$$
where $\varphi \overline{\varphi}$, $\overline{\varphi} \varphi$
stand for the two matrices $(
\varphi^i_{\bar{k}} \overline{\varphi^k_{\bar{j}}})_{\begin{subarray}{l} 1 \leq i \leq n \\
1 \leq j \leq n \\ \end{subarray}}$,
$(\overline{\varphi^i_{\bar{k}}} \varphi^k_{\bar{j}})_{\begin{subarray}{l} 1 \leq i \leq n \\
1 \leq j \leq n \\ \end{subarray}}$, respectively, and $\1$ is the
identity matrix. \elemma

Using this calculation and its corollaries, we are able to reprove
an important result (Proposition \ref{hol}) in deformation theory of
complex structures, which asserts that the holomorphic structure on
$X_t$ is determined by $\varphi(t)$. Actually, we obtain that for a
differentiable function $f$ defined on an open subset of $X_0$
$$
\db_t f=e^{i_{\overline{\varphi}}}\lk\lk
\1-\overline{\varphi}\varphi\rk ^{-1}\lc(\db-\varphi\lc \p)f\rk,
$$
where the differential operator $d$ is decomposed as $d=\p_t + \db_t
$ with respect to the holomorphic structure on $X_t$ and
$e^{i_{\overline{\varphi}}}$ follows the notation
\eqref{0e-convention}.

Motivated by the new proof of Proposition \ref{hol}, we
introduce a map
$$e^{i_{\varphi(t)}|i_{\overline{\varphi(t)}}}:
 A^{p,q}(X_0)\> A^{p,q}(X_t),$$ which plays an important role in
 this paper and is given in Definition \ref{map}. This map
 is a real linear isomorphism as $t$ is arbitrarily small.  Based on this, we
 achieve:

\begin{proposition}[=Proposition \ref{extension-in}]\label{10extension-in}
For any $\alpha\in A^{*,*}(X_0)$,
$$\b{\p}_t(e^{i_{\varphi}|i_{\b{\varphi}}}(\alpha))=0$$
amounts to
$$([\p,i_{\varphi}]+\b{\p})(\1-\b{\varphi}\varphi)\Finv\alpha=0,$$
where '$\Finv$' is the simultaneous  contraction introduced in Subsection \ref{obs-eqn}.
\end{proposition}

This proposition provides a criterion for a specific $\db$-extension from $A^{p,q}(X_0)$ to $A^{p,q}(X_t)$
and generalizes \cite[Theorem 3.4]{lry} (or Proposition \ref{main1}) in deformation significance.
As a direct application of Proposition \ref{10extension-in},
we consider the deformation invariance of Hodge numbers.
Before stating the main theorems in Section \ref{def}, we recall several definitions of related cohomology groups
and mappings.

Let $X$ be a compact complex manifold of complex dimension $n$
with the following commutative diagram
$$\xymatrix{      & H^{p,q}_{\p}(X) \ar[dr]^{\iota^{p,q}_{\p,A}} &            \\
 H^{p,q}_{BC}(X) \ar[ur]^{\iota^{p,q}_{BC,\p}} \ar[dr]_{\iota^{p,q}_{BC,\db}} \ar[rr]^{\iota^{p,q}_{BC,A}} & &  H^{p,q}_{A}(X)  \\
                & H^{p,q}_{\db}(X) \ar[ur]_{\iota^{p,q}_{\db,A}} &          .}$$
\emph{Dolbeault cohomology groups} $H^{\bullet,\bullet}_{\db}(X)$ of $X$ are defined by:
$$H^{\bullet,\bullet}_{\db}(X):=\frac{\ker\db}{\im\ \db},$$
with $H^{\bullet,\bullet}_{\p}(X)$ similarly defined, while \emph{Bott-Chern and Aeppli cohomology groups} are defined as
$$H^{\bullet,\bullet}_{BC}(X):=\frac{\ker \p\cap \ker\db}{\im\ \p\db}\quad
\text{and}\quad H^{\bullet,\bullet}_{A}(X):=\frac{\ker \p\db}{\im\ \p+\im\ \db},$$
respectively. The dimensions of $H^{p,q}_{\db}(X)$, $H^{p,q}_{BC}(X)$, $H^{p,q}_{A}(X)$ and $H^{p,q}_{\p}(X)$ over $\mathbb{C}$
are denoted by $h^{p,q}_{\db}(X)$, $h^{p,q}_{BC}(X)$, $h^{p,q}_{A}(X)$ and $h^{p,q}_{\p}(X)$, respectively, the first three of which
are usually called $(p,q)$-\emph{Hodge numbers}, \emph{Bott-Chern numbers} and \emph{Aeppli numbers}. From the very definition of these cohomology groups,
the following equalities clearly hold
\[ h^{p,q}_{BC} = h^{q,p}_{BC}=h^{n-q,n-p}_{A}=h^{n-p,n-q}_{A},  h^{n-p,n-q}_{\db}=h^{p,q}_{\db}=h^{q,p}_{\p}=h^{n-q,n-p}_{\p}.\]

Now let us describe our basic philosophy to consider the deformation invariance of Hodge numbers briefly. The
Kodaira-Spencer's upper semi-continuity theorem (\cite[Theorem 4]{KS}) tells us that the function
$$t\longmapsto h^{p,q}_{\db_t}(X_t)=\dim_{\mathbb{C}}H^{p,q}_{\db_t}(X_t,\mathbb{C})$$
is always upper semi-continuous for $t\in B$ and thus,
to approach the deformation invariance of $h^{p,q}_{\db_t}(X_t)$, we only need to obtain the lower
semi-continuity. Here our main strategy is a modified iteration
procedure, originally from \cite{LSY} and developed in
\cite{Sun,SY,RZ,lry}, which is to look for an injective extension map from
$H^{p,q}_{\db}(X_0)$ to $H^{p,q}_{\db_t}(X_t)$.
More precisely, for a nice uniquely-chosen representative $\sigma_0$ of the initial Dolbeault
cohomology class 
in $H^{p,q}_{\db}(X_0)$, we try to
construct a convergent power series
$$\label{ps} \sigma_t=\sigma_0+\sum_{j+k=1}^\infty t^k t^{\bar{j}}\sigma_{k\bar{j}}\in
A^{p,q}(X_0),
$$
with $\sigma_t$ varying smoothly on $t$ such
that for each small $t$:
\begin{enumerate}
\item\label{S1} $e^{i_{\varphi}|i_{\overline{\varphi}}}(\sigma_t)\in A^{p,q}(X_t)$
is $\db_t$-closed with respect to the holomorphic structure on $X_t$;
\item\label{S2} The extension map $H^{p,q}_{\db}(X_0) \rightarrow H^{p,q}_{\db_t}(X_t):[\sigma_0]_{\db} \mapsto
[e^{i_{\varphi}|i_{\overline{\varphi}}}(\sigma_t)]_{\db_t}$ is injective.
\end{enumerate}

One main theorem in Section \ref{def} can be stated as:
\begin{theorem}[=Theorem \ref{inv-pq}]\label{inv-pq-intro}
If the injectivity of the mappings $\iota_{BC,\p}^{p+1,q},\iota_{\db,A}^{p,q+1}$ on the central fiber
$X_0$ and the deformation invariance of the $(p,q-1)$-Hodge number $h^{p,q-1}_{\db_t}(X_t)$ holds,
then $h^{p,q}_{\db_t}(X_t)$ are deformation invariant.
\end{theorem}

Obviously, a classical result that a complex manifold satisfying the $\p\db$-lemma admits the deformation invariance of all-type Hodge numbers follows by this theorem and induction.
Three examples \ref{ex10}, \ref{ex20} and \ref{ex23} in the Kuranishi family of the Iwasawa manifold (cf. \cite[Appendix]{A})
are found that the deformation invariance of the $(p,q)$-Hodge number
fails when one of the three conditions in Theorem \ref{inv-pq-intro} does not hold, while the other two do.
It indicates that the three conditions above may not be omitted in order to state a theorem
for the deformation invariance of all the $(p,q)$-Hodge numbers. We also refer
the readers to \cite{Y} (based on \cite{Gt}) for the negative
counterpart of invariance of Hodge numbers.

The speciality of the types may lead to the weakening
of the conditions in Theorem \ref{inv-pq-intro}, such as $(p,0)$ and $(0,q)$:
\begin{theorem}[=Theorems \ref{inv-p0}+\ref{inv-0q}]\label{inv-p0-0q-intro}
\begin{enumerate}[$(1)$]
    \item \label{inv-p0-intro}
If the injectivity of the mappings $\iota^{p+1,0}_{\db,A}$ and $\iota^{p,1}_{\db,A}$
on $X_0$ holds, then $h^{p,0}_{\db_t}(X_t)$ are independent of $t$;
    \item \label{inv-0q-intro}
If the surjectivity of the mapping $\iota^{0,q}_{BC,\db}$ on $X_0$
and the deformation invariance of $h^{0,q-1}_{\db_t}(X_t)$ holds, then
$h^{0,q}_{\db_t}(X_t)$ are independent of $t$.
 \end{enumerate}
\end{theorem}

As mentioned in Remark \ref{01-sgg},
for the case $q=1$ of Theorem \ref{inv-p0-0q-intro}.\eqref{inv-0q-intro}, the surjectivity of the mapping $\iota^{0,1}_{BC,\db}$
is equivalent to the \textbf{sGG} condition proposed by Popovici-Ugarte \cite{P1,PU},
from \cite[Theorem 2.1 (iii)]{PU}.
Hence, the \textbf{sGG} manifolds can be examples of Theorem \ref{inv-0q},
where the Fr\"olicher spectral sequence does not  necessarily degenerate at the $E_1$-level, by
\cite[Proposition 6.3]{PU}. Inspired by the deformation invariance of the $(0,1),(0,2)$ and $(0,3)$-Hodge numbers of the Iwasawa
manifold $\mathbb{I}_3$ shown in \cite[Appendix]{A}, we prove

\begin{corollary}[=Corollary \ref{cplx-prl}]
Let $X = \Gamma \backslash G$ be a complex parallelizable nilmanifold of complex dimension $n$,
where $G$ is a simply connected complex nilpotent Lie group and $\Gamma$ is denoted by
a discrete and co-compact subgroup of $G$.
Then $X$ is an \textbf{sGG} manifold. In addition, the $(0,q)$-Hodge numbers of $X$
are deformation invariant for $1 \leq q \leq n$.
\end{corollary}

Inspired by Console-Fino-Poon \cite[Section 6]{CFP}, we use the proof of Theorem \ref{inv-p0-0q-intro}.\eqref{inv-p0-intro} to give in Example \ref{cfp}
a holomorphic family of nilmanifolds of complex dimension $5$ with the central fiber endowed with an abelian complex structure,
which admits the deformation invariance of the $(p,0)$-Hodge numbers for $1 \leq p \leq 5$, but not the $(1,1)$-Hodge number or $(1,1)$-Bott-Chern number.
This shows the function of Theorem \ref{inv-p0-0q-intro}.\eqref{inv-p0-intro} possibly beyond Kodaira-Spencer's squeeze \cite[Theorem 13]{KS} in this case.

Here is an interesting question:
\begin{question}
What are the sufficient and necessary conditions for a class of compact
complex manifolds to satisfy the deformation invariance for each
prescribed-type Hodge number and all-type Hodge numbers?
\end{question}

In Section \ref{Gau}, we will study various cones to explore the
deformation properties of \textbf{sGG} manifolds. Here are several
notations. The K\"{a}hler cone $\mathcal{K}_X$ and its closure
$\overline{\mathcal{K}}_X$, the numerically effective cone (shortly
nef cone), are important geometric objects on a compact K\"{a}hler
manifold $X$, extensively studied such as in
\cite{D1,DP,DPS,BDPP,WYZ,FX,P1,PU}. J. Fu and J. Xiao \cite{FX} study the relation between the balanced
cone $\mathcal{B}_X$
and the K\"{a}hler cone $\mathcal{K}_X$.
Meanwhile, Popovici \cite{P1}, together with Ugarte \cite{PU},
investigates geometric properties of the Gauduchon cone
$\mathcal{G}_X$ and its related cones. The \emph{Gauduchon cone}
$\mathcal{G}_X$ is defined by
\[ \mathcal{G}_X = \left\{\big[\Omega\big]_{\mathrm{A}} \in
H^{n-1,n-1}_{\mathrm{A}}(X,\mathbb{R})\ \Big|\ \Omega\ \textrm{is}\
\textrm{a}\ \p\db\textrm{-closed}\ \textrm{positive}\
(n-1,n-1)\textrm{-form} \right\}.\]
More
detailed descriptions of real Bott-Chern groups
$H^{p,p}_{\mathrm{BC}}(X,\mathbb{R})$, Aeppli groups
$H^{p,p}_{\mathrm{A}}(X,\mathbb{R})$ and these cones will appear at
the beginning of Section \ref{Gau}.

Inspired by all these, we hope to understand the relation of the
balanced cone $\mathcal{B}_X$ and the Gauduchon cone $\mathcal{G}_X$
via the mapping $\mathscr{J}:
H^{n-1,n-1}_{\mathrm{BC}}(X,\mathbb{R}) \>
 H^{n-1,n-1}_{\mathrm{A}}(X,\mathbb{R})$ induced by the identity
map.  Another direct motivation of this part is the following
conjecture:
\begin{conjecture}[{\cite[Conjecture 6.1]{P4}}] \label{ddbar-b}
Each compact complex manifold $X$ satisfying the $\p\db$-lemma
admits a balanced metric.
\end{conjecture}
One possible approach is to prove
$\mathscr{J}^{-1}(\mathcal{G}_X)=\mathcal{B}_X$, since the Gauduchon
cone of a compact complex manifold is never empty and $\mathscr{J}$
is an isomorphism from the $\p\db$-lemma. See the important argument
in \cite[Section 6]{P4} or \cite[Section 2]{crs} relating a slightly
different conjecture with the quantitative part of Transcendental
Morse Inequalities Conjecture for differences of two nef classes as
in \cite[Conjecture 10.1.(ii)]{BDPP} and (more precisely) also their
main Conjecture \ref{conj-mov}.

A weaker question comes up:
\begin{question}\label{que}
Does the mapping $\mathscr{J}$ map the balanced cone $\mathcal{B}_X$
bijectively onto the Gauduchon cone $\mathcal{G}_X$ on the
K\"{a}hler manifold $X$?
\end{question}
It is clear that $\mathscr{J}$ maps $\mathcal{B}_X$ injectively into
$\mathcal{G}_X$ from the $\p\db$-lemma of K\"{a}hler manifolds. The
affirmation of this question is equivalent to the equality
\begin{equation}\label{int-eq}
\mathcal{E}_X = \mathscr{L}^{-1}(\mathcal{E}_{\p\db})
\end{equation}
by Proposition \ref{biject}. The \emph{pseudo-effective cone}
$\mathcal{E}_X$ is generated by Bott-Chern classes in
$H^{1,1}_{\mathrm{BC}}(X,\mathbb{R})$ represented by $d$-closed
positive $(1,1)$-currents and the convex cone $\mathcal{E}_{\p\db}
\subseteq H^{1,1}_{\mathrm{A}}(X,\mathbb{R})$, is generated by Aeppli
classes represented by $\p\db$-closed positive $(1,1)$-currents,
with the natural isomorphism $\mathscr{L}:
H^{1,1}_{\mathrm{BC}}(X,\mathbb{R}) \>
H^{1,1}_{\mathrm{A}}(X,\mathbb{R})$ induced by the identity map. The
pull-back cone $\mathscr{L}^{-1}(\mathcal{E}_{\p\db})$ denotes the
inverse image of the cone $\mathcal{E}_{\p\db}$ under the
isomorphism $\mathscr{L}$. The closed convex cone $\mathcal{M}_X
\subseteq H^{n-1,n-1}_{\mathrm{BC}}(X,\mathbb{R})$ is called the
movable cone, originating from \cite{BDPP}, and
$\big(\mathcal{M}_X\big)^{\mathrm{v}_c}$ denotes its dual cone
(cf. Definitions \ref{dual-cone} and \ref{movable}).
\begin{lemma}[See Lemma \ref{two-inclusions} and its remarks]\label{ka-in}
Let $X$ be a compact K\"{a}hler manifold. There exist the following inclusions:
\[ \mathcal{E}_X \subseteq \mathscr{L}^{-1}(\mathcal{E}_{\p\db})
\subseteq \big(\mathcal{M}_X\big)^{\mathrm{v}_c}. \] \end{lemma} By
the inclusions in this lemma, the equality \eqref{int-eq} is
actually a part of:
\begin{conjecture}[{\cite[Conjecture 2.3]{BDPP}}]\label{conj-mov}
Let $X$ be a compact K\"{a}hler manifold. Then the equality holds
\[\mathcal{E}_X = \big(\mathcal{M}_X\big)^{\mathrm{v}_c}. \]
\end{conjecture}
An analogous conjecture of the balanced case is proposed as
\cite[Conjecture 5.4]{FX}. The following theorem provides
some evidence for the assertion of Question \ref{que}.

\begin{theorem}[= Theorem \ref{nef-b-g}]\label{0nef-b-g}
Let $X$ be a compact K\"{a}hler manifold and
$\big[\alpha\big]_{\mathrm{BC}}$ a nef class. Then
$\big[\alpha^{n-1}\big]_{\mathrm{A}} \in \mathcal{G}_X$ implies that
$\big[ \alpha^{n-1}\big]_{\mathrm{BC}} \in \mathcal{B}_X$. Hence
$\overline{\mathbb{I}}(\overline{\mathcal{K}}_X) \bigcap
\mathcal{B}_X$ and $\overline{\mathbb{K}}(\overline{\mathcal{K}}_X)
\bigcap \mathcal{G}_X$ can be identified by the mapping
$\mathbb{J}$.
\end{theorem}
The mappings $\overline{\mathbb{I}}$ and $\overline{\mathbb{K}}$ are
contained in the pair of diagrams
$(\mathrm{D},\overline{\mathrm{D}})$ as in the beginning of Section
\ref{relation}.  The proof relies on several important results on
solving complex Monge-Amp\`{e}re equations on the compact K\"{a}hler
manifold $X$. One is the Yau's celebrated results of solutions of
the complex Monge-Amp\`{e}re equations for K\"{a}hler classes
\cite{Yst}. The other one is the Boucksom-Eyssidieux-Guedj-Zeriahi's
work on the equations for the
nef and big classes \cite{BEGZ}.

Popovici and Ugarte in \cite[Theorem 5.7]{PU} prove that the
following inclusion holds $$\label{inclusion} \mathcal{G}_{X_0}
\subseteq \lim\limits_{t\>0} \mathcal{G}_{X_t} $$ for the family $\pi:
\mathcal{X} \rightarrow \Delta_{\epsilon}$ over a small complex disk with the central fiber an
\textbf{sGG} manifold, where $\lim\limits_{t\>0} \mathcal{G}_{X_t}$
is defined by
\[ \lim\limits_{t\>0} \mathcal{G}_{X_t} = \Big\{ \big[ \Omega
\big]_{\mathrm{A}} \in H^{n-1,n-1}_{\mathrm{A}}(X_0,\mathbb{R})
\Big| \mathrm{P}_t \circ \mathrm{Q}_{\,0} \big(\big[ \Omega
\big]_{\mathrm{A}}\big) \in \mathcal{G}_{X_t}\ \textrm{for}\ t\
\textrm{sufficiently}\ \textrm{small}\Big\}.\] The canonical
mappings $\mathrm{P}_t: H^{2n-2}_{\mathrm{DR}}(X_t,\mathbb{R}) \>
H^{n-1,n-1}_{\mathrm{A}}(X_t,\mathbb{R})$ are surjective for all $t$
 and the mapping $\mathrm{Q}_{\,0}: H^{n-1,n-1}_{\mathrm{A}}(X_t,\mathbb{R}) \>
H^{2n-2}_{\mathrm{DR}}(X_t,\mathbb{R})$, depending on a fixed
Hermitian metric $\omega_0$ on $X_0$, is injective, which satisfies
 $\mathrm{P}_0 \circ \mathrm{Q}_{\,0} =
\mathrm{id}_{H^{n-1,n-1}_{\mathrm{A}} (X,\mathbb{R})}$. Here we give
another inclusion from the other side as follows, where Demailly's
regularization of closed positive currents (Theorem
\ref{regularization}) plays an important role in the proof.
\begin{theorem}[= Theorem \ref{inclusion_1}]
Let $\pi: \mathcal{X} \rightarrow \Delta_{\epsilon}$ be a
holomorphic family  with the K\"ahlerian central fiber $X_0$. Then we
have \[  \lim_{t \> \tau} \mathcal{G}_{X_t} \subseteq
\mathcal{N}_{X_{\tau}}\quad \textrm{for each}\ \ \tau \in
\Delta_{\epsilon},\] where $\mathcal{N}_{X_{\tau}}$ is the convex
cone generated by Aeppli classes of $\p_{\tau}\db_{\tau}$-closed
positive $(n-1,n-1)$-currents on $X_{\tau}$. Moreover, the following
inclusion holds, for $\tau \in \Delta_{\epsilon} \setminus \bigcup
S_{\nu}$,
\[  \lim_{t \> \tau} \mathcal{G}_{X_t} \subseteq \overline{\mathcal{G}}_{X_{\tau}}. \]
\end{theorem}
Here $\bigcup S_{\nu}$ is a countable union of analytic subvarieties
$S_{\nu}$ of $\Delta_{\epsilon}$. And Theorem \ref{inclusion_2}
deals with the case of the fiber, satisfying the equality
$\overline{\mathcal{K}}_X = \mathcal{E}_X$, in a K\"ahler family.

In \cite{RwZ}, X. Wan and the authors will apply the
 extension methods developed here to a power series proof of
Kodaira-Spencer's local stability theorem of K\"ahler metrics, which
is motivated by:
\begin{prob}[{Remark $1$ on \cite[p.  180]{MK}}]
A good problem would be to find an elementary proof (for example,
using power series methods). Our proof uses nontrivial results from
partial differential equations.
\end{prob}

\textbf{Acknowledgement}: We would like to express our gratitude
to Professors Daniele Angella, Kwokwai Chan, Huitao Feng, Jixiang Fu, Lei Fu, Conan
Leung, Kefeng Liu, Dan Popovici, Fangyang Zheng, and Dr. Jie Tu,
Yat-hin Suen, Xueyuan Wan, Jian Xiao, Xiaokui Yang, Wanke Yin,
Shengmao Zhu for their useful advice or interest on this work. This
work started when the first author was invited by Professor J.-A.
Chen to Taiwan University in May-July of 2013 with the support of
the National Center for Theoretical Sciences, and was completed
during his visit in June of 2015 to the Mathematics Department of
UCLA. He takes this opportunity to thank them for their hospitality.
Last but not least, the anonymous referee's careful reading and valuable comments
improve the statement significantly.

\section{An extension formula for complex differential
forms}\label{ext-formula}
Inspired by the classical Kodaira-Spencer-Kuranishi deformation
theory of complex structures and the recent work \cite{lry}, we will
present an extension formula for complex differential forms. For a holomorphic family of compact complex manifolds, we adopt the definition \cite[Definition 2.8]{K2}; while for the differentiable one, we follow:
\begin{definition}[{\cite[Definition 4.1]{K2}}]\label{smf} Let $\mc{X}$ be a differentiable manifold, $B$ a domain of $\mathbb{R}^k$ and $\pi$ a smooth map of $\mc{X}$ onto $B$.
By a \emph{differentiable family of $n$-dimensional compact complex manifolds} we mean the triple $\pi:\mc{X}\to B$ satisfying the following conditions:
\begin{enumerate}[$(i)$]
    \item \label{}
The rank of the Jacobian matrix of $\pi$ is equal to $k$ at every point of $\mc{X}$;
    \item \label{}
For each point $t\in B$, $\pi^{-1}(t)$ is a compact connected subset of $\mc{X}$;
    \item \label{}
$\pi^{-1}(t)$ is the underlying differentiable manifold of the $n$-dimensional compact complex manifold $X_t$ associated to each $t\in B$;
     \item \label{}
There is a locally finite open covering $\{\mathcal{U}_j\ |\ j=1,2,\cdots\}$ of $\mc{X}$ and complex-valued smooth functions $\zeta_j^1(p),\cdots,\zeta_j^n(p)$, defined on $\mathcal{U}_j$
such that for each $t$, $$\{p\rightarrow (\zeta_j^1(p),\cdots,\zeta_j^n(p))\ |\ \mathcal{U}_j\cap \pi^{-1}(t)\neq \emptyset\}$$ form a system of local holomorphic coordinates of $X_t$.
\end{enumerate}
\end{definition}

\subsection{Extension maps for deformations}
Let us introduce several new notations. For $\phi\in
A^{0,s}(X,T^{1,0}_X)$ on a complex manifold $X$, the contraction
operator can be extended to
$$ i_\phi:\, A^{p,q}(X)\>A^{p-1,q+s}(X)\label{db3}.$$ For example,
if $\phi=\eta\ts Y$ with $\eta\in A^{0,q}(X)$ and
$Y\in\Gamma(X,T^{1,0}_X)$, then for any $\om \in A^{p,q}(X),$ $$
(i_\phi)(\omega)=\eta\wedge (i_Y\omega).$$
Let $\varphi\in A^{0,p}(X,T^{1,0}_X)$ and $\psi\in A^{0,q}(X,T^{1,0}_X)$, locally
written as
\[
\varphi=\frac{1}{p!}\sum\varphi^{i}_{\bar{j}_{1},\cdots,\bar{j}_{p}}
d\bar{z}^{j_{1}}\wedge\cdots\wedge d\bar{z}^{j_{p}}\ts \p_i\
\textmd{and}\ \psi=\frac{1}{q!}\sum\psi^{i}_{\bar{k}_{1},\cdots,\bar{k}_{q}}d\bar{z}^{k_{1}%
}\wedge\cdots\wedge d\bar{z}^{k_{q}}\ts \p_i.
\]
Then we have
$$
\label{Belt}[\varphi,\psi]=\sum_{i,j=1}^{n}(\varphi^{i}\wedge\partial_{i}%
\psi^{j}-(-1)^{pq}\psi^{i}\wedge\partial_{i}\varphi^{j})\otimes\partial_{j},
$$
where $$\partial_{i}\varphi^{j}=\frac{1}{p!}\sum\partial_{i}\varphi^{j}%
_{\bar{j}_{1},\cdots,\bar{j}_{p}}d\bar{z}^{j_{1}}\wedge\cdots\wedge
d\bar {z}^{j_{p}}$$ and similarly for $\partial_{i}\psi^{j}$. In
particular, if $\varphi,\psi\in A^{0,1}(X,T^{1,0}_X)$,
$$[\varphi,\psi]=\sum
_{i,j=1}^{n}(\varphi^{i}\wedge\partial_{i}\psi^{j}+\psi^{i}\wedge\partial
_{i}\varphi^{j})\otimes\partial_{j}.$$
For any $\phi\in A^{0,q}(X,T^{1,0}_X)$, we can define $\sL_\phi$ by
$$ \sL_\phi=(-1)^q d\circ i_\phi+i_\phi\circ d. $$ According to the
types, we can decompose
$$ \sL_\phi=\sL_\phi^{1,0}+\sL_{\phi}^{0,1},$$ where $$
\sL^{1,0}_\phi=(-1)^q\p\circ i_\phi+i_\phi\circ \p$$ and $$
\sL^{0,1}_\phi=(-1)^q\bp\circ i_\phi+i_\phi\circ \bp. $$ Then one
has the following commutator formula, which originated from
\cite{T,To89} and whose various versions appeared in
\cite{F,BK,Li,LSY,C} and also \cite{LR,lry} for vector bundle valued
forms.
\begin{lemma}\label{aaaa} For $\phi, \phi'\in
A^{0,1}(X,T^{1,0}_X)$ on a complex manifold $X$ and $\sigma\in
A^{*,*}(X)$,
$$\label{f1}
[\phi,\phi']\lrcorner\sigma=-\p(\phi'\lrcorner(\phi
\lrcorner\sigma))-\phi'\lrcorner(\phi \lrcorner\p\sigma)
+\phi\lrcorner\p(\phi'\lrcorner\sigma)+\phi'
\lrcorner\p(\phi\lrcorner\sigma);
$$
or equivalently,
\begin{equation}\label{f11}
i_{[\phi,\phi']}=\sL_\phi^{1,0}\circ i_{\phi'}-i_{\phi'}\circ \sL_\phi^{1,0}.
\end{equation}
\end{lemma}

Let $\phi\in A^{0,1}(X,T^{1,0}_X)$ and $i_\phi$ be the contraction
operator. Define an operator
$$e^{i_\phi}=\sum_{k=0}^\infty \frac{1}{k!} i_\phi^{k},$$
where $i_\phi^k=\underbrace{i_\phi\circ \cdots\circ i_\phi}_{k\
\text{copies}}$. Since the dimension of $X$ is finite, the summation
in the above formulation is also  finite.

\begin{proposition}[{\cite[Theorem 3.4]{lry}}]\label{main1}
Let $\phi\in A^{0,1}(X,T^{1,0}_X)$. Then on the space $A^{*,*}(X)$,
\beq\label{ext-old} e^{- i_{\phi}}\circ d\circ e^{
i_\phi}=d-\sL_{\phi}- \
i_{\frac{1}{2}[\phi,\phi]}=d-\sL_\phi^{1,0}+i_{\bp\phi-\frac{1}{2}[\phi,\phi]}.\eeq
Or equivalently
\begin{equation}\label{db-ext}
e^{- i_{\phi}}\circ \bp\circ e^{ i_\phi}=\bp-\sL_{\phi}^{0,1}
\end{equation}
and $$e^{- i_{\phi}}\circ \p \circ e^{ i_\phi}=\p-\sL^{1,0}_{\phi}-
\ i_{\frac{1}{2}[\phi,\phi]}.\label{f4}$$
\end{proposition}
\begin{proof}
Note that (\ref{db-ext}) proved in \cite[Lemma 8.2]{C} will not be
used in this new proof, but only the commutator formula (\ref{f11})
and
\begin{equation}\label{2phi-phi}
i_{[\phi,\phi]}\circ  i_\phi= i_\phi\circ  i_{[\phi,\phi]}
\end{equation}
by a formula on \cite[p.  $361$]{C}.

Let us first define a bracket
$$[d,i_\phi^k]=d\circ i_\phi^k-i_\phi^k\circ d.$$
Obviously, $[d,i_\phi]=-\sL_{\phi}$ and \eqref{ext-old} is
equivalent to
\begin{equation}\label{eq-ext}
[d, e^{ i_\phi}]=e^{ i_\phi}\circ [d,
 i_\phi]-e^{ i_\phi}\circ i_{\frac{1}{2} [\phi,\phi]}.
\end{equation}
We check the Leibniz rule for the bracket: for $k\geq 2$,
$$[d,i_\phi^k]=\sum_{j=1}^k i_\phi^{j-1}\circ [d, i_\phi]\circ i_\phi^{k-j}.$$
As for $k=2$,
$$[d,i_\phi^2]=d\circ i_\phi^2-i_\phi\circ d\circ i_\phi
+i_\phi\circ d\circ i_\phi-i_\phi^2\circ d=[d,i_\phi]\circ
i_\phi+i_\phi\circ [d,i_\phi].$$ Then similarly, one is able to
prove the cases for $k\geq 3$ by induction.

Now we can prove \eqref{eq-ext}. Actually, the Leibniz rule and the
formulae \eqref{f11} \eqref{2phi-phi} tell us: for $k\geq 2$,
$$
[d, \ i_\phi^k]=k i_{\phi}^{k-1}\circ [d, \
i_\phi]-\frac{k(k-1)}{2}i_\phi^{k-2}\circ i_{[\phi,\phi]},\label{p1}
$$ which implies \eqref{eq-ext}.
\end{proof}

From now on, one considers the smooth family
$$\pi: \mathcal{X} \rightarrow B$$
 of $n$-dimensional compact complex manifolds over a small real domain with the central fiber
$$X_0:= \pi^{-1}(0)$$ and the general fibers denoted by $$X_t:=
\pi^{-1}(t).$$ Assume that $k=1$ for simplicity.
We will use the standard notions in deformation theory as in
the beginning of \cite[Chapter 4]{MK}.
Fix an open coordinate covering $\{\mathcal{U}_j\}$
of $\mathcal{X}$ so that
$$\mathcal{U}_j:=\big\{
(\zeta_j,t)=(\zeta_j^1,\cdots,\zeta_j^n,t)\ |\
|\zeta_j|<1,|t|<\epsilon\big\},$$
$$\pi(\zeta_j,t)=t$$
and $$\zeta^\alpha_j=f^\alpha_{jk}(\zeta_k,t)\ on\
\mathcal{U}_j\cap\mathcal{U}_k,$$ where $f_{jk}$ is holomorphic in
$\zeta_k$ and smooth in $t$. By Ehresmann's theorem \cite{E}, $\mathcal{X}$ is
diffeomorphic to $X\times B$, where $X$ is the
underlying differentiable manifold of $X_0$. Then $$\mathcal{U}_j =
U_j\times B,$$ where $U_j= \{\zeta_j\ |\
|\zeta_j|<1\}.$ Thus, we can consider $X_t$ as a compact manifold
obtained by glueing $U_j$ with $t\in B$ by
identifying $\zeta_k\in U_k$ with $\zeta_j=f_{jk}(\zeta_k,t)\in
U_j$. We refer the readers to \cite[\S 4.1.(b)]{K2} for more details
on this description. If $x$ is a point of the underlying
differentiable manifold $X$ of $X_0$ and $t\in \Delta_{\epsilon}$,
we notice that
$$\zeta^\alpha_j=\zeta^\alpha_j(x,t)$$ is a differentiable function
of $(x,t)$. Use the holomorphic coordinates $z$ of $X_0=X$ as
differentiable coordinates so that
$$\zeta^\alpha_j(x,t)=\zeta^\alpha_j(z,t),$$
where $\zeta^\alpha_j(z,t)$ is a differentiable function of $(z,t)$.
At $t=0$, $\zeta^\alpha_j(z,t)$ is holomorphic in $z$ and otherwise
it is only differentiable.

Then a Beltrami differential $\varphi(t)$ can be calculated out
explicitly on the above local coordinate charts. As we focus on one
coordinate chart, the subscript is suppressed. From \cite[p.
150]{MK}, \beq\label{krnsh data} \varphi(t) = \left(
\frac{\partial\ }{\partial z} \right)^{\!\!\!T} \left( \frac{\partial
\zeta}{\partial z} \right)^{\!\!\!-1} \bp \zeta,
\eeq where $\frac{\partial\ }{\partial z} =
\begin{pmatrix} \frac{\p\ }{\p z^1} \\ \vdots \\ \frac{\p\ }{\p z^n} \end{pmatrix}$,
$\bp \zeta = \begin{pmatrix} \bp \zeta^1 \\ \vdots \\
\bp \zeta^n \end{pmatrix}$,
$\frac{\partial \zeta}{\partial z}$ stands for the matrix $(
\frac{\p \zeta^\alpha}{\p z^j} )_{\begin{subarray}{l} 1 \leq \alpha \leq n \\
1 \leq j \leq n \\ \end{subarray}}$ and $\alpha,j$ are the row and
column indices. Here $\left( \frac{\partial\ }{\partial
z} \right)^{\!T}$ is the transpose of $\frac{\p\ }{\p z}$ and $\db$ denotes
the Cauchy-Riemann operator with respect to the holomorphic structure on $X_0$.

Since $\varphi(t)$ is locally expressed as $\varphi^i_{\bar{j}}
\dz^j \ts \frac{\p\ }{\p z^i} \in A^{0,1}(T^{1,0}_{X_0})$, it can be
considered as a matrix $(\varphi^i_{\bar{j}})_{\begin{subarray}{l} 1 \leq i \leq n \\
1 \leq j \leq n \\ \end{subarray}}$. By \eqref{krnsh data}, this
matrix can be explicitly written as: \beq\label{krnsh matrix}
\begin{aligned}
\varphi = (\varphi^i_{\bar{j}})_{\begin{subarray}{l} 1 \leq i \leq n \\
1 \leq j \leq n \\ \end{subarray}}
        =\varphi(t)\big( \frac{\p\ }{\p \bar{z}^j}, dz^i \big)
         = \lk \lk \frac{\p \zeta}{\p z} \rk^{\!\!\!-1} \lk \frac{\p \zeta}{\p \bar{z}} \rk
         \rk^i_{\bar{j}}.
\end{aligned} \eeq
A fundamental fact is that the Beltrami differential $\varphi(t)$
defined as above satisfies the integrability:
\begin{equation}\label{int-conditin}
\bp\varphi(t)=\frac{1}{2}[\varphi(t),\varphi(t)].
\end{equation}

One needs the following crucial calculation:
\blemma\label{inverse}
$$\begin{pmatrix} \frac{\p z}{\p \zeta} & \frac{\p z}{\p \bar{\zeta}}  \\[4pt]
\frac{\p \bar{z}}{\p \zeta} & \frac{\p \bar{z}}{\p \bar{\zeta}}  \\
\end{pmatrix} = \begin{pmatrix}
\lk \1 - \varphi \overline{\varphi} \rk^{\!-1} \lk \frac{\p \zeta}{\p z}
\rk^{\!\!-1} & - \varphi \lk \1- \overline{\varphi} \varphi \rk^{\!-1} \lk \overline{ \frac{\p \zeta}{\p z} } \rk ^{\!\!-1}  \\[4pt]
- \lk \1- \overline{\varphi} \varphi \rk^{\!-1} \overline{ \varphi }
\lk \frac{\p \zeta}{\p z} \rk^{\!\!-1} & \lk \overline{ \1 - \varphi
\overline{\varphi} } \rk ^{\!-1} \lk \overline{ \frac{\p \zeta}{\p z} } \rk^{\!\!-1}   \\
\end{pmatrix}.$$
Here $\varphi \overline{\varphi}$, $\overline{\varphi} \varphi$
stand for the two matrices $(
\varphi^i_{\bar{k}} \overline{\varphi^k_{\bar{j}}})_{\begin{subarray}{l} 1 \leq i \leq n \\
1 \leq j \leq n \\ \end{subarray}}$,
$(\overline{\varphi^i_{\bar{k}}} \varphi^k_{\bar{j}})_{\begin{subarray}{l} 1 \leq i \leq n \\
1 \leq j \leq n \\ \end{subarray}}$, respectively. \elemma In many
places, $\varphi \overline{\varphi}$ and $\overline{\varphi}
\varphi$ can also be seen as $\varphi^i_{\bar{k}}
\overline{\varphi^k_{\bar{j}}} dz^j \ts \frac{\p\ }{\p z^i} \in
A^{1,0}(T^{1,0}_{X_0})$ and $\overline{\varphi^i_{\bar{k}}}
\varphi^k_{\bar{j}} \dz^j \ts \frac{\p\ }{\p \bar{z}^i} \in
A^{0,1}(T^{0,1}_{X_0})$. Actually, $\varphi \overline{\varphi} =
\overline{\varphi} \lc \varphi, \overline{\varphi} \varphi = \varphi
\lc \overline{\varphi}$ and $\1$ is the identity matrix. \bproof
It is easy to see that $\begin{pmatrix} \frac{\p z}{\p \zeta} & \frac{\p z}{\p \bar{\zeta}}  \\[4pt]
\frac{\p \bar{z}}{\p \zeta} & \frac{\p \bar{z}}{\p \bar{\zeta}}  \\
\end{pmatrix}$ is the inverse matrix of $\begin{pmatrix}
\frac{\p \zeta}{\p z} & \frac{\p \zeta}{\p \bar{z}}  \\[4pt]
\frac{\p \bar{\zeta}}{\p z} & \frac{\p \bar{\zeta}}{\p \bar{z}}  \\
\end{pmatrix}$. Then it follows, \beq\label{trnsf}
\begin{pmatrix} \1 & 0 \\
- \lk \frac{\p \bar{\zeta}}{\p z} \rk \lk \frac{\p \zeta}{\p z} \rk^{\!\!-1} & \1 \\
\end{pmatrix}
\begin{pmatrix}
\frac{\p \zeta}{\p z} & \frac{\p \zeta}{\p \bar{z}}  \\[4pt]
\frac{\p \bar{\zeta}}{\p z} & \frac{\p \bar{\zeta}}{\p \bar{z}}  \\
\end{pmatrix} = \begin{pmatrix}
\frac{\p \zeta}{\p z} & \frac{\p \zeta}{\p \bar{z}}  \\
0 & \frac{\p \bar{\zeta}}{\p \bar{z}}
-\lk \frac{\p \bar{\zeta}}{\p z} \rk \lk \frac{\p \zeta}{\p z} \rk^{\!\!-1} \lk \frac{\p \zeta}{\p \bar{z}} \rk \\
\end{pmatrix}.
\eeq Take the inverse matrices of both sides of \eqref{trnsf},
yielding \beq\label{inverse1}
\begin{pmatrix}
\frac{\p \zeta}{\p z} & \frac{\p \zeta}{\p \bar{z}}  \\[4pt]
\frac{\p \bar{\zeta}}{\p z} & \frac{\p \bar{\zeta}}{\p \bar{z}}  \\
\end{pmatrix}^{\!\!\!-1}
= \begin{pmatrix}
\frac{\p \zeta}{\p z} & \frac{\p \zeta}{\p \bar{z}}  \\
0 & \frac{\p \bar{\zeta}}{\p \bar{z}}
-\lk \frac{\p \bar{\zeta}}{\p z} \rk \lk \frac{\p \zeta}{\p z} \rk^{\!\!-1} \lk \frac{\p \zeta}{\p \bar{z}} \rk \\
\end{pmatrix}^{\!\!\!-1}
\begin{pmatrix} \1 & 0 \\
- \lk \frac{\p \bar{\zeta}}{\p z} \rk \lk \frac{\p \zeta}{\p z} \rk^{\!\!-1} & \1 \\
\end{pmatrix}.
\eeq From Linear Algebra, we have the basic equality below
\beq\label{linear}
\begin{pmatrix} A & C \\ 0 & B \\ \end{pmatrix}^{\!\!\!-1}
=\begin{pmatrix} A^{-1} & -A^{-1}CB^{-1} \\ 0 & B^{-1} \\
\end{pmatrix},
\eeq where $A,B$ are invertible matrices. Combine with \eqref{krnsh
matrix} and \eqref{linear} and go back to \eqref{inverse1}:
\[\begin{aligned}
\begin{pmatrix}
\frac{\p \zeta}{\p z} & \frac{\p \zeta}{\p \bar{z}}  \\[4pt]
\frac{\p \bar{\zeta}}{\p z} & \frac{\p \bar{\zeta}}{\p \bar{z}}  \\
\end{pmatrix}^{\!\!\!-1}
&= \begin{pmatrix}
\frac{\p \zeta}{\p z} & \frac{\p \zeta}{\p \bar{z}}  \\[4pt]
0 & \lk \frac{\p \bar{\zeta}}{\p \bar{z}} \rk \lk \1 - \lk
\frac{\p\bar{\zeta}}{\p \bar{z}} \rk^{\!\!-1} \lk \frac{\p
\bar{\zeta}}{\p z} \rk
\lk \frac{\p \zeta}{\p z} \rk^{\!\!-1} \lk \frac{\p \zeta}{\p \bar{z}} \rk \rk\\
\end{pmatrix}^{\!\!\!-1}
\begin{pmatrix} \1 & 0 \\
- \lk \frac{\p \bar{\zeta}}{\p z} \rk \lk \frac{\p \zeta}{\p z} \rk^{\!\!-1} & \1 \\
\end{pmatrix} \\
&= \begin{pmatrix}
\frac{\p \zeta}{\p z} & \frac{\p \zeta}{\p \bar{z}}  \\
0 & \lk \frac{\p \bar{\zeta}}{\p \bar{z}} \rk \lk \1 - \overline{\varphi} \varphi \rk\\
\end{pmatrix}^{\!\!\!-1}
\begin{pmatrix} \1 & 0 \\
- \lk \frac{\p \bar{\zeta}}{\p z} \rk \lk \frac{\p \zeta}{\p z} \rk^{\!\!-1} & \1 \\
\end{pmatrix}  \\
&= \begin{pmatrix} \lk \frac{\p \zeta}{\p z} \rk^{\!\!-1}
& - \varphi \lk \1 - \overline{\varphi} \varphi \rk^{\!-1} \lk \frac{\p \bar{\zeta}}{\p \bar{z}} \rk^{\!\!-1} \\[5pt]
0 & \lk \1 - \overline{\varphi} \varphi \rk^{\!-1} \lk \frac{\p \bar{\zeta}}{\p \bar{z}} \rk^{\!\!\!-1} \\
\end{pmatrix}
\begin{pmatrix} \1 & 0 \\
- \lk \frac{\p \bar{\zeta}}{\p z} \rk \lk \frac{\p \zeta}{\p z} \rk^{\!\!-1} & \1 \\
\end{pmatrix}  \\
&= \begin{pmatrix} \lk \1 + \varphi \lk \1 - \overline{\varphi}
\varphi \rk^{\!-1} \overline{\varphi} \rk \lk \frac{\p \zeta}{\p
z} \rk^{\!\!-1}
& - \varphi \lk \1 - \overline{\varphi} \varphi \rk^{\!-1} \lk \frac{\p \bar{\zeta}}{\p \bar{z}} \rk^{\!\!-1} \\[4pt]
- \lk \1- \overline{\varphi} \varphi \rk^{\!-1} \overline{ \varphi }
\lk \frac{\p \zeta}{\p z} \rk^{\!-1}
& \lk \1 - \overline{\varphi} \varphi \rk^{\!-1} \lk \frac{\p \bar{\zeta}}{\p \bar{z}} \rk^{\!\!-1} \\
\end{pmatrix}\\
&= \begin{pmatrix} \lk \1 - \varphi \overline{\varphi} \rk^{\!-1} \lk
\frac{\p \zeta}{\p z}
\rk^{\!\!-1} & - \varphi \lk \1- \overline{\varphi} \varphi \rk^{\!-1} \lk \overline{ \frac{\p \zeta}{\p z} } \rk ^{\!\!-1}  \\[4pt]
- \lk \1- \overline{\varphi} \varphi \rk^{\!-1} \overline{ \varphi }
\lk \frac{\p \zeta}{\p z} \rk^{\!-1} & \lk \overline{ \1 - \varphi
\overline{\varphi} } \rk ^{\!-1} \lk \overline{ \frac{\p \zeta}{\p z} } \rk^{\!\!-1}   \\
\end{pmatrix}.
\end{aligned} \]
\eproof

We need a few more local formulae:
 \blemma\label{dwdz} $\begin{cases} d\zeta^{\alpha} &= \frac{\p
\zeta^{\alpha}}{\p z^i} \lk e^{i_{\varphi}}(dz^i) \rk ,\\
\frac{\p\ }{\p \zeta^{\alpha}} &= \lk \lk \1 - \varphi
\overline{\varphi} \rk^{\!-1} \lk \frac{\p \zeta}{\p z} \rk^{\!\!-1}
\rk^j_{\alpha} \frac{\p\ }{\p z^j} - \lk \lk \1- \overline{\varphi}
\varphi \rk^{\!-1} \overline{ \varphi } \lk \frac{\p \zeta}{\p z}
\rk^{\!-1} \rk^{\bar{j}}_{\alpha} \frac{\p\ }{\pz^j}.
\end{cases}$ \elemma
\bproof For the the first equality, \[\begin{aligned}
d\zeta^{\alpha} &= \frac{\p \zeta^{\alpha}}{\p z^i}dz^i + \frac{\p
\zeta^{\alpha}}{\pz^j}\dz^j\\
&= \frac{\p \zeta^{\alpha}}{\p z^i} \lk dz^i + \lk \Big(\frac{\p
\zeta}{\p
z}\Big)^{\!-1} \rk^{i}_{\beta} \frac{\p \zeta^{\beta}}{\pz^{j}} \dz^j \rk \\
&= \frac{\p \zeta^{\alpha}}{\p z^i} \lk dz^i + \varphi^i_{\bar{j}}
\dz^j \rk
= \frac{\p \zeta^{\alpha}}{\p z^i} \lk e^{i_{\varphi}}(dz^i) \rk. \\
\end{aligned}\]
Then the second one follows from Lemma \ref{inverse}:
\[\begin{aligned}
\frac{\p\ }{\p \zeta^{\alpha}} &= \frac{\p z^i}{\p \zeta^{\alpha}}
\frac{\p\ }{\p z^i} + \frac{\pz^j}{\p \zeta^{\alpha}} \frac{\p\ }{\pz^j} \\
&= \lk \lk \1 - \varphi \overline{\varphi} \rk^{\!-1} \lk \frac{\p
\zeta}{\p z} \rk^{\!\!\!-1} \rk^j_{\alpha} \frac{\p\ }{\p z^j} - \lk
\lk \1- \overline{\varphi} \varphi \rk^{\!-1} \overline{ \varphi }
\lk \frac{\p \zeta}{\p z} \rk^{\!\!\!-1} \rk^{\bar{j}}_{\alpha} \frac{\p\ }{\pz^j}. \\
\end{aligned}\] \eproof

\bcorollary\label{dwdzdz} $\frac{\p \zeta^{\alpha}}{\p z^i}
\frac{\p\ }{\p \zeta^{\alpha}} = \lk \lk \1 - \varphi
\overline{\varphi} \rk^{\!-1} \rk^j_i \frac{\p\ }{\p z^j} - \lk \lk
\1- \overline{\varphi} \varphi \rk^{\!-1} \overline{ \varphi }
\rk^{\bar{j}}_{i} \frac{\p\ }{\pz^j}$. \ecorollary

\bproof It is a direct corollary of the second equality in Lemma
\ref{dwdz}.\eproof

By the above preparation, we can reprove the following important
proposition in deformation theory of complex structures, which can be dated back to \cite{Fr} (see \cite[Section
1]{NN} and also \cite[pp. 151-152]{MK}).

\begin{proposition}[]\label{hol}
The holomorphic structure on $X_t$ is determined by $\varphi(t)$. More
specifically, a differentiable  function $f$ defined on any open subset
of $X_0$ is holomorphic with respect to the holomorphic structure of
$X_t$ if and only if
\begin{equation}\label{f-hol}
\lk \overline{\partial}-\sum_i\varphi^i(t)\partial_i \rk f(z)=0,
\end{equation}
where
$\varphi^i(t)=\sum_{j}\varphi(t)^i_{\overline{j}}d\overline{z}^j$,
or equivalently,
$$\lk \overline{\partial}-\varphi(t)\lc\partial\rk f(z)=0.$$
\end{proposition}
\begin{proof}
By use of Lemma \ref{dwdz} and Corollary \ref{dwdzdz}, we get
\begin{align*}\label{dbtf}
df
 &=\frac{\p f}{\p \zeta^{\alpha}} d\zeta^\alpha + \frac{\p f}{\partial\bar{\zeta}^{\beta}} d\overline{\zeta}^{\beta}\\
 &= \frac{\p f}{\p \zeta^{\alpha}} \frac{\partial \zeta^\alpha}{\p z^i}\lk e^{i_{\varphi}}(dz^i) \rk
  + \frac{\p f}{\partial\bar{\zeta}^{\beta}} \overline{\frac{\partial \zeta^\beta}{\p z^i} \lk e^{i_{\varphi}}(dz^i) \rk}\\
 &= \lk \lk \lk \1 - \varphi \overline{\varphi} \rk^{\!-1}
 \rk^j_i \frac{\p f}{\p z^j} - \lk \lk \1- \overline{\varphi} \varphi
 \rk^{\!-1} \overline{ \varphi } \rk^{\bar{j}}_{i} \frac{\p f}{\pz^j} \rk \lk e^{i_{\varphi}}(dz^i) \rk\\
 &\quad + \lk \lk \lk \1 - \overline{\varphi} \varphi \rk^{\!-1}
 \rk^{\bar{j}}_{\bar{i}} \frac{\p f}{\pz^j} - \lk  \varphi \lk \1- \overline{\varphi}
 \varphi \rk^{\!-1} \rk^j_{\bar{i}} \frac{\p f}{\p z^j} \rk \lk
 \overline{e^{i_{\varphi}}(dz^i)} \rk \\
 &=e^{i_{\varphi}} \lk
 \lk \lk \1 - \varphi \overline{\varphi} \rk^{\!-1} \rk^k_i
 \lk \frac{\p f}{\p z^k} - \overline{\varphi}^j_{\bar{k}} \frac{\p f}{\pz^j} \rk dz^i
 \rk \\
 &\quad +e^{i_{\overline{\varphi}}}\lk \lk \lk \1 - \overline{\varphi} \varphi \rk^{\!-1}
 \rk^{\bar{k}}_{\bar{i}}
 \lk \frac{\p f}{\pz^k} - \varphi^j_{\bar{k}} \frac{\p f}{\p z^j} \rk \dz^i
 \rk.
\end{align*}
Now, let us calculate the second term in the bracket:
\begin{align*}
   &e^{i_{\overline{\varphi}}}\lk \lk \lk \1 - \overline{\varphi} \varphi \rk^{\!-1}
   \rk^{\bar{k}}_{\bar{i}}
 \lk \frac{\p f}{\pz^k} - \varphi^j_{\bar{k}} \frac{\p f}{\p z^j} \rk \dz^i\rk\\
 =&e^{i_{\overline{\varphi}}} \bigg( \lk \1-\overline{\varphi}\varphi\rk ^{\!-1}\lc \db f
 -\lk \1-\overline{\varphi}\varphi\rk ^{\!-1}\lc\varphi\lc \p f \bigg)\\
 =&e^{i_{\overline{\varphi}}}\bigg( \lk \1-\overline{\varphi}\varphi\rk ^{\!-1}\lc(\db-\varphi\lc
 \p)f\bigg).
\end{align*}
Thus,
\begin{equation}\label{dbt-f}
\db_t f= e^{i_{\overline{\varphi}}} \lk \lk \lk \1 -
\overline{\varphi} \varphi \rk^{\!-1} \rk^{\bar{k}}_{\bar{i}}
 \lk \frac{\p f}{\pz^k} - \varphi^j_{\bar{k}} \frac{\p f}{\p z^j} \rk \dz^i \rk
 =e^{i_{\overline{\varphi}}} \bigg( \lk \1-\overline{\varphi}\varphi\rk ^{-1}\lc(\db-\varphi\lc \p)f \bigg)
\end{equation}
since $df$ can be decomposed into $\p_t f + \db_t f$ with respect to
the holomorphic structure on $X_t$. Hence, the desired result follows from the invertibility of
$e^{i_{\overline{\varphi}}}$ and $\lk \1 - \overline{\varphi}
\varphi \rk^{\!\!-1}$.
\end{proof}
See also another proof in \cite[Proposition 3.1]{cs} and our proof
gives an explicit expression of $\db_t$ on the differentiable functions
as in \eqref{dbt-f}. The formula used in the classical proof of
Proposition \ref{hol} is
$$(\db-\varphi\lc \p)f=\lk \1 - \overline{\varphi}\varphi\rk^{\bar{i}}_{\bar{j}}\overline{\p_i \zeta^\alpha}\dz^{\bar{j}}\frac{\p f}{\partial\bar{\zeta}^\alpha},$$
which is just an equivalent version of (\ref{dbt-f})
\[(\db-\varphi\lc \p)f=\lk \1-\overline{\varphi}\varphi\rk\lc e^{-i_{\overline{\varphi}}}(\db_t f)\]
by use of the first formula of Lemma \ref{dwdz}.

By the Leibniz rule, one has
\begin{equation}\label{nn-zz}
\frac{\p z^k}{\p \bar\zeta^\alpha}+\varphi^k_{\bar i}\frac{\pz^i}{\p
\bar\zeta^\alpha}=0,
\end{equation}
which is equivalent to the definition (\ref{krnsh matrix}). In fact,
if (\ref{krnsh matrix}) is assumed, then the Leibniz rule yields
that
\begin{align*}
\frac{\p z^k}{\p\bar\zeta^\alpha}+\varphi^k_{\bar i}\frac{
\pz^i}{\p\bar\zeta^\alpha}
 &=\frac{\p z^k}{\p\bar\zeta^\alpha}+\lk\lk\frac{\p \zeta}{\p z}\rk^{-1}\rk^{k}_{\beta}\frac{\p\zeta^\beta}{\pz^i}\frac{\pz^i}{\p \bar\zeta^\alpha}\\
 &=\frac{\p z^k}{\p\bar\zeta^\alpha}-\lk\lk\frac{\p \zeta}{\p z}\rk^{-1}\rk^{k}_{\beta}\frac{\p\zeta^\beta}{\p z^i}\frac{\p z^i}{\p \bar\zeta^\alpha}\\
 &=0;
\end{align*}
while the converse is similar. Thus, when $f$ satisfies
(\ref{f-hol}), one has
\begin{equation}\label{nn-hol}
\begin{aligned}
\frac{\p f}{\p \bar\zeta^\alpha}
 &=\frac{\p f}{\p z^k}\frac{\p z^k}{\p \bar\zeta^\alpha}+\frac{\p f}{\pz^k}\frac{\pz^k}{\p \bar\zeta^\alpha}\\
 &=\frac{\p f}{\p z^k}\frac{\p z^k}{\p \bar\zeta^\alpha}+\frac{\p f}{\p z^i}\varphi^i_{\bar k}\frac{\pz^k}{\p \bar\zeta^\alpha}\\
 &=\frac{\p f}{\p z^i}\left(\frac{\p z^i}{\p \bar\zeta^\alpha}+\varphi^i_{\bar k}\frac{\pz^k}{\p \bar\zeta^\alpha}\right)\\
 &=0.
\end{aligned}
\end{equation}
Conversely, $\frac{\p f}{\p \bar\zeta^\alpha}=0$ implies that $f$
satisfies (\ref{f-hol}). Actually, we can substitute (\ref{nn-zz}) into
the first equality of (\ref{nn-hol}) to get
$$\frac{\p f}{\p \bar\zeta^\alpha}=\frac{\pz^k}{\p \bar\zeta^\alpha}\lk\frac{\p f}{\pz^k} - \varphi^j_{\bar{k}} \frac{\p f}{\p z^j}\rk.$$
By Lemma \ref{inverse}, one knows that $\frac{\pz^k}{\p
\bar\zeta^\alpha}$ is an invertible matrix as $t$ is small. Hence,
this is the third proof of Proposition \ref{hol}, which is implicit
in Newlander-Nirenberg's proof of their integrability theorem
\cite{NN}.

Let us recall the Newlander-Nirenberg integrability theorem.
Let  $\varphi$ be a holomorphic tangent bundle-valued (0,1)-form
defined on a domain $U$ of $\mathbb{C}^n$ and
$L_i=\db_i-\varphi^j_{\bar{i}}\p_j$. Assume that
$L_1,\cdots,L_n,\bar{L}_1,\cdots,\bar{L}_n$ are linearly
independent, and that they satisfy the integrability condition
(\ref{int-conditin}). Then the system of partial differential
equations
\begin{equation}\label{sys-pde}
L_if=0, i=1,\cdots,n,
\end{equation}
has $n$ linearly independent smooth solutions
$f=\zeta^\alpha=\zeta^\alpha(z),\alpha=1,\cdots,n,$ in a small
neighbourhood of any point of $U$. Here the solutions
$\zeta^1,\cdots,\zeta^n$ are said to be linearly independent if
$$\det \frac{\p(\zeta^1,\cdots,\zeta^n,\overline{\zeta^1},\cdots,\overline{\zeta^n})}{\p(z^1,\cdots,z^n,\overline{z^1},\cdots,\overline{z^n})}\neq0,$$
which obviously implies
$$\det (\1 - \overline{\varphi} \varphi)
\left|\det\frac{\p(\zeta^1,\cdots,\zeta^n)}{\p(z^1,\cdots,z^n)}\right|^2\neq
0$$ since the resolution of the system \eqref{sys-pde} of partial
differential equations yields
\[ \begin{aligned} \begin{pmatrix} \frac{\p \zeta}{\p z}
& \frac{\p \zeta}{\p z}\varphi\\[3pt]
\overline{\left(\frac{\p \zeta}{\p z}\varphi\right)}
& \frac{\p \bar{\zeta}}{\p \bar{z}}\\
\end{pmatrix}\begin{pmatrix} \1
& -\varphi\\
0
& \1\\
\end{pmatrix}
=\begin{pmatrix} \frac{\p \zeta}{\p z}
& 0\\[3pt]
\frac{\p\bar\zeta}{\p z}
& \frac{\p \bar{\zeta}}{\p \bar{z}}\lk \1 - \overline{\varphi} \varphi \rk\\
\end{pmatrix}.
\end{aligned} \]
This theorem, together with Proposition \ref{hol}, is actually the
starting point of Kodaira-Nirenberg-Spencer's existence theorem for
deformations and a quite clear description can be found in
\cite[pp. 268-269]{K2}. We also find that the term $\1 -
\overline{\varphi} \varphi$ in Lemma \ref{inverse} is natural.

Motivated by the new proof of Proposition \ref{hol}, we introduce a
map $$e^{i_{\varphi(t)}|i_{\overline{\varphi(t)}}}:
 A^{p,q}(X_0)\> A^{p,q}(X_t),$$ which plays an important role in
 this paper.
\begin{definition}\label{map}\rm
For $\sigma \in A^{p,q}(X_0)$, we define
$$\label{lbro} e^{i_{\varphi(t)}|i_{\overline{\varphi(t)}}}(\sigma)=
\sigma_{i_1\cdots i_p \bar{j}_1 \cdots \bar{j}_q}(z)
\Big(e^{i_{\varphi(t)}}\left(dz^{i_1}\wedge\cdots\wedge
dz^{i_p}\right)\Big) \wedge
\Big(e^{i_{\overline{\varphi(t)}}}\left(d\overline{z}^{j_1}\wedge\cdots\wedge
d\overline{z}^{j_q}\right)\Big),$$
where $\sigma$ is locally written as
$$\sigma=\sigma_{i_1\cdots i_p \bar{j}_1\cdots
\bar{j}_q}(z)dz^{i_1}\wedge\cdots\wedge dz^{i_p}\wedge
d\overline{z}^{j_1}\wedge\cdots\wedge d\overline{z}^{j_q}$$ and the
operators $e^{i_{\varphi(t)}}$, $e^{i_{\overline{\varphi(t)}}}$
follow the convention:
\begin{equation}\label{e-convention}
e^{\spadesuit}=\sum_{k=0}^\infty \frac{1}{k!} \spadesuit^{k},
\end{equation}
where $\spadesuit^{k}$ denotes $k$-time action of the operator
$\spadesuit$. Since the dimension of $X$ is finite, the summation in
the above formulation is always finite.
\end{definition}

Then we have:
\blemma\label{identification}
The extension map
$e^{i_{\varphi(t)}|i_{\overline{\varphi(t)}}}:
 A^{p,q}(X_0)\> A^{p,q}(X_t)$
is a linear isomorphism as $t$ is arbitrarily small. \elemma \bproof
Notice that $$(dz^1+\varphi(t)\lc dz^1,\cdots, dz^n+\varphi(t)\lc
dz^n)\quad \text{and}\quad (d\overline{z}^1+\overline{\varphi(t)}\lc
d\overline{z}^1,\cdots, d\overline{z}^n+\overline{\varphi(t)}\lc
d\overline{z}^n)$$ are two local bases of $A^{1,0}(X_t)$ and
$A^{0,1}(X_t)$, respectively, thanks to the first identity of Lemma
\ref{dwdz} and the matrix $\lk\frac{\p \zeta^{\alpha}}{\p z^i}\rk$
therein is invertible as $t$ is small. Then the map
$e^{i_{\varphi(t)}|i_{\overline{\varphi(t)}}}$ is obviously
well-defined since $\varphi(t)$ is a well-defined, global
$(1,0)$-vector valued $(0, 1)$-form on $X_0$ as on \cite[pp. 150-151]{MK}.

For the desired isomorphism, we define the inverse map
$$e^{-i_{\varphi(t)}|-i_{\overline{\varphi(t)}}}: A^{p,q}(X_t)\>
A^{p,q}(X_0)$$ of $e^{i_{\varphi(t)}|i_{\overline{\varphi(t)}}}$ as:
\begin{align*}
   & e^{-i_{\varphi(t)}|-i_{\overline{\varphi(t)}}}(\eta)\\
 =&\eta_{i_1\cdots i_p \bar{j}_1 \cdots \bar{j}_q}(\zeta)
 \bigg(e^{-i_{\varphi(t)}}\Big( \big(dz^{i_1}+\varphi(t)\lc
dz^{i_1}\big) \wedge \cdots \wedge
\big(dz^{i_p}+\varphi(t) \lc dz^{i_p}\big) \Big) \wedge\\
&\qquad\qquad\qquad
e^{-i_{\overline{\varphi(t)}}} \Big( \big(d\overline{z}^{j_1}+\overline{\varphi(t)}\lc
d\overline{z}^{j_1}\big) \wedge \cdots \wedge
\big( d\overline{z}^{j_q}+\overline{\varphi(t)}\lc
d\overline{z}^{j_q} \big)\Big)\bigg),
\end{align*}
where $\eta\in A^{p,q}(X_t)$ is locally written as
$$\eta=\eta_{i_1\cdots i_p \bar{j}_1 \cdots \bar{j}_q}(\zeta)
(dz^{i_1}+\varphi(t)\lc dz^{i_1}) \wedge \cdots\wedge
(dz^{i_p}+\varphi(t)\lc dz^{i_p})\wedge
(d\overline{z}^{j_1}+\overline{\varphi(t)}\lc
d\overline{z}^{j_1})\wedge\cdots\wedge
(d\overline{z}^{j_q}+\overline{\varphi(t)}\lc
d\overline{z}^{j_q}),$$
 and the operators $e^{-i_{\varphi(t)}}$,
$e^{-i_{\overline{\varphi(t)}}}$ also follow the convention
(\ref{e-convention}). \eproof

The dual version of the fact about the basis in the proof is used by
K. Chan-Y. Suen \cite{cs} to prove Proposition \ref{hol} and also by
L. Huang in the second paragraph of \cite[Subsection (1.2)]{h}.
Notice that the extension map
$e^{i_{\varphi(t)}|i_{\overline{\varphi(t)}}}$ admits more complete
deformation significance than $e^{i_{\varphi(t)}}$ which extends
only the holomorphic part of a complex differential form.

\begin{lemma}\label{realopt}
The map $e^{i_{\varphi(t)}|i_{\overline{\varphi(t)}}}:
A^{p,q}(X_0)\> A^{p,q}(X_t)$ is a real operator.
\end{lemma}
\begin{proof}
It suffices to prove, for any $\sigma\in A^{p,q}(X_0)$,
$$\overline{e^{i_{\varphi(t)}|i_{\overline{\varphi(t)}}}(\sigma)}=
e^{i_{\varphi(t)}|i_{\overline{\varphi(t)}}}(\overline{\sigma}).$$
In fact, let
$$\sigma=\sum_{|I|=p,|J|=q}\sigma_{I\bar{J}}(z)dz^I\w \dz^J$$
by multi-index notation and then
\begin{align*}
\overline{e^{i_{\varphi(t)}|i_{\overline{\varphi(t)}}}(\sigma)}
 &=\overline{\sigma_{I\bar{J}}(z)e^{i_{\varphi(t)}}(dz^I)\w e^{i_{\overline{\varphi(t)}}}(\dz^J)}\\
 &=\overline{\sigma_{I\bar{J}}(z)}e^{i_{\overline{\varphi(t)}}}(\dz^I)\w e^{i_{\varphi(t)}}(dz^J)\\
 &=\overline{\sigma_{I\bar{J}}(z)}(-1)^{|I|\cdot|J|}e^{i_{\varphi(t)}}(dz^J)\w
 e^{i_{\overline{\varphi(t)}}}(\dz^I)\\
 &=e^{i_{\varphi(t)}|i_{\overline{\varphi(t)}}}(-1)^{|I|\cdot|J|}\overline{\sigma_{I\bar{J}}(z)} dz^J\w \dz^I\\
 &=e^{i_{\varphi(t)}|i_{\overline{\varphi(t)}}}(\overline{\sigma}).
\end{align*}
\end{proof}

\subsection{Obstruction equation} \label{obs-eqn}
This section is to obtain obstruction equation for $\db$-extension, i.e., obstruction equation for extending a $\db$-closed
$(p,q)$-form on $X_0$ to the one on $X_t$.

\blemma\label{dvarphidz} \begin{align*} d \lk e^{i_{\varphi}} \lc
dz^i \rk &= \lk \lk \1- \overline{\varphi} \varphi \rk^{\!-1}
\overline{ \varphi } \rk^{\bar l}_{k} \frac{\p \varphi^i_{\bar
l}}{\p z^j} \lk e^{i_{\varphi}}(dz^k) \rk \wedge\lk
e^{i_{\varphi}}(dz^j) \rk  \\
& \quad - \lk \lk \1- \overline{\varphi} \varphi \rk^{\!-1}
\rk^{\bar l}_{\bar k} \frac{\p \varphi^i_{\bar l}}{\p z^j} \lk
\overline{e^{i_{\varphi}}(dz^k)} \rk \wedge\lk e^{i_{\varphi}}(dz^j)
\rk. \end{align*} \elemma

\bproof Here we use Proposition \ref{main1}. By \eqref{ext-old}, one
has
$$ \begin{aligned} d \lk
e^{i_{\varphi}}(dz^i) \rk
&= (d\circ e^{i_{\varphi}}-e^{i_{\varphi}}\circ d)(dz^i)\\
&=e^{i_{\varphi}}(\p\circ i_{\varphi}-i_{\varphi}\circ \p)(dz^i)\\
&=\frac{\p \varphi^i_{\bar l}}{\p z^j} \lk
e^{i_{\varphi}}(dz^j)\rk\wedge \dz^l.
\end{aligned} $$
Moreover, we have
\begin{equation}
\begin{aligned}\label{dbtf1}
\dz^l
 &=\frac{\p \bar{z}^l}{\p \zeta^{\alpha}} d\zeta^\alpha + \frac{\p \bar{z}^l}{\partial\bar{\zeta}^{\beta}} d\overline{\zeta}^{\beta}\\
 &= \frac{\p \bar{z}^l}{\p \zeta^{\alpha}} \frac{\partial \zeta^\alpha}{\p z^i}\lk e^{i_{\varphi}}(dz^i) \rk
  + \frac{\p \bar{z}^l}{\partial\bar{\zeta}^{\beta}} \overline{\frac{\partial \zeta^\beta}{\p z^i} \lk e^{i_{\varphi}}(dz^i) \rk}\\
 &=  - \lk \lk \1- \overline{\varphi} \varphi
 \rk^{\!-1} \overline{ \varphi } \rk^{\bar{l}}_{k}\lk e^{i_{\varphi}}(dz^k) \rk + \lk \lk \1- \overline{\varphi} \varphi \rk^{\!-1}
\rk^{\bar l}_{\bar k} \lk \overline{e^{i_{\varphi}}(dz^k)} \rk.
\end{aligned}
\end{equation}

\eproof

For a general $\sigma\in A^{p,q}(X_0)$, Proposition \ref{main1} and the
integrability condition \eqref{int-conditin} give
\begin{align}\label{2.5.1}
  \begin{split}
    d(e^{i_{\varphi}|i_{\b{\varphi}}}(\sigma))&=d\circ e^{i_{\varphi}}\circ e^{-i_{\varphi}}\circ e^{i_{\varphi}|i_{\b{\varphi}}}(\sigma)\\
    &=e^{i_{\varphi}}\circ\left([\p,i_{\varphi}]+\b{\p}+\p\right)\circ e^{-i_{\varphi}}\circ e^{i_{\varphi}|i_{\b{\varphi}}}(\sigma)\\
    &=e^{i_{\varphi}|i_{\b{\varphi}}}\circ\left(e^{-i_{\varphi}|-i_{\b{\varphi}}}\circ e^{i_{\varphi}}\circ\left([\p,i_{\varphi}]+\b{\p}+\p\right)
    \circ e^{-i_{\varphi}}\circ
    e^{i_{\varphi}|i_{\b{\varphi}}}(\sigma)\right).
  \end{split}
\end{align}
Here
$$e^{-i_{\varphi(t)}|-i_{\overline{\varphi(t)}}}: A^{p,q}(X_t)\>
A^{p,q}(X_0)$$
 is the
 inverse map of
$e^{i_{\varphi(t)}|\iota_{\overline{\varphi(t)}}}$ as defined in the proof of Lemma \ref{identification}.
We introduce one more new notation $\Finv$ to denote the
\emph{simultaneous contraction} on each component of a complex
differential form as in \cite[Subsection 2.1]{RwZ}. For example,
$(\1-\b{\varphi}\varphi+\b{\varphi})\Finv\sigma$ means that the
operator $(\1-\b{\varphi}\varphi+\b{\varphi})$ acts on $\sigma$
simultaneously as:
\begin{align}\label{2.8}
\begin{split}
&(\1-\b{\varphi}\varphi+\b{\varphi})\Finv(f_{i_1\cdots
i_p\o{j_1}\cdots\o{j_q}}dz^{i_1}\wedge\cdots \wedge dz^{i_p}\wedge
d\b{z}^{j_1}\wedge\cdots \wedge d\b{z}^{j_q})
\\
=&\sigma_{i_1\cdots
i_p\o{j_1}\cdots\o{j_q}}(\1-\b{\varphi}\varphi+\b{\varphi})\l
dz^{i_1} \wedge\cdots\wedge(\1-\b{\varphi}\varphi +\b{\varphi})\l
dz^{i_p}
\\&\qquad\quad\quad\wedge(\1-\b{\varphi}\varphi+\b{\varphi})\l
d\b{z}^{j_1}\wedge\cdots\wedge(\1-\b{\varphi}\varphi+\b{\varphi})\l
d\b{z}^{j_q},
\end{split}
\end{align}
if $\sigma$ is locally expressed by:
$$\sigma=\sigma_{i_1\cdots i_p\o{j_1}\cdots\o{j_q}}dz^{i_1}\wedge \cdots\wedge dz^{i_p}\wedge d\b{z}^{j_1}\wedge \cdots\wedge d\b{z}^{j_q}.$$
This new simultaneous contraction is well-defined since $\varphi(t)$
is a global $(1,0)$-vector valued $(0, 1)$-form on $X_0$ (on
\cite[pp. $150-151$]{MK}) as reasoned in the proof of Lemma
\ref{identification}. Using this notation, one can rewrite the extension map
$e^{i_{\varphi}|i_{\b{\varphi}}}$ in Definition \ref{map}:
$$e^{i_{\varphi}|i_{\b{\varphi}}}=(\1+\varphi+\b{\varphi})\Finv.$$
Then one has:
\begin{lemma}[{\cite[Lemmata $2.2$+$2.3$]{RwZ}}]
For any $\sigma\in A^{p,q}(X_0)$,
\begin{align}\label{2.3}
  e^{-i_{\varphi}}\circ e^{i_{\varphi}|i_{\b{\varphi}}}(\sigma)=(\1-\b{\varphi}\varphi+\b{\varphi})\Finv\sigma
\end{align}
and
 \begin{align}\label{2.4.1}
   e^{-i_{\varphi}|-i_{\b{\varphi}}}\circ e^{i_{\varphi}}(\sigma)=\left((\1-\b{\varphi}\varphi)^{-1}-(\1-\b{\varphi}\varphi)^{-1}\b{\varphi}\right)\Finv\sigma,
 \end{align}
 where $\left((\1-\b{\varphi}\varphi)^{-1}-(\1-\b{\varphi}\varphi)^{-1}\b{\varphi}\right)$ acts on $\sigma$ just as
 (\ref{2.8}).
\end{lemma}

\begin{proof} Here we give a different proof from those in \cite[Lemmata $2.2$+$2.3$]{RwZ}. Locally set
$$\sigma=\sigma_{I_p\b J_q}dz^{I_p}\wedge d\b z^{J_q}$$
by multi-index notation. So
$$e^{i_{\varphi}|i_{\b{\varphi}}}(\sigma)=\sigma_{I_p\b J_q}e^{i_{\varphi}}(dz^{I_p})\wedge e^{i_{\b{\varphi}}}(d\b z^{J_q})$$
and thus,
$$e^{-i_{\varphi}}\circ e^{i_{\varphi}|i_{\b{\varphi}}}(\sigma)
=\sigma_{I_p\b J_q} dz^{I_p}\wedge e^{-i_{\varphi}}\circ e^{i_{\b{\varphi}}}(d\b z^{J_q})
=\sigma_{I_p\b J_q} dz^{I_p}\wedge (\1-\b{\varphi}\varphi+\b{\varphi})\Finv(d\b z^{J_q}).$$
As for \eqref{2.4.1}, \eqref{dbtf1} tells us that
\begin{align*}
   e^{-i_{\varphi}|-i_{\b{\varphi}}}\circ e^{i_{\varphi}}(\sigma)
 &=\sigma_{I_p\b J_q} e^{-i_{\varphi}|-i_{\b{\varphi}}}(e^{i_{\varphi}}(dz^{I_p})\wedge d\b z^{J_q})\\
 &=\sigma_{I_p\b J_q} e^{-i_{\varphi}|-i_{\b{\varphi}}}\left(e^{i_{\varphi}}(dz^{I_p})\wedge e^{i_{\varphi}|i_{\b{\varphi}}}\left((\1-\b{\varphi}\varphi)^{-1}-(\1-\b{\varphi}\varphi)^{-1}\b{\varphi}\right)\Finv d\b z^{J_q}\right)\\
 &=\sigma_{I_p\b J_q} dz^{I_p}\wedge \left((\1-\b{\varphi}\varphi)^{-1}-(\1-\b{\varphi}\varphi)^{-1}\b{\varphi}\right)\Finv d\b z^{J_q}.
\end{align*}
\end{proof}

The following equivalence describes the $\db$-extension obstruction for
$(p,q)$-forms of the smooth family.
\begin{proposition}\label{extension-in} For any $\sigma\in A^{p,q}(X_0)$,
$$\b{\p}_t(e^{i_{\varphi}|i_{\b{\varphi}}}(\sigma))=0$$
amounts to
$$([\p,i_{\varphi}]+\b{\p})(\1-\b{\varphi}\varphi)\Finv\sigma=0.$$
\end{proposition}
\begin{proof}
Substituting (\ref{2.3}) and (\ref{2.4.1}) into (\ref{2.5.1}), one has
\begin{equation}
 \begin{aligned}\label{2.7}
    &d(e^{i_{\varphi}|i_{\b{\varphi}}}(\sigma))\\=&
    e^{i_{\varphi}|i_{\b{\varphi}}}\left(\left((\1-\b{\varphi}\varphi)^{-1}-(\1-\b{\varphi}\varphi)^{-1}\b{\varphi}\right)\Finv\left([\p,i_{\varphi}]
    +\b{\p}+\p\right)(\1-\b{\varphi}\varphi+\b{\varphi})\Finv\sigma\right).
 \end{aligned}
\end{equation}
From (\ref{2.4.1}), we know that
$$
  e^{-i_{\varphi}|-i_{\b{\varphi}}}\circ e^{i_{\varphi}}: A^{p,q}(X_0)\to \bigoplus_{i=0}^{\min\{q,n-p\}}A^{p+i,q-i}(X_0).
$$
Thus, by carefully comparing the form types in both sides  of
(\ref{2.7}), we have
$$\label{2.6}
   \b{\p}_t(e^{i_{\varphi}|i_{\b{\varphi}}}(\sigma))
   =e^{i_{\varphi}|i_{\b{\varphi}}}\left((\1-\b{\varphi}\varphi)^{-1}\Finv([\p,i_{\varphi}]+\b{\p})(\1-\b{\varphi}\varphi)\Finv\sigma\right),
$$
which implies the desired equivalence follows from the invertibility of the operators
$e^{i_{\varphi}|i_{\b{\varphi}}}$ and
$(\1-\b{\varphi}\varphi)^{-1}\Finv$.
\end{proof}

\subsection{Kuranishi family and Beltrami differentials}\label{reduction}
By (the proof of) Kuranishi's completeness theorem \cite{ku}, for any compact complex manifold $X_0$, there exists a complete holomorphic family
$\varpi:\mc{K}\to T$ of complex manifolds at the reference point $0\in T$ in the sense that for any differentiable family $\pi:\mc{X}\to B$ with $\pi^{-1}(s_0)=\varpi^{-1}(0)=X_0$, there is a sufficiently small neighborhood $E\subseteq B$ of $s_0$, and smooth maps $\Phi: \mathcal {X}_E\rightarrow \mathcal {K}$,  $\tau: E\rightarrow T$ with $\tau(s_0)=0$ such that the diagram commutes
$$\xymatrix{\mathcal {X}_E \ar[r]^{\Phi}\ar[d]_\pi& \mathcal {K}\ar[d]^\varpi\\
(E,s_0)\ar[r]^{\tau}  & (T,0),}$$
$\Phi$ maps $\pi^{-1}(s)$ biholomorphically onto $\varpi^{-1}(\tau(s))$ for each $s\in E$, and $$\Phi: \pi^{-1}(s_0)=X_0\rightarrow \varpi^{-1}(0)=X_0$$ is the identity map.
This family is called \emph{Kuranishi family} and constructed as follows. Let $\{\eta_\nu\}_{\nu=1}^m$ be a basis for $\mathbb{H}^{0,1}(X_0,T^{1,0}_{X_0})$, where some suitable Hermitian metric is fixed on $X_0$ and $m\geq 1$; Otherwise the complex manifold $X_0$ would be \emph{rigid}, i.e., for any differentiable family $\kappa:\mc{M}\to P$ with $s_0\in P$ and $\kappa^{-1}(s_0)=X_0$, there is a neighborhood $V \subseteq P$ of $s_0$ such that $\kappa:\kappa^{-1}(V)\to V$ is trivial. Then one can construct a holomorphic family
$$\label{phi-ps-pp}\varphi(t) = \sum_{|I|=1}^{\infty}\varphi_{I}t^I:=\sum_{j=1}^{\infty}\varphi_j(t),\ I=(i_1,\cdots,i_m),\ t=(t_1,\cdots,t_m)\in \mathbb{C}^m,$$
 {for $|t|< \rho$ a small positive constant,} of Beltrami differentials as follows:
$$\label{phi-ps-0}
 \varphi_1(t)=\sum_{\nu=1}^{m}t_\nu\eta_\nu
$$
and for $|I|\geq 2$,
$$\label{phi-ps}
  \varphi_I=\frac{1}{2}\db^*\G\sum_{J+L=I}[\varphi_J,\varphi_{L}],
$$
where $\G$ is the associated Green's operator.
It is obvious that $\varphi(t)$ satisfies the equation
$$\varphi(t)=\varphi_1+\frac{1}{2}\db^*\G[\varphi(t),\varphi(t)].$$
Let
$$T=\{t\ |\ \mathbb{H}[\varphi(t),\varphi(t)]=0 \},$$
where $\mathbb{H}$ is the associated harmonic projection.
Thus, for each $t\in T$, $\varphi(t)$ satisfies
\begin{equation}\label{int}
\b{\p}\varphi(t)=\frac{1}{2}[\varphi(t),\varphi(t)],
\end{equation}
and determines a complex structure $X_t$ on the underlying differentiable manifold of $X_0$. More importantly, $\varphi(t)$ represents the complete holomorphic family $\varpi:\mc{K}\to T$ of complex manifolds. Roughly speaking, Kuranishi family $\varpi:\mc{K}\to T$ contains all sufficiently small differentiable deformations of $X_0$.
We call $\varphi(t)$ the \emph{canonical family of Beltrami differentials} for this Kuranishi family.

By means of these, one can reduce our argument on the deformation invariance of Hodge numbers for a smooth family of complex manifolds to that of the Kuranishi family by shrinking $E$ if necessary, that is, one considers the Kuranishi family with the canonical family of Beltrami differentials constructed as above. From now on, one uses $\varphi(t)$ and  $\varphi$ interchangeably to denote this holomorphic family of integrable Beltrami differentials, and assumes $m=1$ for simplicity.

\section{Deformation invariance of Hodge numbers and its
applications}\label{def} Throughout this section, one just considers the Kuranishi family
$\pi:\mathcal{X}\rightarrow \Delta_{\epsilon}$  of $n$-dimensional complex manifolds over a small complex disk
with the general fibers $X_t:=\pi^{-1}(t)$ according to the reduction in Subsection \ref{reduction}
and fixes a Hermitian metric $g$ on the central fiber $X_0$.
As a direct application of the extension formulae developed in Section \ref{ext-formula}, we
obtain several deformation invariance theorems of Hodge numbers in this
section.

\subsection{Basic philosophy, main results and examples}\label{cf}

Now let us describe our basic philosophy to consider the deformation
invariance of Hodge numbers briefly. The Kodaira-Spencer's upper
semi-continuity theorem (\cite[Theorem $4$]{KS}) tells us that the
function
$$t\longmapsto h^{p,q}_{\db_t}(X_t):=\dim_{\mathbb{C}}H^{p,q}_{\db_t}(X_t)$$
is always upper semi-continuous for $t\in \Delta_{\varepsilon}$ and
thus, to approach the deformation invariance of $h^{p,q}_{\db_t}(X_t)$, we
only need to obtain the lower semi-continuity. Here our main
strategy is a modified iteration procedure, originally from
\cite{LSY} and developed in
\cite{Sun,SY,RZ,lry}, which is to look for an injective extension map from
$H^{p,q}_{\db}(X_0)$ to $H^{p,q}_{\db_t}(X_t)$.
More precisely, for a nice uniquely-chosen representative $\sigma_0$ of the initial Dolbeault
cohomology class 
in $H^{p,q}_{\db}(X_0)$, we try to
construct a convergent power series
$$\label{ps} \sigma_t=\sigma_0+\sum_{j+k=1}^\infty t^k t^{\bar{j}}\sigma_{k\bar{j}}\in
A^{p,q}(X_0),
$$
with $\sigma_t$ varying smoothly on $t$ such
that for each small $t$:

\begin{enumerate}
\item\label{S1} $e^{i_{\varphi}|i_{\overline{\varphi}}}(\sigma_t)\in A^{p,q}(X_t)$
is $\db_t$-closed with respect to the holomorphic structure on $X_t$;
\item\label{S2} The extension map $H^{p,q}_{\db}(X_0) \rightarrow H^{p,q}_{\db_t}(X_t):[\sigma_0]_{\db} \mapsto
[e^{i_{\varphi}|i_{\overline{\varphi}}}(\sigma_t)]_{\db_t}$ is injective.
\end{enumerate}

The key point is to solve the
obstruction equation, induced by the canonical family $\varphi(t)$
of Beltrami differentials, for the $\db_t$-closedness in \eqref{S1},
and verification of the injectivity of the extension map in \eqref{S2}.
Then we state the main theorem of this section, whose proof
will be postponed to Subsection \ref{sbs-pq}.
\begin{theorem}\label{inv-pq}
If the injectivity of the mappings $\iota_{BC,\p}^{p+1,q},\iota_{\db,A}^{p,q+1}$ on the central fiber
$X_0$ and the deformation invariance of the $(p,q-1)$-Hodge number $h^{p,q-1}_{\db_t}(X_t)$ holds,
then $h^{p,q}_{\db_t}(X_t)$ are deformation invariant.
\end{theorem}

There are three conditions involved in the theorem above, namely
the injectivity of the mappings $\iota^{p+1,q}_{BC,\p}$, $\iota^{p,q+1}_{\db,A}$ and
the deformation invariance of the $(p,q-1)$-Hodge number, to assure the
deformation invariance of the one of $(p,q)$-type. Resorting to Hodge,
Bott-Chern and Aeppli numbers of manifolds in the Kuranishi family of the Iwasawa manifold (cf. \cite[Appendix]{A}),
we find the following three examples that the deformation invariance of the $(p,q)$-Hodge number
fails when one of the three conditions is not true, while the other two hold. It indicates that the three
conditions above may not be omitted in order to state a theorem for the deformation invariance
of all the $(p,q)$-Hodge numbers.

Let $\mathbb{I}_3$ be the Iwasawa manifold of complex dimension $3$ with $\varphi^1,\varphi^2,\varphi^3$
denoted by the basis of the holomorphic one form $H^0(\mathbb{I}_3,\Omega^1)$ of $\mathbb{I}_3$, satisfying
the relation \[ d\varphi^1 =0,\ d \varphi^2=0,\ d\varphi^3 = - \varphi^1 \w \varphi^2. \]
And the convention $\varphi^{12\bar{1}\bar{3}}:= \varphi^1 \w \varphi^2 \w \overline{\varphi}^1 \w \overline{\varphi}^3$
will be used for simplicity.
\begin{example}[The case $(p,q)=(1,0)$]\label{ex10}
The injectivity of $\iota^{1,1}_{\db,A}$ holds on $\mathbb{I}_3$
with the deformation invariance of $h^{1,-1}_{\db_t}(X_t)$ trivially established but
$\iota^{2,0}_{BC,\p}$ is not injective. In this case, $h^{1,0}_{\db_t}(X_t)$ are deformation variant.
\end{example}
\begin{proof}
It is revealed from \cite[Appendix]{A} that $h^{1,1}_{\db}=6,h^{1,1}_{A}=8$ and $h^{2,0}_{BC}=3,h^{2,0}_{\p}=2$.
And thus $\iota^{2,0}_{BC,\p}$ is not injective. It is easy to check that
\[ \begin{aligned}
H^{1,1}_{\db}(X) &= \langle[\varphi^{1\bar{1}}]_{\db}, [\varphi^{1\bar{2}}]_{\db},[\varphi^{2\bar{1}}]_{\db},[\varphi^{2\bar{2}}]_{\db}, [\varphi^{3\bar{1}}]_{\db},[\varphi^{3\bar{2}}]_{\db}\rangle, \\
H^{1,1}_{A}(X) &= \langle[\varphi^{1\bar{1}}]_{A}, [\varphi^{1\bar{2}}]_{A},[\varphi^{2\bar{1}}]_{A},[\varphi^{2\bar{2}}]_{A}, [\varphi^{3\bar{1}}]_{A},[\varphi^{3\bar{2}}]_{A},[\varphi^{1\bar{3}}]_{A},[\varphi^{2\bar{3}}]_{A}\rangle, \\
\end{aligned}\]
which implies the injectivity of $\iota^{1,1}_{\db,A}$.
The deformation variance of $h^{1,0}_{\db_t}(X_t)$ can be read from \cite[Appendix]{A}.
\end{proof}

\begin{example}[The case $(p,q)=(2,0)$]\label{ex20}
The injectivity of $\iota^{3,0}_{BC,\p}$ holds on $\mathbb{I}_3$
with the deformation invariance of $h^{2,-1}_{\db_t}(X_t)$ trivially established but
$\iota^{2,1}_{\db,A}$ is not injective. In this case, $h^{2,0}_{\db_t}(X_t)$ are deformation variant.
\end{example}

\begin{proof}
We know that $h^{3,0}_{BC} =1, h^{3,0}_{\p}=1$ and $h^{2,1}_{\db}=6,h^{2,1}_{A}=6$ from \cite[Appendix]{A}.
The bases of respective cohomology groups can be illustrated as follow:
\[\begin{aligned}
H^{3,0}_{BC} &= \langle [\varphi^{123}]_{BC}\rangle, H^{3,0}_{\p}= \langle [\varphi^{123}]_{\p}\rangle; \\
H^{2,1}_{\db} &= \langle [\varphi^{12\bar{1}}]_{\db},[\varphi^{12\bar{2}}]_{\db},[\varphi^{13\bar{1}}]_{\db},[\varphi^{13\bar{2}}]_{\db},
[\varphi^{23\bar{1}}]_{\db},[\varphi^{23\bar{2}}]_{\db} \rangle, \\
H^{2,1}_{A} &= \langle \qquad \qquad \qquad\ [\varphi^{13\bar{1}}]_{A},[\varphi^{13\bar{2}}]_{A},
[\varphi^{23\bar{1}}]_{A},[\varphi^{23\bar{2}}]_{A}, [\varphi^{13\bar{3}}]_{A},[\varphi^{13\bar{2}}]_{A} \rangle,
\end{aligned}\]
which indicates the injectivity of $\iota^{3,0}_{BC,\p}$ and non-injectivity of $\iota^{2,1}_{\db,A}$.
The deformation variance of $h^{2,0}_{\db_t}(X_t)$ can be also got from \cite[Appendix]{A}.
\end{proof}

\begin{example}[The case $(p,q)=(2,3)$]\label{ex23}
The mapping $\iota^{3,3}_{BC,\p}$ is injective on $\mathbb{I}_3$
with the injectivity of $\iota^{2,4}_{\db,A}$ trivially established but
$h^{2,2}_{\db_t}(X_t)$ are deformation variant. In this case, $h^{2,3}_{\db_t}(X_t)$ are deformation variant.
\end{example}

\begin{proof}
It is obvious that $\iota^{3,3}_{BC,\p}$ is injective, since $h^{3,3}_{BC}=1,h^{3,3}_{\p}=1$ and
\[ H^{3,3}_{BC}=\langle [\varphi^{123\bar{1}\bar{2}\bar{3}}]_{BC} \rangle,
H^{3,3}_{\p}=\langle [\varphi^{123\bar{1}\bar{2}\bar{3}}]_{\p} \rangle.\]
And \cite[Appendix]{A} conveys the fact of the deformation variance of $h^{2,2}_{\db_t}(X_t)$ and $h^{2,3}_{\db_t}(X_t)$.
\end{proof}

It is observed that the injectivity of $\iota^{p+1,q}_{BC,\p}$ or $\iota^{p,q+1}_{\db,A}$ is equivalent
to a certain type of $\p\db$-lemma, for which we introduce the following notations:
\begin{notation}\label{notaion-class}\rm
We say a compact complex manifold $X$ satisfies $\mathbb{S}^{p,q}$
and $\mathbb{B}^{p,q}$, if for any $\db$-closed
$\p g\in A^{p,q}(X)$, the equation
\begin{equation}\label{dbar-equ}
\db x = \p g
\end{equation} has a solution 
and a $\p$-exact solution, respectively. Similarly,
a compact complex manifold $X$ is said to satisfy
$\mathcal{S}^{p,q}$
and $\mathcal{B}^{p,q}$, if
for any $\db$-closed $g\in A^{p-1,q}(X)$, the equation
\eqref{dbar-equ} has a solution 
and a $\p$-exact solution, respectively.

The following implications clearly hold
\[ \begin{array}{ccc}
 \mathbb{B}^{p,q} & \Rightarrow & \mathbb{S}^{p,q} \\
 \Downarrow &   & \Downarrow \\
\mathcal{B}^{p,q} & \Rightarrow & \mathcal{S}^{p,q}. \\
\end{array} \]
And it is apparent that a compact complex manifold $X$, where the $\p\db$-lemma holds, satisfies
$\mathbb{B}^{p,q}$ for any $(p,q)$. Here the \emph{$\p\db$-lemma}
refers to:  for every pure-type $d$-closed form on $X$, the
properties of $d$-exactness, $\partial$-exactness,
$\bar\partial$-exactness and $\partial\bar\partial$-exactness are
equivalent.
\end{notation}
It is easy to check that the following equivalent statements:
\[ \text{the injectivity of $\iota^{p,q}_{BC,\p}$ holds on $X$} \Leftrightarrow
\text{$X$ satisfies $\mathbb{B}^{p,q}$};\]
\[ \text{the injectivity of $\iota^{p,q}_{\db,A}$ holds on $X$} \Leftrightarrow
\text{$X$ satisfies $\mathbb{S}^{p,q}$}; \]
\[ \text{the surjectivity of $\iota^{p-1,q}_{BC,\db}$ holds on $X$} \Leftrightarrow
\text{$X$ satisfies $\mathcal{B}^{p,q}$}.\]
Details of the proofs of theorems in this section will frequently apply
Notation \ref{notaion-class} for the convenience of solving $\db$-equations.

The speciality of the types may lead to the weakening
of the conditions in Theorem \ref{inv-pq}, such as $(p,0)$ and $(0,q)$.
Hence, another two theorems follow, whose proofs will be given in Subsection \ref{p00q-section}.

\begin{theorem}\label{inv-p0}
If the injectivity of the mappings $\iota^{p+1,0}_{\db,A}$ and $\iota^{p,1}_{\db,A}$
on $X_0$ holds, then $h^{p,0}_{\db_t}(X_t)$ are independent of $t$.
\end{theorem}
\begin{theorem}\label{inv-0q}
If the surjectivity of the mapping $\iota^{0,q}_{BC,\db}$ on $X_0$
and the deformation invariance of $h^{0,q-1}_{\db_t}(X_t)$ holds, then
$h^{0,q}_{\db_t}(X_t)$ are independent of $t$.
\end{theorem}

\begin{remark}\label{01-sgg}
In the case of $q=1$ of Theorem \ref{inv-0q}, the surjectivity of the mapping $\iota^{0,1}_{BC,\db}$
is equivalent to the \textbf{sGG} condition proposed by Popovici-Ugarte \cite{P1,PU},
from \cite[Theorem 2.1 (iii)]{PU}.
\end{remark}

Hence, the \textbf{sGG} manifolds can be examples of Theorem \ref{inv-0q},
where the Fr\"olicher spectral sequence does not  necessarily degenerate at the $E_1$-level, by
\cite[Proposition 6.3]{PU}. Inspired by the deformation invariance of the $(0,1),(0,2)$ and $(0,3)$-Hodge numbers of the Iwasawa
manifold $\mathbb{I}_3$ shown in \cite[Appendix]{A}, we prove

\begin{corollary}\label{cplx-prl}
Let $X = \Gamma \backslash G$ be a complex parallelizable nilmanifold of complex dimension $n$,
where $G$ is a simply connected nilpotent Lie group and $\Gamma$ is denoted by
a discrete and co-compact subgroup of $G$.
Then $X$ is an \textbf{sGG} manifold. In addition, the $(0,q)$-Hodge numbers of $X$
are deformation invariant for $1 \leq q \leq n$.
\end{corollary}

\begin{proof}
It is well known from \cite[Theorem 1]{Sak} and \cite[Theorem 3.8]{A} that the isomorphisms
\[ H^{p,q}_{BC}(X) \cong H^{p,q}_{BC}(\mathfrak{g},J), H^{p,q}_{\db}(X) \cong H^{p,q}_{\db}(\mathfrak{g},J), \]
hold on the complex parallelizable nilmanifold $X$, where $\mathfrak{g}$ is the corresponding Lie algebra of $G$
and $J$ denotes the complex parallelizable structure on $\mathfrak{g}$. Then from Theorem \ref{inv-0q},
the corollary amounts to the verification of the surjectivity of the mappings of $\iota_{BC,\db}^{0,q}$
on the level of the Lie algebra $(g,J)$ for $1 \leq q \leq n$, which is equivalent to that
the conditions $\mathcal{B}^{1,q}$ hold on the Lie algebra $(g,J)$ for $1 \leq q \leq n$.

Since $J$ is complex parallelizable, it yields that
$ d \mathfrak{g}^{*(1,0)} \subseteq \bigwedge^{2} \mathfrak{g}^{*(1,0)}$,
which implies that $\p \big( \bigwedge^{q} \mathfrak{g}^{*(0,1)} \big) =0$ for $1 \leq q \leq n$, where
$\mathfrak{g}^{*}_{\mathbb{C}} = \mathfrak{g}^{*} \otimes_{\mathbb{R}} \mathbb{C} = \mathfrak{g}^{*(1,0)} \oplus \mathfrak{g}^{*(0,1)}$
with respect to $J$.
Therefore, the conditions $\mathcal{B}^{1,q}$ for $1 \leq q \leq n$ are satisfied on the Lie algebra $(g,J)$
and the corollary follows.
\end{proof}

\begin{remark}
The deformation invariance for the $(0,2)$-Hodge number
of a complex parallelizable nilmanifold has been shown
in \cite[Corollary 4.3]{LUV}.
\end{remark}

Since nilmanifolds with complex parallelizable structures and abelian complex structures
are conjugate to some extent, it is tempting to consider the deformation invariance of
the $(p,0)$-Hodge numbers of nilmanifolds with abelian complex structures for $1 \leq p \leq n$
under the spirit of Corollary \ref{cplx-prl}. The following example, inspired by Console-Fino-Poon \cite[Section 6]{CFP},
is a holomorphic family of nilmanifolds of complex dimension $5$, whose central fiber is endowed with an abelian complex structure.
This family admits the deformation invariance of the $(p,0)$-Hodge numbers for $1 \leq p \leq 5$,
but not the $(1,1)$-Hodge number or $(1,1)$-Bott-Chern number, which shows the function of Theorem \ref{inv-p0} possibly beyond Kodaira-Spencer's squeeze \cite[Theorem 13]{KS} in this case.

\begin{example}\label{cfp}
\emph{Let $X_0$ be the nilmanifold determined by a ten-dimensional $3$-step nilpotent Lie algebra $\mathfrak{n}$ endowed with
the complex structure $J_{s,t}$ for $s=1,t=0$, as in \cite[Section 6]{CFP}. The natural decompositions with respect
to the complex structure $J_{1,0}$ yield
\[ \mathfrak{n}_{\mathbb{C}} = \mathfrak{n} \otimes_{\mathbb{R}} \mathbb{C} = \mathfrak{n}^{1,0} \oplus \mathfrak{n}^{0,1};
\mathfrak{n}^*_{\mathbb{C}} = \mathfrak{n}^* \otimes_{\mathbb{R}} \mathbb{C} = \mathfrak{n}^{*(1,0)} \oplus \mathfrak{n}^{*(0,1)}.\]
By contrast with the basis $\omega^1,\cdots,\omega^5$ of $\mathfrak{n}^{*(1,0)}$ used in \cite[Section 6]{CFP},
another basis $\tau^1,\cdots,\tau^5$ will be applied, with the transition formula given by
\[ \tau^1 = \omega^1, \tau^2 = (1+i)\omega^2 - \omega^3, \tau^3 = -(1+i)\omega^2, \tau^4 =\omega^4, \tau^5 = \omega^5. \]
Hence, the structure equation with respect to $\{\tau^k\}_{k=1}^5$ follows
\begin{equation}\label{str-equ-1}
\begin{cases}
d\tau^1 = d\tau^2 = d\tau^4= 0, \\
d\tau^3 = -\big( \tau^1 \w \bar{\tau}^1 + (1+i) \tau^1 \w \bar{\tau}^4 \big), \\
d\tau^5 = \frac{1}{2} \big( \tau^1 \w \bar{\tau}^3 + \tau^3 \w \bar{\tau}^1 - \tau^2 \w \bar{\tau}^2 \big). \\
\end{cases}
\end{equation}
It is easy to see $d \bar{\tau}^5 = -d \tau^5$, which implies $\p \bar{\tau}^5 =- \db {\tau^5}$.
Denote the basis of $\mathfrak{n}^{1,0}$ dual to $\{ \tau^k \}_{k=1}^5$ by $\theta_1, \cdots, \theta_5$.
The equation $d\omega(\theta,\theta') = - \omega([\theta,\theta'])$ for $\omega \in \mathfrak{n}^{*}_{\mathbb{C}}$
and $\theta,\theta' \in \mathfrak{n}_{\mathbb{C}}$, establishes the equalities
\[ [\bar{\theta}_1,\theta_4]=(1-i) \bar{\theta}_3,\ [\bar{\theta}_i,\theta_4]=0\ \ \text{for}\ 2 \leq i \leq 5.\]
According to \cite[Theorem 3.6]{CFP}, the linear operator $\db$ on $\mathfrak{n}^{1,0}$, defined in \cite[Section 3.2]{CFP} by
\[ \db: \mathfrak{n}^{1,0} \rightarrow \mathfrak{n}^{*(0,1)} \otimes \mathfrak{n}^{1,0}:\ \db_{\bar{U}}V = [\bar{U},V]^{1,0}\ \ \text{for}\ U,V \in \mathfrak{n}^{1,0},\]
produces an isomorphism $H^1(X_0,T^{1,0}_{X_0}) \cong H^1_{\db}(\mathfrak{n}^{1,0})$. Therefore, from Kodaira-Spencer deformation theory,
an analytic deformation $X_t$ of $X_0$ can be constructed by use of the integrable Beltrami differential
\[ \varphi(t) = t_1 \bar{\tau}^5 \otimes \theta_4 + t_2 \bar{\tau}^4 \otimes \theta_4 \]
for $t_1,t_2$ small complex numbers and $t=(t_1,t_2)$, which satisfies $\db \varphi(t) = \frac{1}{2}[\varphi(t),\varphi(t)]$
and the so-called Schouten-Nijenhuis bracket $[\cdot,\cdot]$ (cf. \cite[Formula (4.1)]{CFP}) works as
\[ [\bar{\omega} \otimes V, \bar{\omega}' \otimes V']
=  \bar{\omega}' \w i_{V'} d\bar{\omega} \otimes V + \bar{\omega} \w i_{V} d\bar{\omega}' \otimes V'\ \ \text{for}\ \omega,\omega' \in \mathfrak{n}^{*(1,0)},
V,V' \in \mathfrak{n}^{1,0},\]
since $\db \theta_4=0$ and $i_{\theta_4} d\bar{\tau}^5 = i_{\theta_4} d\bar{\tau}^4 =0 $.
Then the general fibers $X_t$ are still nilmanifolds, determined by the Lie algebra $\mathfrak{n}$ and the decompositions
\[ \mathfrak{n}^*_{\mathbb{C}} = \mathfrak{n}^* \otimes_{\mathbb{R}} \mathbb{C} = \mathfrak{n}^{*(1,0)}_{\varphi(t)} \oplus \mathfrak{n}^{*(0,1)}_{\varphi(t)}, \]
with the basis of $\mathfrak{n}^{*(1,0)}_{\varphi(t)}$ given by $\tau^k(t) = e^{i_{\varphi(t)}} \big( \tau^k \big) = \big(\1 + \varphi(t)\big) \lrcorner \tau^k$ for $1 \leq k \leq 5$.
Hence, the structure equation of $\{ \tau^k(t) \}_{k=1}^5$ amounts to
\begin{equation}\label{str-equ-2}
\begin{cases}
d\tau^1(t)=d\tau^2(t) = 0, \\
d\tau^4(t)=-t_1  d\tau^5(t), \\
d\tau^3(t)=\frac{1+i}{1-|t_2|^2}\big( \bar{t}_2 \tau^1(t) \w \tau^4(t) + \bar{t}_1 \tau^1(t) \w \tau^5(t) \big)\\
\qquad\qquad- \tau^1(t) \w \bar{\tau}^1(t) - \frac{1+i}{1-|t_2|^2}
\big( \tau^1(t) \w \bar{\tau}^4(t) + t_1 \bar{t}_2 \tau^1(t) \w \bar{\tau}^5(t) \big), \\
d\tau^5(t)=\frac{1}{2} \big( \tau^1(t) \w \bar{\tau}^3(t) + \tau^3(t) \w \bar{\tau}^1(t) - \tau^2(t) \w \bar{\tau}^2(t) \big).\\
\end{cases}
\end{equation}
The proof of Theorem \ref{inv-p0}, which is contained in Proposition \ref{thmp0}, shows that
the obstruction of the deformation invariance of the $(p,0)$-Hodge numbers along the family determined by $\varphi(t)$
actually lies in the equation \eqref{p0-extension-equ}, where the differential forms involved are invariant ones in this case.
For any $\db$-closed $\sigma_0 \in \bigwedge^p \mathfrak{n}^{*(1,0)}$, it is easy to check that
\[ \sigma_t = \sigma_0 + t_1 \tau^5 \w (\theta_4 \lrcorner \sigma_0) \]
solves the equation \eqref{p0-extension-equ}, due to the equalities $\p \bar{\tau}^5 =- \db {\tau^5}$
and $d \tau^4 =0$. However, based on the structure equations \eqref{str-equ-1} and \eqref{str-equ-2},
it yields that \[ h^{1,1}_{\db}(X_0) = 14,\ h^{1,1}_{\db_t}(X_t)=11\quad \text{and}\quad  h^{1,1}_{BC}(X_0) = 11,\ h^{1,1}_{BC}(X_t)=9,\]
where $t_2 \neq 0$ and $t_1 \bar{t}_2 - \bar{t}_1 \neq 0$.
}
\end{example}

\subsection{Proofs of the invariance of Hodge numbers $h^{p,q}_{\db_t}(X_t)$}\label{sbs-pq}
This subsection is to prove Theorem \ref{inv-pq}, which can be restated
by use of Notation \ref{notaion-class}:
if the central fiber $X_0$ satisfies both
$\mathbb{B}^{p+1,q}$ and $\mathbb{S}^{p,q+1}$ with
the deformation invariance of $h^{p,q-1}_{\db_t}(X_t)$ established,
then $h^{p,q}_{\db_t}(X_t)$ are independent of $t$.

The basic strategy is described at the beginning of
Subsection \ref{cf} and obviously our task is divided into two steps \eqref{S1} and \eqref{S2},
which are to be completed in Propositions \ref{thmpq} and \ref{pq-pq-1}, respectively.

To complete \eqref{S1},  we need a lemma due to \cite[Theorem $4.1$]{P1} or
\cite[Lemma 3.14]{RwZ} for the resolution of $\p\db$-equations.
\begin{lemma}\label{ddbar-eq}
Let $(X,\omega)$ be a compact Hermitian complex manifold with any suitable
pure-type complex differential forms $x$ and $y$. Assume that the
$\p \db$-equation
\begin{equation}\label{ddb-eq}
\p \db x =y
\end{equation}
admits a solution. Then an explicit solution of the $\p
\db$-equation \eqref{ddb-eq} can be chosen as
$$(\p\db)^*\G_{BC}y,$$
which uniquely minimizes the $L^2$-norms of all the solutions with
respect to $\omega$.
\end{lemma}
Here $\G_{BC}$ is the associated Green's operator of
the \emph{first $4$-th order Kodaira-Spencer
operator} (also often called \emph{Bott-Chern Laplacian}) given by
$$\label{bc-Lap}
\square_{{BC}}=\p\db\db^*\p^*+\db^*\p^*\p\db+\db^*\p\p^*\db+\p^*\db\db^*\p+\db^*\db+\p^*\p.
$$

We need one more lemma
inspired by {\cite[Lemma 3.1]{P2}}.
\begin{lemma}\label{lemma1}
Assume that a compact complex manifold $X$ satisfies $\mathcal{B}^{p+1,q}$.
Each Dolbeault class $[\sigma]_{\db}$ of the $(p,q)$ type can be canonically
represented by a uniquely-chosen $d$-closed $(p,q)$-form $\gamma_{\sigma}$.
\end{lemma}

\begin{proof}
We first choose the unique harmonic representative of $[\sigma]_{\db}$,
still denoted by $\sigma$. It is clear that the $d$-closed representative
$\gamma_{\sigma} \in A^{p,q}(X)$ satisfies
$$\sigma+\db \beta_{\sigma}=\gamma_{\sigma}$$
for some $\beta_{\sigma}\in A^{p,q-1}(X)$. This is equivalent that
some $\beta_{\sigma} \in A^{p,q-1}(X)$ solves the following equation
$$\p\db \beta_{\sigma}=-\p\sigma.$$
The existence of $\beta_{\sigma}$ is assured by our assumption on $X$ and uniqueness with $L^2$-norm minimum by Lemma \ref{ddbar-eq}, that is, one can choose $\beta_{\sigma}$ as $-(\p\db)^*\G_{BC}\p\sigma$.
\end{proof}

\begin{proposition}\label{thmpq}
Assume that $X_0$ satisfies $\mathbb{B}^{p+1,q}$ and $\mathbb{S}^{p,q+1}$.
Then for each Dolbeault class in $H^{p,q}_{\db}(X_0)$ with the unique canonical $d$-closed representative $\sigma_0$
given as Lemma \ref{lemma1}, there exists a power series on $X_0$
$$\sigma_t=\sigma_0+\sum_{j+k=1}^\infty t^k t^{\bar{j}}\sigma_{k\bar{j}}\in
A^{p,q}(X_0),$$ such that $\sigma_t$ varies smoothly on $t$ and
$e^{i_{\varphi}|i_{\overline{\varphi}}}(\sigma_t)\in A^{p,q}(X_t)$ is
$\db_t$-closed with respect to the holomorphic structure on $X_t$.
\end{proposition}
\begin{proof}
The construction of $\sigma_t$ is presented at first.
The canonical choice of the representative for the initial Dolbeault cohomology class is guaranteed by
the assumption that $X_0$ satisfies $\mathbb{B}^{p+1,q}$, which implies that $\mathcal{B}^{p+1,q}$ holds,
and Lemma \ref{lemma1}. By Proposition \ref{extension-in}, the desired
$\db_t$-closedness is equivalent to the resolution of the equation
\begin{equation}\label{pq-extension-equ}
([\p,i_{\varphi}]+\b{\p})(\1-\b{\varphi}\varphi)\Finv\sigma_t=0.
\end{equation}
Set $\widetilde{\sigma}_t=(\1-\b{\varphi}\varphi)\Finv\sigma_t$ and we just need to resolve the system of
equations
\begin{equation}\label{holt}
\begin{cases}{\partial}\widetilde{\sigma}_t=0,\\
\db  \widetilde{\sigma}_t + \p \big({\varphi(t)} \lc \widetilde\sigma_t \big)=0.
\end{cases}
\end{equation}
An iteration method, developed in \cite{LSY,Sun,SY,lry,RZ,RZ2,RZ15,RwZ}, will be
applied to resolve this system. Let
$$\widetilde{\sigma}_t=\widetilde{\sigma}_0+\sum_{j=1}^\infty \widetilde{\sigma}_{j}t^j$$
be a power series of $(p,q)$-forms on $X_0$.
By substituting this power series into \eqref{holt} and comparing the coefficients of $t^k$, we turn to resolving
\begin{equation}\label{holt-3}
\begin{cases}d\widetilde\sigma_0=0,\\
\overline{\partial}{\widetilde\sigma_k}=-\partial(\sum_{i=1}^k\varphi_i\lrcorner{\widetilde\sigma_{k-i}}),\qquad
\text{for each $k\geq 1$},\\
{\partial}{\widetilde\sigma_k}=0,\qquad \text{for each $k\geq 1$}.
\end{cases}
\end{equation}
Notice that $\widetilde\sigma_0=\sigma_0$ and thus $d\widetilde\sigma_0=0$
by the choice of the canonical $d$-closed representative
for the initial Dolbeault class in $H^{p,q}_{\db}(X_0)$.

As for the second equation of \eqref{holt-3}, we may assume that $\widetilde\sigma_{i}$,
satisfying ${\partial}{\widetilde\sigma_i}=0$, has been resolved for $0 \leq i \leq k-1$, and then check
$$\db\partial(\sum_{i=1}^k\varphi_i\lrcorner{\widetilde\sigma_{k-i}})=0.$$
In fact, by the integrability \eqref{int} and the commutator formula \eqref{aaaa}, one has
\begin{equation}\label{dbar-closed-c}
\begin{aligned}
 &-\overline{\partial}\partial\lk\sum_{i=1}^{k}\varphi_i\lrcorner\widetilde\sigma_{k-i}\rk\\
 =&\partial\lk\sum_{i=1}^{k}\overline{\partial}\varphi_i\lrcorner\widetilde\sigma_{k-i}
 +\sum_{i=1}^{k}\varphi_i\lrcorner\overline{\partial}\widetilde\sigma_{k-i}\rk\\
=&\partial\lk\frac{1}{2}\sum_{i=1}^{k}\sum_{j=1}^{i-1}[\varphi_j,\varphi_{i-j}]\lrcorner\widetilde\sigma_{k-i}
 -\sum_{i=1}^{k}\varphi_i\lrcorner{\partial}\bigg(\sum_{j=1}^{k-i}\varphi_{j}\lrcorner\widetilde\sigma_{k-i-j}\bigg)\rk\\
=&\partial\left(\frac{1}{2}\sum_{i=1}^{k}\sum_{j=1}^{i-1}\bigg(-\partial\big(\varphi_{i-j}\lrcorner(\varphi_{j}\lrcorner\widetilde\sigma_{k-i})\big)
 -\varphi_{i-j}\lrcorner(\varphi_{j}\lrcorner\partial\widetilde\sigma_{k-i})\right. \\
  &\quad +\varphi_j\lrcorner\partial(\varphi_{i-j}\lrcorner\widetilde\sigma_{k-i})
  +\varphi_{i-j}\lrcorner\partial(\varphi_j\lrcorner\widetilde\sigma_{k-i})\bigg)
  -\sum_{i=1}^{k}\varphi_{i}\lrcorner\partial \left.
 \bigg(\sum_{j=1}^{k-i}\varphi_j\lrcorner\widetilde\sigma_{k-i-j}\bigg) \right)\\
= &\partial\lk\sum_{1\leq j<i\leq k}\varphi_j\lrcorner\partial(\varphi_{i-j}\lrcorner\widetilde\sigma_{k-i})
 -\sum_{i=1}^{k}\sum_{j=1}^{k-i}\varphi_{i}\lrcorner\partial(\varphi_j\lrcorner\widetilde\sigma_{k-i-j})\rk\\
= &0.
\end{aligned}
\end{equation}
Hence, one can obtain a canonical solution
$$\widetilde{\sigma_k^1}=-\db^*\G_{\db}\partial\left(\sum_{i=1}^k\varphi_i\lrcorner{\widetilde\sigma_{k-i}}\right)$$
by the assumption that $X_0$ satisfies $\mathbb{S}^{p,q+1}$ and the useful fact that $\db^*\G_{\db}y$ is the unique
solution, minimizing the $L^2$-norms of all the solutions, of the
equation
\[ \db x=y \]
on a compact complex manifold if the equation admits one, where
$x,y$ are pure-type complex differential forms and the operator
$\G_{\db}$ denotes the corresponding Green's operator of the
$\db$-Laplacian $\square$.

To fulfill the third equation $\p\widetilde\sigma_k=0$,
we try to find some $\widetilde{\sigma_k^2}\in A^{p,q-1}(X_0)$
such that
\begin{equation}\label{sgm_2}
\p(\widetilde{\sigma_k^1}+\db\widetilde{\sigma_k^2})=0.
\end{equation}
Then the solution $\widetilde{\sigma}_k$ can be set as
$$\widetilde\sigma_k=\widetilde{\sigma_k^1}+\db\widetilde{\sigma_k^2},$$
which satisfies both the second and the third equation of \eqref{holt-3}.
At this moment, the assumption $\mathbb{B}^{p+1,q}$ on $X_0$ and Lemma \ref{lemma1}
will also provide us a solution of \eqref{sgm_2}
$$\widetilde{\sigma_k^2}=-(\p\db)^*\G_{BC}\p\widetilde{\sigma_k^1},$$
which yields
$$\widetilde\sigma_k=-\db^*\G_{\db}\partial\left(\sum_{i=1}^k\varphi_i\lrcorner{\widetilde\sigma_{k-i}}\right)
+\db(\p\db)^*\G_{BC}\p\db^*\G_{\db}\partial\left(\sum_{i=1}^k\varphi_i\lrcorner{\widetilde\sigma_{k-i}}\right).$$

Finally we resort to the elliptic estimates for the regularity of $\widetilde{\sigma}_t$, which is quite analogous to that in \cite[Theorems 2.12 and 3.11]{RwZ}. So we just sketch this argument, which is divided into two steps:
\begin{enumerate}[$(i)$]
    \item \label{reg-1}
$\|\sum_{j=1}^\infty \widetilde{\sigma}_{j}t^j\|_{k, \alpha}\ll A(t)$;
    \item \label{reg-2}
$\widetilde\sigma_t$ is a real analytic family of $(p,q)$-forms in $t$ .
\end{enumerate}
Here are explicit details for the first step \eqref{reg-1}.
 Consider an important power series in deformation theory of complex
structures
\begin{equation}\label{cru-ps}
A(t)=\frac{\beta}{16\gamma}\sum_{m=1}^\infty\frac{(\gamma
t)^m}{m^2}:=\sum_{m=1}^\infty A_m t^m,
\end{equation}
where $\beta, \gamma$ are positive constants to be determined. The
power series \eqref{cru-ps} converges for $|t|<\frac{1}{\gamma}$ and
has a nice property:
$$\label{p-prin}
A^i(t)\ll {\left(\frac{\beta}{\gamma}\right)}^{i-1}A(t).
$$
 See \cite[Lemma 3.6 and its Corollary in Chapter 2]{MK} for these
basic facts. We use the following notation: For the series with
real positive coefficients
$$a(t)=\sum_{m=1}^\infty a_m t^m, \quad b(t)=\sum_{m=1}^\infty b_m t^m,$$
say that \emph{$a(t)$ dominates $b(t)$}, written as $b(t)\ll a(t)$,
if $ b_m\leq a_m$. But for a power series of (bundle-valued) complex differential
forms
$$\eta(t)=\sum_{m= 0}^\infty \eta_{m} t^m,$$ the
notation
$$\|\eta(t)\|_{k, \alpha}\ll A(t)$$ means
$$\|\eta_{m}\|_{k, \alpha}
\leq A_m$$ with the $C^{k, \alpha}$-norm $\|\cdot\|_{k, \alpha}$ as
defined on \cite[p.  159]{MK}.
Recall that the canonical family of Beltrami differentials $\varphi(t)$
satisfies a nice convergence property:
$$\|\varphi(t)\|_{k, \alpha}\ll A(t)$$
as given in the proof of \cite[Proposition 2.4 in Chapter $4$]{MK}.
We need three more a priori
elliptic estimates as follows.
For any complex differential form $\phi$,
$$\label{ee1}
\|\overline{\partial}^*\phi\|_{k-1, \alpha}\leq C_1\|\phi\|_{k,
\alpha},
$$
$$\label{db-ee2}
\|\G_{\db}\phi\|_{k, \alpha}\leq C_{k,\alpha}\|\phi\|_{k-2, \alpha},
$$
where  $k>1$, $C_1$ and $C_{k,\alpha}$
depend only on $k$ and $\alpha$, not on $\phi$, as shown in \cite[Proposition
$2.3$ in Chapter $4$]{MK}, and
$$\label{ee2}
\|\G_{BC}\phi\|_{k, \alpha}\leq C_{k,\alpha}\|\phi\|_{k-4, \alpha},
$$
where $k>3$ and $C_{k,\alpha}$ depends on only on $k$ and $\alpha$,
not on $\phi$, as shown in \cite[Appendix.Theorem $7.4$]{K2} for example.
Based on these, an inductive argument implies
$$\left\|\sum_{j=1}^l \widetilde{\sigma}_{j}t^j\right\|_{k, \alpha}\ll A(t)$$
for any large $l>0$ and each $k>3$. Then \eqref{reg-1} follows.

We proceed to \eqref{reg-2} since there is possibly no uniform
lower bound for the convergence radius obtained in the $C^{k,\alpha}$-norm as $k$ converges to $+\infty$.
Applying the $\db$-Laplacian
$\square=\db^*\db+\db\db^*$ to
$$\widetilde\sigma_t=-\db^*\G_{\db}\partial\left(\varphi\lrcorner{\widetilde\sigma_t}\right)
+\db(\p\db)^*\G_{BC}\p\db^*\G_{\db}\partial\left(\varphi\lrcorner{\widetilde\sigma_t}\right)+\sigma_0$$
and the proof of \cite[Appendix.Theorem 2.3]{K2} or \cite[Proposition 3.15]{RwZ}, one proves the following result.
For each $l=1,2,\cdots$, choose a
smooth function $\eta^l(t)$ with values in $[0,1]$:
$$\label{eta-l}
\eta^l(t)\equiv
    \begin{cases}
      1,\ \text{for $|t|\leq (\frac{1}{2}+\frac{1}{2^{l+1}})r$};\\
      0,\ \text{for $|t|\geq (\frac{1}{2}+\frac{1}{2^{l}})r$},
    \end{cases}
$$
where $r$ is a positive constant to be determined.
Inductively, for any $l=1,2,\cdots$,
$\eta^{2l+1}\widetilde\sigma_t$ is $C^{k+l, \alpha}$, where
 $r$ can be chosen independently of $l$. Since $\eta^{2l+1}(t)$ is
identically equal to $1$ on $|t|<\frac{r}{2}$ which is independent
of $l$, $\widetilde\sigma_t$ is $C^{\infty}$ on $X_0$ with
$|t|<\frac{r}{2}$.  Then $\widetilde\sigma_t$ can be considered as
a real analytic family of $(p,q)$-forms in $t$ and thus it is smooth on $t$.
\end{proof}
In the first version \cite{RZ15} of this paper, we resort to J.
Wavrik's work \cite[Section $3$]{Wa} for the above regularity.

To guarantee \eqref{S2}, it suffices to prove:
\begin{proposition}\label{pq-pq-1}
If the $\db$-extension of $H^{p,q}_{\db}(X_0)$
as in Proposition \ref{thmpq} holds for a complex manifold $X_0$,
then the deformation invariance of $h^{p,q-1}_{\db_t}(X_t)$
assures that the extension map
$$H^{p,q}_{\db}(X_0) \rightarrow H^{p,q}_{\db_t}(X_t):[\sigma_0]_{\db} \mapsto [e^{i_{\varphi}|i_{\overline{\varphi}}}(\sigma_t)]_{\db_t}$$
is injective.
\end{proposition}
\begin{proof}
Let us  fix a family of smoothly varying Hermitian metrics
$\{\omega_t\}_{t \in \Delta_{\epsilon}}$ for the infinitesimal
deformation $\pi: \mathcal{X} \rightarrow \Delta_{\epsilon}$ of
$X_0$. Thus, if the Hodge numbers $h^{p,q-1}_{\db_t}(X_t)$ are deformation
invariant, the Green's operator $\G_t$, acting on the
$A^{p,q-1}(X_t)$, depends differentiably with respect to $t$ from
\cite[Theorem 7]{KS} by Kodaira and Spencer. Using this, one ensures
that this extension map can not send a non-zero class in
$H^{p,q}_{\db}(X_0)$ to a zero class in $H^{p,q}_{\db_t}(X_t)$.

If we suppose that
$$e^{i_{\varphi(t)}|i_{\overline{\varphi(t)}}}(\sigma_t)=\db_t \eta_t$$
for some $\eta_t\in A^{p,q-1}(X_t)$ when $t \in \Delta_{\epsilon}
\setminus \{0\}$, the Hodge decomposition of $\db_t$ and the
commutativity of $\G_t$ with $\dbs_t$ and $\db_t$ yield that
\begin{align*}
e^{i_{\varphi(t)}|i_{\overline{\varphi(t)}}}(\sigma_t)
 &=\db_t \eta_t= \db_t\big(\mathbb{H}_t(\eta_t) + \square_t \mathbb{G}_t \eta_t\big)\\
 &=\db_t\big(\dbs_t\db_t\G_t \eta_t\big)\\
 &=\db_t\G_t\big(\dbs_t\db_t \eta_t\big)\\
 &=\db_t\G_t\big(\dbs_te^{i_{\varphi(t)}|i_{\overline{\varphi(t)}}}(\sigma_t)\big),
\end{align*}
where $\mathbb{H}_t$ and $\square_t$ are the harmonic projectors and
the Laplace operators with respect to $(X_t,\omega_t)$, respectively.
Let $t$ converge to $0$ on both sides of the equality
\[ e^{i_{\varphi(t)}|i_{\overline{\varphi(t)}}}(\sigma_t)
=
\db_t\G_t\big(\dbs_te^{i_{\varphi(t)}|i_{\overline{\varphi(t)}}}(\sigma_t)\big),\]
which turns out that $\sigma_0$ is $\db$-exact on the central fiber
$X_0$. Here we use that the Green's operator $\mathbb{G}_t$ depends
differentiably with respect to $t$.
\end{proof}

\begin{example}[The case $q=n$]\label{qn-qn-1}
The deformation invariance for $h^{p,n}_{\db_t}(X_t)$ can be obtained
from the one for $h^{p,n-1}_{\db_t}(X_t)$.
\end{example}

\begin{proof}
Actually, it is easy to see that
$e^{i_{\varphi(t)}|i_{\overline{\varphi(t)}}}(\sigma)\in A^{p,n}(X_t)$
for any $\sigma\in A^{p,n}(X_0)$.
By the consideration of types, the equality
\begin{equation}\label{0n-ob}
\db_t(e^{i_{\varphi(t)}|i_{\overline{\varphi(t)}}}(\sigma))=0
\end{equation}
trivially holds, without the necessity of the choice of a canonical $d$-closed
representative or solving the equation \eqref{0n-ob} as in Proposition \ref{thmpq}.
And thus, from Proposition \ref{pq-pq-1}, the extension map
$$H^{p,n}_{\db}(X_0) \rightarrow H^{p,n}_{\db_t}(X_t):
[\sigma]_{\db} \mapsto [e^{i_{\varphi}|i_{\overline{\varphi}}}(\sigma)]_{\db}$$ is injective.
We can also revisit this example by \cite[Formula (7.74)]{K2}
$$h^{p,q}_{\db_t}(X_t)+\nu^{q}(t)+\nu^{q+1}(t)=h^{p,q}_{\db}(X),$$
where $\nu^q(t)$ is the number of eigenvalues $\sigma_j^q(t)$ for
the canonical base $f_{tj}^q$ of eigenforms of the Laplacian
$\square_t=\db_t\dbs_t+\dbs_t\db_t$ less than some fixed positive
constant. Notice that $\nu^{n+1}(t)=0$. For more details see
\cite[Section 7.2.(c)]{K2}.
\end{proof}

Proposition \ref{pq-pq-1} and Example \ref{qn-qn-1} are indeed inspired
by Nakamura's work \cite[Theorem 2]{N}, which asserts that all plurigenera
are not necessarily invariant under infinitesimal deformations, particularly
for the Hodge number $h^{n,0}_{\db}$ and thus $h^{0,n}_{\db}$, while the obstruction
equation \eqref{0n-ob} for extending $\db_t$-closed $(0,n)$-forms is
un-obstructed. This example actually tells us that deformation
invariance of $h^{0,n}_{\db}$ relies on the one of $h^{0,n-1}_{\db}$.

\begin{proposition}
If $h^{p,q+1}_{\db}(X_0) = 0$ and the deformation invariance of
$h^{p,q-1}_{\db_t}(X_t)$ holds, then $h^{p,q}_{\db_t}(X_t)$ are deformation invariant.
\end{proposition}
\begin{proof}
With the notations in the proof of Proposition \ref{thmpq},
we can resolve Equation
\eqref{pq-extension-equ} directly, which is equivalent to the
following equation:
\begin{equation}\label{holt-2}
\overline{\partial}\sigma_k =-\partial(\sum_{i=1}^k
\varphi_i\lrcorner\sigma_{k-i})+ \sum_{i=1}^k
\varphi_i\lrcorner\partial\sigma_{k-i}\quad \textrm{for each}\ k
\geq 1,
\end{equation}
by use of the assumption that $h^{p,q+1}_{\db}(X_0) =0$. Also interestingly
notice that we are not able to deal with this case by the system
\eqref{holt-3} of equations.
Set $$\tau_k=-\partial(\sum_{i=1}^k
\varphi_i\lrcorner\sigma_{k-i})+ \sum_{i=1}^k
\varphi_i\lrcorner\partial\sigma_{k-i},$$
$$\eta_k=-\partial(\sum_{i=1}^k
\varphi_i\lrcorner\sigma_{k-i}).$$
When $k=1$, we
have \begin{align*} \db \tau_1 &= \db \big( -\partial( \varphi_1
\lrcorner\sigma_0)+ \varphi_1\lrcorner\partial\sigma_0 \big)\\
& = \p (\db \varphi_1 \lrcorner \sigma_0 + \varphi_1 \lrcorner \db
\sigma_0) + \db \varphi_1 \lrcorner \p \sigma_0 + \varphi_1
\lrcorner \db \p \sigma_0 \\
& =0,
\end{align*}
since $\db \varphi_1 =0$ and $\db \sigma_0 =0$. The assumption
$h^{p,q+1}_{\db}(X_0) =0$ implies that the equation \[ \db \sigma_1 = \tau_1 \]
has a solution $\sigma_1$.

Assume that the equation \eqref{holt-2} is solved for all $k \leq
l$. Based on the assumption $h^{p,q+1}_{\db}(X_0) =0$, the equation \[ \db
\sigma_{l+1} = \tau_{l+1} \] will have a solution $\sigma_{l+1}$,
after we verify
\[ \db \tau_{l+1}=0. \]
Hence, we check it as follows, by use of the calculation
\eqref{dbar-closed-c}, which implies that \begin{align*} \db\eta_{l+1}
&= \p \lk -\frac{1}{2} \sum_{i=1}^{l+1} \sum_{j=1}^{i-1} \varphi_j
\lrcorner (\varphi_{i-j} \lrcorner \p \sigma_{l+1-i}) +
\sum_{i=1}^{l+1} \sum_{j=1}^{l+1-i} \varphi_i \lrcorner (\varphi_j
\lrcorner \p \sigma_{l+1-i-j}) \rk\\
&=  \p \lk \frac{1}{2} \sum_{i=1}^{l+1} \sum_{j=1}^{l+1-i} \varphi_i
\lrcorner (\varphi_j \lrcorner \p \sigma_{l+1-i-j}) \rk,
\end{align*}
in this case. Then it follows that \begin{align*} \db \tau_{l+1} &=
\db \eta_{l+1} + \sum_{i=1}^{l+1} \db \varphi_i \lrcorner \p
\sigma_{l+1-i} - \sum_{i=1}^{l+1} \varphi_i \lrcorner \p \db
\sigma_{l+1-i} \\
&= \p \lk \frac{1}{2} \sum_{i=1}^{l+1} \sum_{j=1}^{l+1-i} \varphi_i
\lrcorner (\varphi_j \lrcorner \p \sigma_{l+1-i-j}) \rk +
\sum_{i=1}^{l+1} \sum_{j=1}^{i-1} \frac{1}{2}
[\varphi_j,\varphi_{i-j}] \lrcorner \p \sigma_{l+1-i} \\
& \quad + \sum_{i=1}^{l+1} \varphi_i \lrcorner \p \lk \p \bigg(
\sum_{j=1}^{l+1-i} \varphi_j \lrcorner \sigma_{l+1-i-j}\bigg) -
\sum_{j=1}^{l+1-i} \varphi_j \lrcorner \p \sigma_{l+1-i-j} \rk \\
&= \p \lk \frac{1}{2} \sum_{i=1}^{l+1} \sum_{j=1}^{l+1-i} \varphi_i
\lrcorner (\varphi_j \lrcorner \p \sigma_{l+1-i-j}) \rk +
\sum_{i=1}^{l+1} \sum_{j=1}^{i-1} \frac{1}{2} \Bigg( - \p
\Big(\varphi_j \lrcorner (\varphi_{i-j} \lrcorner \p
\sigma_{l+1-i})\Big) \\
&\quad + \varphi_j \lrcorner \p \Big( \varphi_{i-j} \lrcorner \p
\sigma_{l+1-i} \Big) + \varphi_{i-j} \lrcorner \p \Big( \varphi_{j}
\lrcorner \p \sigma_{l+1-i} \Big) \Bigg) -
\sum_{i=1}^{l+1}\sum_{j=1}^{l+1-i} \varphi_i \lrcorner \p \Bigg(
\varphi_j \lrcorner \p \sigma_{l+1-i-j}\Bigg) \\
&=0.
\end{align*}
Therefore, we can also resolve the equation \eqref{holt-2} and
extend $\db$-closed $(p,q)$-forms un-obstructed under the
assumption that $h^{p,q+1}_{\db}(X_0)=0$.
\end{proof}

\subsection{Proofs of the invariance of
Hodge numbers $h^{p,0}(X_t)$, $h^{0,q}(X_t)$: special cases}\label{p00q-section}
This subsection is devoted to the deformation invariance of
$({p,0})$ and $(0,q)$-Hodge numbers as two special cases of Theorem
\ref{inv-pq}.

Theorem \ref{inv-p0} can be restated by use of Notation \ref{notaion-class}
as follows:
\begin{theorem}[]\label{invp0}
If the central fiber $X_0$ satisfies both $\mathbb{S}^{p+1,0}$ and $\mathbb{S}^{p,1}$,
then $h^{p,0}_{\db_t}(X_t)$ are independent of $t$.
\end{theorem}

According to the philosophy described in Section \ref{cf}, 
Theorem \ref{invp0} amounts to:
\begin{proposition}\label{thmp0}
Assume that $X_0$ satisfies $\mathbb{S}^{p+1,0}$ and $\mathbb{S}^{p,1}$.
Then for any holomorphic $(p,0)$-form $\sigma_0$ on ${X_0}$, there
exits a power series
$$\sigma_t=\sigma_0+\sum_{k=1}^\infty t^k\sigma_k \in
A^{p,0}(X_0),$$ such that $\sigma_t$ varies smoothly on $t$ and $e^{i_{\varphi(t)}}(\sigma_t)\in
A^{p,0}(X_t)$ is holomorphic with respect to the holomorphic structure
on $X_t$.
\end{proposition}
\begin{proof}
With the notations in the proof of Proposition \ref{thmpq}, we just present
the construction of $\sigma_t$ since the regularization argument is quite similar.
Obviously, under the assumption $\mathbb{S}^{p+1,0}$ on $X_0$,
the holomorphic $(p,0)$-form $\sigma_0$ is actually $d$-closed.
By Proposition \ref{extension-in} and type-consideration, the desired
holomorphicity is equivalent to the resolution of the equation
\begin{equation}\label{p0-extension-equ}
([\p,i_{\varphi}]+\b{\p})(\1-\b{\varphi}\varphi)\Finv\sigma_t=([\p,i_{\varphi}]+\b{\p})\sigma_t=0.
\end{equation}
Let
$${\sigma}_t={\sigma}_0+\sum_{j=1}^\infty {\sigma}_{j}t^j$$
be a power series of $(p,0)$-forms on $X_0$.

We will also resolve \eqref{p0-extension-equ} by an iteration method.
It suffices to consider the system of equations
\begin{equation}\label{holt-1}
\begin{cases}\db \sigma_0=0,\\
\db\sigma_k = -\partial(\sum_{i=1}^k
\varphi_i\lrcorner\sigma_{k-i}),\qquad
\text{for each $k\geq 1$},\\
\p \sigma_k =0,\qquad \text{for each $k\geq 0$},
\end{cases}
\end{equation}
after the comparison of the coefficients of $t^k$.

As for the second equation of \eqref{holt-1}, we may also assume that, for $i=0,\cdots,k-1,$
$\widetilde\sigma_{i}$ with ${\partial}{\widetilde\sigma_i}=0$ has been resolved, and then check
$$\db\partial(\sum_{i=1}^k\varphi_i\lrcorner{\sigma_{k-i}})=0$$
as reasoned in \eqref{dbar-closed-c}.
The assumption $\mathbb{S}^{p,1}$ enables us to obtain a canonical solution
$${\sigma_k}=-\db^*\G_{\db}\partial\left(\sum_{i=1}^k\varphi_i\lrcorner{\sigma_{k-i}}\right).$$
Meanwhile, the third equation $\p\sigma_k=0$ holds, due to the assumption $\mathbb{S}^{p+1,0}$ and
the equality
$$\db\p\sigma_k=\p\partial(\sum_{i=1}^k
\varphi_i\lrcorner\sigma_{k-i})=0.$$
\end{proof}

\begin{corollary}[The case of $(p,q)=(1,0)$]\label{cor10}
If the central fiber $X_0$ satisfies both $\mathbb{S}^{2,0}$ and $\mathcal{S}^{1,1}$, then
$h^{1,0}_{\db_t}(X_t)$ are independent of $t$.
\end{corollary}
\begin{proof}
From Theorem \ref{invp0}, $h^{1,0}_{\db_t}(X_t)$ are independent of $t$ when $X_0$
satisfies $\mathbb{S}^{2,0}$ and $\mathbb{S}^{1,1}$. The condition
$\mathbb{S}^{1,1}$ can be replaced by a weaker one $\mathcal{S}^{1,1}$.

A close observation to \eqref{dbar-closed-c} and
the fact that $\sigma_i$ are all of the special type $(1,0)$ show that
\[ \begin{aligned}
\db(\sum_{i=1}^k \varphi_i\lrcorner\sigma_{k-i})
 &=\frac{1}{2}\sum_{i=1}^{l+1}\sum_{j=1}^{i-1}\bigg(-\partial\big(\varphi_{j}\lrcorner(\varphi_{i-j}\lrcorner\sigma_{l+1-i})\big)
 -\varphi_{j}\lrcorner(\varphi_{i-j}\lrcorner\partial\sigma_{l+1-i}) \\
  & \quad +\varphi_j\lrcorner\partial(\varphi_{i-j}\lrcorner\sigma_{l+1-i})
  +\varphi_{i-j}\lrcorner\partial(\varphi_j\lrcorner\sigma_{l+1-i})\bigg)
  -\sum_{i=1}^{l+1}\varphi_{i}\lrcorner\partial \left.
 \bigg(\sum_{j=1}^{l+1-i}\varphi_j\lrcorner\sigma_{l+1-i-j} \right)\\
 &=\sum_{1\leq j<i\leq l+1}\varphi_j\lrcorner\partial(\varphi_{i-j}\lrcorner\sigma_{l+1-i})
 -\sum_{i=1}^{l+1}\sum_{j=1}^{l+1-i}\varphi_{i}\lrcorner\partial(\varphi_j\lrcorner\sigma_{l+1-i-j})\\
 &=0\\
\end{aligned} \] for $k \geq 1$, by the induction method.
Hence, it suffices to use the condition $\mathcal{S}^{1,1}$ to
solve the second one of the system \eqref{holt-1} of equations.
\end{proof}

Actually, by Example \ref{qn-qn-1}, we can get a more general result
that the deformation invariance for $h^{p,0}$ of an $n$-dimensional
compact complex manifold $X$ can be obtained from the one for
$h^{p,1}$.

\begin{corollary}[The case $(p,q)=(n-1,0)\ \text{or}\ (n,0)$]\label{p0-reduced}
For $p=n-1$ or $n$, the condition $\mathbb{S}^{p,1}$ on $X_0$ assures the deformation
invariance of $h^{p,0}_{\db_t}(X_t)$.
\end{corollary}
\begin{proof} Analogously to Kodaira \cite[Theorem 1]{K} or \cite[Lemma 1.2]{N}
that any holomorphic $(n-1)$-form on an $n$-dimensional
compact complex manifold is $d$-closed, one is able to prove that any $d$-closed $\p$-exact $(n,0)$-form is
zero. Hence, any compact complex manifold $X_0$ satisfies $\mathbb{S}^{n,0}$ and thus this corollary is proved by Theorem
\ref{invp0}.
\end{proof}

One restates Theorem \ref{inv-0q} by use of Notation \ref{notaion-class}:
\begin{theorem}[]\label{inv0q}
If the central fiber $X_0$ satisfies $\mathcal{B}^{1,q}$ with the deformation invariance of $h^{0,q-1}_{\db_t}(X_t)$ established,
then $h^{0,q}_{\db_t}(X_t)$ are independent of $t$.
\end{theorem}

For Theorem \ref{inv0q}, it suffices to prove:
\begin{proposition}\label{thm0q}
Assume that $X_0$ satisfies $\mathcal{B}^{1,q}$. Then for each Dolbeault class in $H^{0,q}_{\db}(X_0)$
with the unique canonical $d$-closed representative $\sigma_0$ given as Lemma \ref{lemma1},
there exists $\sigma_t\in A^{0,q}(X_0)$ varying smoothly on $t$ and
$e^{i_{\overline{\varphi}}}(\sigma_t)\in A^{0,q}(X_t)$ is
$\db_t$-closed with respect to the holomorphic structure on $X_t$.
\end{proposition}

\begin{proof}
We just need to present
the construction of $\sigma_t$. By Proposition \ref{extension-in} and type-consideration, the desired
$\db_t$-closedness is equivalent to the resolution of the equation
$$
([\p,i_{\varphi}]+\b{\p})(\1-\b{\varphi}\varphi)\Finv\sigma_t
=\b{\p}((\1-\b{\varphi}\varphi)\Finv\sigma_t)
-\varphi\lc\p((\1-\b{\varphi}\varphi)\Finv\sigma_t)
=0.
$$
Therefore, it suffices to take $\sigma_t=(\1-\b{\varphi}\varphi)^{-1}\Finv\sigma_0$.
\end{proof}

\begin{corollary}[]\label{}
All the Hodge numbers on a compact complex surface $X$ are deformation
invariant.
\end{corollary}
\begin{proof}
From these standard results in
\cite[Section $IV.2$]{BHPV}, the $\p\db$-lemma holds on
$X$ for weight $2$, and thus the Hodge numbers $h^{1,0}(X_t)$, $h^{0,1}(X_t)$ of the
small deformation of $X$ is independent of $t$ by Corollary \ref{cor10} and Remark \ref{01-sgg}, respectively.
The deformation invariance of the remaining Hodge numbers is
obtained by Serre duality and the deformation invariance of the
Euler-Poincar\'e characteristic (see, for example, \cite[Theorem
14]{KS}).
\end{proof}

\section{The Gauduchon cone $\mathcal{G}_X$}\label{Gau} In this
section we will study the Gauduchon cone and its relation with the
balanced one, to explore the deformation properties of an
\textbf{sGG} manifold proposed by Popovici \cite{P1}.

Let us first recall some notations. Aeppli cohomology groups
$H^{p,q}_{\mathrm{A}}(X,\mathbb{C})$ and Bott-Chern cohomology
groups $H^{p,q}_{\mathrm{BC}}(X,\mathbb{C})$ are defined on any
compact complex manifold $X$, even on non-compact ones (cf. for
instance, \cite{A,P1}). Accordingly, the  \emph{real Aeppli
cohomology group} $H^{p,p}_{\mathrm{A}}(X,\mathbb{R})$ is defined by
\[ H^{p,p}_{\mathrm{A}}(X,\mathbb{R}):= \frac{\Big\{ \text{$\p\db$-closed smooth real $(p,p)$-forms} \Big\}}
{ \Big\{ \p \eta + \overline{\p \eta}\ \big|\ \text{$\eta$ is a
smooth complex valued $(p-1,p)$-forms} \Big\}}. \] And the
\emph{real Bott-Chern cohomology group}
$H^{p,p}_{\mathrm{BC}}(X,\mathbb{R})$ is given by
\[ H^{p,p}_{\mathrm{BC}}(X,\mathbb{R}) := \frac{\Big\{
d\textrm{-closed}~\textrm{smooth}~\textrm{real}~(p,p)
\textrm{-forms} \Big\}} { \Big\{ \sqrt{-1} \p \db \eta\ \big|\ \eta\
\textrm{is a smooth}~\textrm{real}~(p-1,p-1)\textrm{-forms} \Big\}}.
\] Also, similar types of currents can represent Aeppli classes or Bott-Chern
ones. By \cite[Lemme 2.5]{Sc} or \cite[Theorem 2.1.(iii)]{P1}, a
canonical non-degenerate duality between
$H^{n-p,n-p}_{\mathrm{A}}(X,\mathbb{C})$ and
$H^{p,p}_{\mathrm{BC}}(X,\mathbb{C})$ is given by
\[ \begin{array}{ccc}
H^{n-p,n-p}_{\mathrm{A}}(X,\mathbb{C}) \times
H^{p,p}_{\mathrm{BC}}(X,\mathbb{C}) & \longrightarrow & \mathbb{C}\\
\left(\big[\Omega\big]_{\mathrm{A}} ,
\big[\omega\big]_{\mathrm{BC}}\right) & \longmapsto & \int_{X}
\Omega \w \omega.
\end{array} \]
The pairing $(\bullet,\bullet)$, restricted to real cohomology
groups, also becomes the duality between the two corresponding
groups.

The  \emph{Gauduchon cone} $\mathcal{G}_X$ is defined by \[
\mathcal{G}_X = \left\{\big[\Omega\big]_{\mathrm{A}} \in
H^{n-1,n-1}_{\mathrm{A}}(X,\mathbb{R})\ \Big|\ \Omega\ \textrm{is}\
\textrm{a}\ \p\db\textrm{-closed}\ \textrm{positive}\
(n-1,n-1)\textrm{-form} \right\},\] where $\omega =
\Omega^{\frac{1}{n-1}}$ is called a  \emph{Gauduchon metric}. It is
a known fact in linear algebra, by Michelsohn \cite[the part after
Lemma 4.8]{M}, that for every positive $(n-1,n-1)$-form $\Gamma$ on
$X$, there exists a unique positive $(1,1)$-form $\gamma$ such that
$\gamma^{n-1}=\Gamma$. Thus, the symbol $\Omega^{\frac{1}{n-1}}$
makes sense. Gauduchon metric exists on any compact complex
manifold, thanks to Gauduchon's work \cite{G}. Hence, the Gauduchon
cone $\mathcal{G}_X$ is never empty. Similarly, the \emph{K\"{a}hler
cone} $\mathcal{K}_X$ and the  \emph{balanced cone} $\mathcal{B}_X$
are defined as
\[ \mathcal{K}_X =
\left\{\big[\omega\big]_{\mathrm{BC}} \in
H^{1,1}_{\mathrm{BC}}(X,\mathbb{R})\ \Big|\ \omega\ \textrm{is}\
\textrm{a}\ d\textrm{-closed}\ \textrm{positive}\
(1,1)\textrm{-form} \right\}, \]
\[ \mathcal{B}_X =
\left\{\big[\Omega\big]_{\mathrm{BC}} \in
H^{n-1,n-1}_{\mathrm{BC}}(X,\mathbb{R})\ \Big|\ \Omega\ \textrm{is}\
\textrm{a}\ d\textrm{-closed}\ \textrm{positive}\
(n-1,n-1)\textrm{-form} \right\}, \] where $\Omega^{\frac{1}{n-1}}$
is called a  \emph{balanced metric}. And the three cones are open
convex cones (cf. \cite[Observation 5.2]{P1} for the Gauduchon
cone).

The numerically effective (shortly  \emph{nef}) cone, can be defined
as
\[ \left\{\big[\omega\big]_{\mathrm{BC}} \in
H^{1,1}_{\mathrm{BC}}(X,\mathbb{R}) \Big|\forall \epsilon>0,
\exists\ \textrm{a smooth real}\ (1,1)\textrm{-form}\
\alpha_{\epsilon}\!\in\!\big[\omega\big]_{\mathrm{BC}},\,
\textrm{such that}\ \alpha_{\epsilon} \geq - \epsilon
\widetilde{\omega} \right\},\] where $\widetilde{\omega}$ is a fixed
Hermitian metric on the compact complex manifold $X$. And the nef
cone is a closed convex cone by \cite[Proposition 6.1]{D1}. When $X$
is K\"{a}hler, the nef cone is the closure of the K\"{a}hler cone
$\mathcal{K}_X$. Thus, we will use the symbol
$\overline{\mathcal{K}}_X$ for the nef cone in any situation.
Similar definitions adapt to $\overline{\mathcal{B}}_X$ and
$\overline{\mathcal{G}}_X$, which are also closed convex cones.
There are many studies, such as \cite{D1,DP,DPS,BDPP,WYZ,FX,P1,PU}
on these cones and their relations.

\begin{definition}\rm
Degenerate cones.

We say that the  \emph{Gauduchon cone $\mathcal{G}_X$ degenerates}
when $\mathcal{G}_X = H^{n-1,n-1}_{\mathrm{A}}(X,\mathbb{R})$, which
comes from \cite[Section 5]{P1}. Similarly, the  \emph{balanced cone
$\mathcal{B}_X$ degenerates} if the equality $\mathcal{B}_X =
H^{n-1,n-1}_{\mathrm{BC}}(X,\mathbb{R})$ holds.
\end{definition}

\subsection{The K\"{a}hler case of $\mathcal{G}_X$}\label{Gcone-Kahler}
We will consider various cones on K\"{a}hler manifolds at first.
Thus, let $X$ be a compact K\"{a}hler manifold.

\begin{lemma}\label{Gcone}
The Gauduchon cone $\mathcal{G}_X$ does not  degenerate on the compact
K\"{a}hler manifold $X$. Moreover, $\mathcal{G}_{X}$ lies in one
open half semi-space determined by some linear subspace of
codimension one in $H^{n-1,n-1}_{\mathrm{A}}(X,\mathbb{R})$.
\end{lemma}

\begin{proof}
$X$ carries a K\"{a}hler metric $\omega_{X}$. Then
$\big[\omega_{X}\big]_{\mathrm{BC}}$ lives in the K\"{a}hler cone
$\mathcal{K}_X$, which can not  be the zero class of
$H^{1,1}_{\mathrm{BC}}(X, \mathbb{R})$. This implies that
\[\dim_{\mathbb{R}}H^{n-1,n-1}_{\mathrm{A}}(X,\mathbb{R}) =
\dim_{\mathbb{R}}H^{1,1}_{\mathrm{BC}}(X,\mathbb{R}) \geq 1.\] Thus,
the Gauduchon cone $\mathcal{G}_X$ is a non-empty open cone in a
vector space with the dimension at least one, which implies that
$\mathcal{G}_{X}$ must contain a non-zero class.

Meanwhile, the Gauduchon $\mathcal{G}_{X}$ can not  degenerate. If
$\mathcal{G}_X$ degenerates, i.e., $0\in \mathcal{G}_X =
H_{\mathrm{A}}^{n-1,n-1}(X,\mathbb{R})$, $X$ carries a Hermitian
metric $\widetilde{\omega}$ such that $\widetilde{\omega}^{n-1}$ is
the type of $\db \psi + \p \overline{\psi}$, where $\psi$ is a
smooth $(n-1,n-2)$-form on $X$. It is easy to check that
$\widetilde{\omega}^{n-1} \w \omega_{X}$ is $d$-exact but $\int_{X}
\widetilde{\omega}^{n-1} \w \omega_{X} > 0$, where a contradiction
emerges. As an easy consequence of this, the Gauduchon cone
$\mathcal{G}_X$ can not  contain the origin of
$H_{\mathrm{A}}^{n-1,n-1}(X,\mathbb{R})$.

It is easy to see that the K\"{a}hler class $\big[ \omega_{X}
\big]_{\mathrm{BC}}$ determines one open half semi-space
$\mathbf{H}^{+}_{\omega_{X}}$ in
$H^{n-1,n-1}_{\mathrm{A}}(X,\mathbb{R})$ given by
\[ \mathbf{H}^{+}_{\omega_{X}} = \left\{ \big[\Omega\big]_{\mathrm{A}} \in
H^{n-1,n-1}_{\mathrm{A}}(X,\mathbb{R})\ \Big|\ \int_{X} \Omega \w
\omega_{X} > 0 \right\}, \] which is clearly cut out by the linear
subspace of codimension one
\[\mathbf{H}_{\omega_X} =\left\{\big[\Omega\big]_{\mathrm{A}} \in
H^{n-1,n-1}_{\mathrm{A}}(X,\mathbb{R})\ \Big|\ \int_{X} \Omega \w
\omega_{X} = 0 \right\}. \] And the Gauduchon cone $\mathcal{G}_{X}$
obviously lies in $\mathbf{H}^{+}_{\omega_{X}}$. Hence the lemma is
proved.
\end{proof}

\begin{remark}\rm
It is well known that neither the K\"{a}hler cone $\mathcal{K}_X$
nor the balanced cone $\mathcal{B}_X$ degenerates on the K\"{a}hler
manifold $X$ .
\end{remark}

It is known that the quotient topology of Bott-Chern groups induced
by the Fr\'{e}chet topology of smooth forms or the weak topology of
currents is Hausdorff (cf. \cite[the part before Definition
1.3]{D1}). And every Hausdorff finite-dimensional topological real
vector space is isomorphic to $\mathbb{R}^n$ with the Euclidean
topology. Then it is harmless to fix an inner product $\langle
\bullet, \bullet \rangle$ on the real vector space
$H^{1,1}_{\mathrm{BC}}(X,\mathbb{R})$, which induces the given
topology on $H^{1,1}_{\mathrm{BC}}(X,\mathbb{R})$. The space
$H^{n-1,n-1}_{\mathrm{A}}(X,\mathbb{R})$ can be viewed as the vector
space of continuous linear functionals on
$\Big(H^{1,1}_{\mathrm{BC}}(X,\mathbb{R}),\langle \bullet, \bullet
\rangle\Big)$. By the finite-dimensional case of Riesz
representation theorem, there is a canonical isomorphism from
$H^{n-1,n-1}_{\mathrm{A}}(X,\mathbb{R})$ to
$H^{1,1}_{\mathrm{BC}}(X,\mathbb{R})$ with $\big[ \Omega
\big]_{\mathrm{A}}$ to $\big[ \omega_{\Omega}\big]_{\mathrm{BC}}$.
That is, for any $\big[\Omega\big]_{\mathrm{A}} \in
H^{n-1,n-1}_{\mathrm{A}}(X,\mathbb{R})$, there exists a unique
$\big[ \omega_{\Omega}\big]_{\mathrm{BC}} \in
H^{1,1}_{\mathrm{BC}}(X,\mathbb{R})$, such that
\[ \left(\big[\Omega\big]_{\mathrm{A}} ,
\big[\omega\big]_{\mathrm{BC}}\right) = \left\langle \big[ \omega
\big]_{\mathrm{BC}}, \big[\omega_{\Omega}\big]_{\mathrm{BC}}
\right\rangle \] for any $\big[ \omega \big]_{\mathrm{BC}} \in
H^{1,1}_{\mathrm{BC}}(X,\mathbb{R})$. Thus, this isomorphism enables
us to define the dual inner product on
$H^{n-1,n-1}_{\mathrm{A}}(X,\mathbb{R})$ by the equality \[ \left
\langle \big[\Omega_1
\big]_{\mathrm{A}},\big[\Omega_2\big]_{\mathrm{A}} \right \rangle :=
\left \langle \big[\omega_{\Omega_1} \big]_{\mathrm{BC}},\big[
\omega_{\Omega_2}\big]_{\mathrm{BC}} \right \rangle. \] Let $\left
\{\big[ \omega_i\big]_{\mathrm{BC}}\right\}_{i=1}^m$ be an
orthonormal basis of $H^{1,1}_{\mathrm{BC}}(X,\mathbb{R})$. Then,
$\left\{\big[ \Omega_{\omega_i}\big]_{\mathrm{A}}\right\}_{i=1}^m$,
the inverse image of $\left \{\big[
\omega_i\big]_{\mathrm{BC}}\right\}_{i=1}^m$ under the above
canonical isomorphism, is also an orthonormal one of
$H^{n-1,n-1}_{\mathrm{A}}(X,\mathbb{R})$ under the dual metric. And
$\left \{\big[ \omega_i\big]_{\mathrm{BC}}, \big[
\Omega_{\omega_i}\big]_{\mathrm{A}} \right\}_{i=1}^m$ become dual
bases with respect to $(\bullet,\bullet)$.

\begin{definition}\label{cone}\rm
The open circular cone $\mathcal{C}(v,\theta)$.

Let $\big(V_{\mathbb{R}}, \langle \bullet, \bullet \rangle \big)$ be
a real vector space $V_{\mathbb{R}}$, which equips with an inner
product $\langle \bullet, \bullet \rangle$. Denote the induced norm
by $\| \bullet \|$. The  \emph{open circular cone}
$\mathcal{C}(v,\theta)$
 is determined by a non-zero vector $v$ in $V_{\mathbb{R}}$ and an
 angle $\theta \in \big[0,\frac{\pi}{2} \big]$, given by
\[ \mathcal{C}(v,\theta) = \left\{  w \in V_{\mathbb{R}} \setminus
0\ \Big|\ \frac{\langle w, v \rangle}{\|w\|\|v\|} > \cos \theta
\right\}. \] And $2\theta$ is called the  \emph{cone angle}.
 It is clear that the cone $\mathcal{C}(v,\theta)$
does not  change if $v$ is replaced by any vector in $\mathbb{R}^{
\scriptsize
> 0}v$.
\end{definition}

As stated in the proof of Lemma \ref{Gcone}, the Gauduchon cone
$\mathcal{G}_{X}$ must contain a non-zero class. Let us  fix a nonzero
class $\big[ \Omega_0 \big]_{\mathrm{A}} \in \mathcal{G}_{X}$.

\begin{proposition}\label{Gcone_2} On a compact K\"{a}hler manifold
$X$, there exists a small angle $\tilde{\theta} \in \big(0,
\frac{\pi}{2}\big)$ such that \[ \mathcal{C}\Big( \big[ \Omega_0
\big]_{\mathrm{A}},\tilde{\theta} \Big) \subseteq \mathcal{G}_X
\subseteq \mathcal{C}\Big( \big[\Omega_{\omega_X}\big]_{\mathrm{A}},
\frac{\pi}{2}- \tilde{\theta} \Big),\] where the class
$\big[\Omega_{\omega_X}\big]_{\mathrm{A}}$ in
$H^{n-1,n-1}_{\mathrm{A}}(X,\mathbb{R})$ denotes the inverse image
of the K\"{a}hler class $\big[\omega_X\big]_{\mathrm{BC}}$ under the
canonical isomorphism discussed before Definition \ref{cone}.
\end{proposition}

\begin{proof}
Since $\big[ \Omega_0\big]_{\mathrm{A}}$ is a non-zero class of
$\mathcal{G}_X$, there exists a neighborhood of $\big[
\Omega_0\big]_{\mathrm{A}}$, belonging to $\mathcal{G}_X$, namely,
\[\Big\{ \big[ \Omega \big]_{\mathrm{A}} \in
H^{n-1,n-1}_{\mathrm{A}}(X,\mathbb{R}) \Big| \left\|
\big[\Omega\big]_{\mathrm{A}} - \big[\Omega_0\big]_{\mathrm{A}}
\right\| < \epsilon \Big\} \subseteq \mathcal{G}_X \] for some
$\epsilon>0$. Since $\mathcal{G}_X$ is an open convex cone, the
inclusion follows
\[ \mathcal{C}\Big( \big[ \Omega_0
\big]_{\mathrm{A}}, \arcsin
\frac{\epsilon}{\big\|[\Omega_0]_{\mathrm{A}}\big\|} \Big) \subseteq
\mathcal{G}_X. \] Similarly, there exists $\tilde{\epsilon}>0$, such
that \[ \mathcal{C}\Big( \big[ \omega_X \big]_{\mathrm{BC}}, \arcsin
\frac{\tilde{\epsilon}}{\big\|[\omega_X]_{\mathrm{BC}}\big\|} \Big)
\subseteq \mathcal{K}_X. \] It is easy to see that \[ \mathcal{G}_X
\subseteq \bigcap_{ [\omega]_{\mathrm{BC}} \in
\mathcal{C}\big([\omega_X]_{\mathrm{BC}}, \theta_0 \big)}
\mathbf{H}^{+}_{\omega}, \] where $\theta_0$ can be chosen as
$\arcsin
\frac{\tilde{\epsilon}}{\big\|[\omega_X]_{\mathrm{BC}}\big\|}$. From
the discussion before Definition \ref{cone}, we know that
\[\bigcap_{ [\omega]_{\mathrm{BC}} \in
\mathcal{C}\big([\omega_X]_{\mathrm{BC}}, \theta_0 \big)}
\mathbf{H}^{+}_{\omega} = \mathcal{C}\Big( \big[ \Omega_{\omega_X}
\big]_{\mathrm{A}}, \frac{\pi}{2} - \theta_0 \Big).\] Let the angle
$\tilde{\theta}$ be
\[\min \left(\arcsin
\frac{\epsilon}{\big\|[\Omega_0]_{\mathrm{A}}\big\|},\
\arcsin\frac{\tilde{\epsilon}}{\big\|[\omega_X]_{\mathrm{BC}}\big\|}\right).\]
\end{proof}
As in \cite[Section $5$]{P1}, if the finite-dimensional vector space
$H^{n-1,n-1}_{\mathrm{A}}(X,\mathbb{R})$ of a compact complex
manifold $X$ is endowed with the unique norm-induced topology, the
\emph{closure of the Gauduchon cone} in
$H^{n-1,n-1}_{\mathrm{A}}(X,\mathbb{R})$ is defined by
\begin{equation}\label{clo-G}
\overline{\mathcal{G}}_{X}= \left\{\alpha\in
H^{n-1,n-1}_{\mathrm{A}}(X,\mathbb{R})\ \big|\ \text{$\forall
\epsilon>0$, $\exists$ smooth $\Omega_\epsilon\in \alpha$, such that
$\Omega_\epsilon\geq -\epsilon\Omega$}\right\},
\end{equation}
where $\Omega>0$ is a fixed smooth $(n-1,n-1)$-form on $X$ with
$\p\db\Omega=0$. This cone is convex and closed, which is shown in
\cite[Proposition 6.1.(i)]{D1}.
\begin{corollary}
The closure of the Gauduchon cone $\overline{\mathcal{G}}_{X}$ on
the K\"{a}hler manifold $X$ must lie in some closed circular cone
with the cone angle smaller than $\pi$, for example the closure of
$\mathcal{C}\Big( \big[\Omega_{\omega_X}\big]_{\mathrm{A}},
\frac{\pi}{2}- \tilde{\theta} \Big)$.
\end{corollary}

In a similar manner, we can also show that the K\"{a}hler cone
$\mathcal{K}_{X}$ on a K\"{a}hler manifold $X$ must lie in some open
circular cone with the cone angle smaller than $\pi$ in
$H^{1,1}_{\mathrm{BC}}(X,\mathbb{R})$.

The following definition is inspired by \cite[Observation 5.7 and
Question 5.9]{P1}.
\begin{definition}\label{dual-cone}\rm
$\big( \mathcal{A} \big)^{\mathrm{v}_{\!o}}$ and $\big( \mathcal{A}
\big)^{\mathrm{v}_{\!c}}$

Let $\mathcal{A}$ be a convex cone in a finite-dimensional vector
space $W_{\mathbb{R}}$, whose dual vector space is denoted by
$W_{\mathbb{R}}^{\mathrm{v}}$.
\begin{enumerate}
\item $\big( \mathcal{A} \big)^{\mathrm{v}_{\!o}}$ denotes the set of
linear functions in $W_{\mathbb{R}}^{\mathrm{v}}$, evaluating
positively on $\mathcal{A}$;
\item $\big( \mathcal{A}
\big)^{\mathrm{v}_{\!c}}$ denotes the set of linear functionals in
$W_{\mathbb{R}}^{\mathrm{v}}$, evaluating non-negatively on
$\mathcal{A}$.
\end{enumerate}
Let $\mathcal{P}$ and $\mathcal{Q}$ be two closed convex cones in
the $W_{\mathbb{R}}$ and $W_{\mathbb{R}}^{\mathrm{v}}$,
respectively. We say that $\mathcal{P}$ and $\mathcal{Q}$ are
\emph{dual cones}, if $\mathcal{P} =
\big(\mathcal{Q}\big)^{\mathrm{v}_c}$ and $\mathcal{Q} =
\big(\mathcal{P}\big)^{\mathrm{v}_c}$.
\end{definition}

The  \emph{pseudo-effective cone} $\mathcal{E}_X$, the set of
classes in $H^{1,1}_{\mathrm{BC}}(X,\mathbb{R})$ represented by
$d$-closed positive $(1,1)$-currents, is a closed convex cone when
$X$ is any compact complex manifold (cf. \cite[Proposition
6.1]{D1}). The \emph{big cone} $\mathcal{E}_X^{\circ}$, an open
convex cone in $H^{1,1}_{\mathrm{BC}}(X,\mathbb{R})$, is defined to
be the interior of the pseudo-effective cone $\mathcal{E}_X$ when
$X$ is K\"{a}hler, in which classes are represented by K\"ahler
$(1,1)$-currents (cf. \cite[Definition 1.6]{DP}).
\begin{theorem}\label{open_dual}
For a compact K\"{a}hler manifold $X$,
$$\overline{\mathcal{G}}_{X} \setminus \big[0\big]_{\mathrm{A}}
\subseteq\big( \mathcal{E}_X^{\circ} \big)^{\mathrm{v}_{\!o}}$$ and
thus $\mathcal{G}_{X} \subsetneqq
\big(\mathcal{E}_X^{\circ}\big)^{\mathrm{v}_o}$.
\end{theorem}
\begin{proof}
It is clear that each class in $\overline{\mathcal{G}}_{X} \setminus
\big[0\big]_{\mathrm{A}}$ evaluates non-negatively on the big cone
$\mathcal{E}_X^{\circ}$. Suppose that some class $\big[ \Omega
\big]_{\mathrm{A}}$ in $\overline{\mathcal{G}}_{X} \setminus
\big[0\big]_{\mathrm{A}}$ does not  evaluate positively on
$\mathcal{E}_X^{\circ}$, i.e., there exists a class
$\big[T(\Omega)\big]_{\mathrm{BC}} \in \mathcal{E}_X^{\circ}$, with
$T(\Omega)$ a K\"{a}hler current, such that \[\int_{X} \Omega \w
T(\Omega) =0. \] Then note that the big cone $\mathcal{E}_X^{\circ}$
actually lies in the closed half semi-space $\mathbf{H}^{+}_{\Omega}
\bigcup \mathbf{H}_{\Omega}$ of
$H^{1,1}_{\mathrm{BC}}(X,\mathbb{R})$ with $\big[ T(\Omega)
\big]_{\mathrm{BC}}$ attached to the linear subspace
$\mathbf{H}_{\Omega}$. But a small neighborhood of $\big[ T(\Omega)
\big]_{\mathrm{BC}}$ will run out of the closed half semi-space
$\mathbf{H}^{+}_{\Omega} \bigcup \mathbf{H}_{\Omega}$ into the other
open half $\mathbf{H}^{-}_{\Omega}$. Meanwhile, the neighborhood is
still contained in $\mathcal{E}_X^{\circ}$, since the big cone
$\mathcal{E}_X^{\circ}$ is an open convex cone. This contradiction
tells us that each class in $\overline{\mathcal{G}}_{X} \setminus
\big[0\big]_{\mathrm{A}}$ evaluates positively on
$\mathcal{E}_X^{\circ}$. Hence, we have \[
\overline{\mathcal{G}}_{X} \setminus \big[0\big]_{\mathrm{A}}
\subseteq \big(\mathcal{E}_X^{\circ}\big)^{\mathrm{v}_{\!o}}. \]

It is clear that $\mathcal{G}_{X} \subseteq
\big(\mathcal{E}_X^{\circ}\big)^{\mathrm{v}_o}$. Now suppose that
$\big(\mathcal{E}_X^{\circ}\big)^{\mathrm{v}_{\!o}} =
\mathcal{G}_X$. Then
\[ \overline{\mathcal{G}}_{X} \setminus
\big[0\big]_{\mathrm{A}} \subseteq
\big(\mathcal{E}_X^{\circ}\big)^{\mathrm{v}_{\!o}} = \mathcal{G}_X\]
follows directly, which is equivalent to the equality \[
\overline{\mathcal{G}}_X = \mathcal{G}_X \bigcup \big[ 0
\big]_{\mathrm{A}}. \] Hence, the hyperplane
$\mathbf{H}_{\omega_X}(1)$ in $H^{n-1,n-1}_{\mathrm{A}}
(X,\mathbb{R})$, defined by
\[ \mathbf{H}_{\omega_X}(1) = \left\{ \big[\Omega\big]_{\mathrm{A}} \in
H^{n-1,n-1}_{\mathrm{A}}(X,\mathbb{R})\ \Big|\ \int_{X} \Omega \w
\omega_X =1 \right\}, \] has the same intersection with
$\mathcal{G}_X$ and $ \overline{\mathcal{G}}_{X}$. This implies that
the intersection $\mathcal{G}_X \bigcap \mathbf{H}_{\omega_X}(1)$ is
both open and closed on the hyperplane $\mathbf{H}_{\omega_X}(1)$,
which is clearly connected. Then, we get $\mathcal{G}_X \bigcap
\mathbf{H}_{\omega_X}(1) = \mathbf{H}_{\omega_X}(1)$, which leads to
the inclusion
\[ \mathbf{H}_{\omega_X}(1) \subseteq \mathcal{G}_X. \] Hence, the open
half semi-space $\mathbf{H}^{+}_{\omega_X}$ is contained in the
Gauduchon cone $\mathcal{G}_X$. However, from the proof of
Proposition \ref{Gcone_2}, we know that $\mathcal{G}_X$ actually
lies in $\mathcal{C}\Big( \big[ \Omega_{\omega_X}
\big]_{\mathrm{A}}, \frac{\pi}{2} - \tilde{\theta} \Big)$, which is
strictly contained in $\mathbf{H}^{+}_{\omega_X}$. Here is a
contradiction. So $\mathcal{G}_{X} \subsetneqq
\big(\mathcal{E}_X^{\circ}\big)^{\mathrm{v}_o}$.
\end{proof}
\begin{remark}\rm
It is shown that $\overline{\mathcal{G}}_{X} \setminus
\big[0\big]_{\mathrm{A}} =
\big(\mathcal{E}_X^{\circ}\big)^{\mathrm{v}_o}$ in Remark
\ref{open_dual_add}.
\end{remark}
\subsection{The relation between balanced cone $\mathcal{B}_X$
and Gauduchon cone $\mathcal{G}_X$}\label{relation}

There exists a pair of diagrams $(\mathrm{D},
\overline{\mathrm{D}})$ on a compact K\"{a}hler manifold $X$ as
follows, which is inspired by Fu-Xiao's work \cite{FX}. The diagrams
$\mathrm{D}$ reads
\[ \begin{tabular}{cc}
$\xymatrix{\mathcal{B}_X \ar[rr]^{\mathbb{J}} & & \mathcal{G}_X\\
& \mathcal{K}_X \ar[lu]^{\mathbb{I}} \ar[ru]_{\mathbb {K}} & }$ &
$\xymatrix{[ \omega^{n-1} ]_{\mathrm{BC}} \ar[rr]^{\mathbb{J}} & & \big[ \omega^{n-1} \big]_{\mathrm{A}}\\
& [\omega]_{\mathrm{BC}} \ar[lu]^{\mathbb{I}} \ar[ru]_{\mathbb {K}}
& }$, \end{tabular} \] and the diagram $\overline{\mathrm{D}}$
follows,
\[ \begin{tabular}{cc}
$\xymatrix{\overline{\mathcal{B}}_X \ar[rr]^{\overline{\mathbb{J}}} & & \overline{\mathcal{G}}_X\\
& \overline{\mathcal{K}}_X \ar[lu]^{\overline{\mathbb{I}}}
\ar[ru]_{\overline{\mathbb {K}}} & }$ &
$\xymatrix{[ \omega^{n-1} ]_{\mathrm{BC}} \ar[rr]^{\overline{\mathbb{J}}} & & \big[ \omega^{n-1} \big]_{\mathrm{A}}\\
& [\omega]_{\mathrm{BC}} \ar[lu]^{\overline{\mathbb{I}}}
\ar[ru]_{\overline{\mathbb {K}}} & }$. \end{tabular} \] The former
consists of three mappings among K\"{a}hler cone $\mathcal{K}_X$,
balanced cone $\mathcal{B}_X$ and Gauduchon cone $\mathcal{G}_X$.
And the latter is actually the extension of the former to the
closures of respective cones. It is easy to see that all the
mappings are well-defined and both diagrams are commutative. The
mappings $(\mathbb{I},\overline{\mathbb{I}})$,
$(\mathbb{J},\overline{\mathbb{J}})$ and
$(\mathbb{K},\overline{\mathbb{K}})$ are the restrictions of three
natural maps $\mathscr{I}$, $\mathscr{J}$ and $\mathscr{K}$,
respectively, which are independent of the K\"{a}hlerness of $X$.
The three mappings are given as follows:
\[ \begin{array}{cccc}
\mathscr{I}: & H^{1,1}_{\mathrm{BC}}(X,\mathbb{R})
& \> & H^{n-1,n-1}_{\mathrm{BC}}(X,\mathbb{R}) \\
& \big[\omega\big]_{\mathrm{BC}} & \mapsto &
 \big[\omega^{n-1}\big]_{\mathrm{BC}}, \\[5pt]
\mathscr{J}: & H^{n-1,n-1}_{\mathrm{BC}}(X,\mathbb{R})
& \> & H^{n-1,n-1}_{\mathrm{A}}(X,\mathbb{R}) \\
& \big[\Omega\big]_{\mathrm{BC}} & \mapsto &
 \big[\Omega\big]_{\mathrm{A}}, \\[5pt]
\mathscr{K}: & H^{1,1}_{\mathrm{BC}}(X,\mathbb{R})
& \> & H^{n-1,n-1}_{\mathrm{A}}(X,\mathbb{R}) \\
& \big[\omega\big]_{\mathrm{BC}} & \mapsto &
 \big[\omega^{n-1}\big]_{\mathrm{A}}. \\
\end{array} \]
Moreover, when $X$ is a complex manifold satisfying $\p\db$-lemma,
the mapping $\mathscr{J}$ is an isomorphism and thus the mappings
$(\mathbb{J},\overline{\mathbb{J}})$ are injective.

By \cite[Proposition 1.1 and Theorem 1.2]{FX}, the mapping
$\mathbb{I}$ is injective. Meanwhile, $\overline{\mathbb{I}}$,
restricted to the intersection of the nef cone and the big cone
$\overline{\mathcal{K}}_X \bigcap \mathcal{E}_X^{\circ}$, is also
injective. This is true, even when $X$ is in the  \emph{Fujiki
class} $\mathcal{C}$ (i.e., the class of compact complex manifolds
bimeromorphic K\"ahler manifolds), see \cite[Corollary 2.7]{FX}. The
existence of classes in $\overline{\mathbb{I}}(\p \mathcal{K}_X)
\bigcap \mathcal{B}_X$ implies that the mapping $\mathbb{I}$ is not
surjective. In fact, the class $\big[ \tilde{\omega}
\big]_{\mathrm{BC}} \in \p \mathcal{K}_X$, mapped into the balanced
cone $\mathcal{B}_X$, necessarily lies in the big cone
$\mathcal{E}^{\circ}_X$, by \cite[Theorem 1.3]{FX}. Thus, the class
$\overline{\mathbb{I}}(\big[ \tilde{\omega}\big]_{\mathrm{BC}})$ in
$\mathcal{B}_X$ can not  be mapped by a K\"{a}hler class, since
$\overline{\mathbb{I}}$ is injective on the intersection cone
$\overline{\mathcal{K}}_X \bigcap \mathcal{E}_X^{\circ}$. Besides,
Theorem 1.3 in \cite{FX} gives a precise description of
$\overline{\mathbb{I}}(\p \mathcal{K}_{\mathrm{NS}}) \bigcap
\mathcal{B}_X$ when $X$ is a projective Calabi-Yau manifold. The
cone $\mathcal{K}_{\mathrm{NS}}$ denotes the intersection
$\mathcal{K}_X \bigcap \mathrm{NS}_{\mathbb{R}}$, where
$\mathrm{NS}_{\mathbb{R}}$ is the real Neron-Severi group of $X$.

Recall that \cite[Lemma 3.3]{FX} states that a Bott-Chern class
$\big[ \Omega \big]_{\mathrm{BC}} \in
H^{n-1,n-1}_{\mathrm{BC}}(X,\mathbb{R})$ on a compact complex
manifold $X$, lives in the balanced cone $\mathcal{B}_X$ if and only
if \[ \int_X \Omega \w T
>0 ,\] for every non-zero $\p\db$-closed positive $(1,1)$-current
$T$. Similarly, one has:
\begin{lemma}\label{G_class}
Let $X$ be a compact complex manifold and $\Omega$ a real $\p
\db$-closed $(n-1,n-1)$-form on $X$. Then the class $\big[ \Omega
\big]_{\mathrm{A}}$ lives in $\mathcal{G}_X$ if and only if
\[ \int_X \Omega \w T > 0, \]
for every non-zero $d$-closed positive $(1,1)$-current $T$ on $X$.
\end{lemma}

\begin{proof}
We mainly follow the ideas of the proof of \cite[Lemma 3.3]{FX}. The
necessary part is quite obvious. As to the sufficient part, let
$\mathfrak{D'}_{\mathbb{R}}^{1,1}$ be the set of real
$(1,1)$-currents on $X$ with the weak topology. Fix a Hermitian
metric $\omega_{X}$ on $X$ and apply the Hahn-Banach separation
theorem, which originates from Sullivan's work \cite{S}. See also in
\cite[Lemma 3.3]{FX} and \cite[Propostion 5.4]{P1}.

Set \[ \mathfrak{D}_1 = \left\{ T \in
\mathfrak{D'}_{\mathbb{R}}^{1,1}\ \Big|\ \int_{X} \Omega \w T =0\
\textrm{and}\ dT=0 \right\}, \] \[ \mathfrak{D}_2 = \left\{ T \in
\mathfrak{D'}_{\mathbb{R}}^{1,1}\ \Big|\ \int_{X} \omega_X^{n-1} \w
T =1\ \textrm{and} \ T \geq 0 \right\}. \] It is easy to see that
$\mathfrak{D}_1$ is a closed linear subspace of the locally convex
space $\mathfrak{D'}_{\mathbb{R}}^{1,1}$, while $\mathfrak{D}_2$ is
a compact convex one in $\mathfrak{D'}_{\mathbb{R}}^{1,1}$. Since a
$d$-closed positive $(1,1)$-current $T$, satisfying $\int_X \Omega
\w T = 0$, has to be zero current from the assumption of the lemma,
$\mathfrak{D}_1 \bigcap \mathfrak{D}_2 = \emptyset$ by $\int_{X}
\omega_X^{n-1} \w T =1$. Then there exists a continuous linear
functional on $\mathfrak{D'}^{1,1}_{\mathbb{R}}$, denoted by
$\widetilde{\Omega}$, a real $(n-1,n-1)$-form, such that it vanishes
on $\mathfrak{D}_1$, which contains all real $\p\db$-exact
$(1,1)$-currents, and evaluates positively on $\mathfrak{D}_2$.
Hence, $\widetilde{\Omega}$ has to be a $\p\db$-closed positive
$(n-1,n-1)$-form.

The following mapping
\[ \begin{array}{cccc}
\pi: & \left\{ T \in \mathfrak{D'}^{1,1}_{\mathbb{R}}\ \Big|\ dT=0
\right\} & \> & H^{1,1}_{\mathrm{BC}}(X,\mathbb{R}) \\[2pt]
& T & \mapsto & \big[ T \big]_{\mathrm{BC}} \\
\end{array} \] is a canonical projection. $\pi(\mathfrak{D}_1)$ is
the null space determined by the linear functional
$\big[\Omega\big]_{\mathrm{A}}$ on
$H^{1,1}_{\mathrm{BC}}(X,\mathbb{R})$, namely
\[ \left\{ \big[T\big]_{\mathrm{BC}} \in
H^{1,1}_{\mathrm{BC}}(X,\mathbb{R})\ \Big|\ \int_X \Omega \w T = 0
\right\},\] since the class $\big[\Omega\big]_{\mathrm{A}}$ belongs
to $H^{n-1,n-1}_{\mathrm{A}}(X,\mathbb{R})$, which can be seen as
the dual space of $H^{1,1}_{\mathrm{BC}}(X,\mathbb{R})$.
The linear functional $\big[ \widetilde{\Omega} \big]_{\mathrm{A}}$
vanishes on the null space, which implies $\big[ \widetilde{\Omega}
\big]_{\mathrm{A}} = a \big[ \Omega \big]_{\mathrm{A}}$ for some $a
\in \mathbb{R}$.

If there exists no non-zero $d$-closed positive $(1,1)$-current on
$X$, by \cite[Proposition 5.4]{P1}, the Gauduchon cone
$\mathcal{G}_X$ will degenerate. Therefore, the class $\big[ \Omega
\big]_{\mathrm{A}}$ will surely lie in $\mathcal{G}_X$. Assume that
there exists a non-zero $d$-closed positive $(1,1)$-current $T$.
Clearly, $\int_X \widetilde{\Omega} \w T = a\int_X \Omega \w T$.
Moreover, $\widetilde{\Omega}$ is positive on $\mathfrak{D}_2$,
which implies $\int_X \widetilde{\Omega} \w T> 0$, and $\int_X
\Omega \w T> 0$ by the assumption of the lemma. Thus $a>0$.
Therefore, $\big[ \Omega \big]_{\mathrm{A}} = \frac{1}{a} \big[
\widetilde{\Omega} \big]_{\mathrm{A}}$, with $\widetilde{\Omega}$ a
positive form, lives in $\mathcal{G}_X$.
\end{proof}

The closure of the Gauduchon cone $\overline{\mathcal{G}}_X$ (cf.
\eqref{clo-G} and \cite[the part before Proposition 5.8]{P1}) and
the pseudo-effective cone $\mathcal{E}_X$ are closed convex cones
when $X$ is any compact complex manifold. By use of Lemma
\ref{G_class}, we can get the so-called \emph{Lamari's duality}. See
\cite[Lemma 3.3]{L} and \cite[the remark before Theorem 1.8 and the
proof of Theorem 5.9]{PU}.

\begin{proposition}\label{closed-dual}
Let $X$ be a compact complex manifold. Then
$\overline{\mathcal{G}}_X$ and $\mathcal{E}_X$ are dual cones, i.e.,
$\big(\overline{\mathcal{G}}_X\big)^{\mathrm{v}_{\!c}} =
\mathcal{E}_X$ and $\big({\mathcal{E}}_X\big)^{\mathrm{v}_{\!c}} =
\overline{\mathcal{G}}_X$.
\end{proposition}

\begin{proof}
It is clear that $\mathcal{E}_X \subseteq
\big(\overline{\mathcal{G}}_X\big)^{\mathrm{v}_{\!c}}$ and $
\overline{\mathcal{G}}_X \subseteq \big(
{\mathcal{E}}_X\big)^{\mathrm{v}_{\!c}}$. Let $\big[ \Omega
\big]_{\mathrm{A}} \in \big( {\mathcal{E}}_X
\big)^{\mathrm{v}_{\!c}}$, where $\Omega$ is a real $\p\db$-closed
$(n-1,n-1)$-form. Fix one class $\big[ \Omega_0\big]_{\mathrm{A}}
\in \mathcal{G}_X$ with $\Omega_0$ positive. Obviously, for any
fixed $\epsilon>0$, the integral \[ \int_{X} \big( \Omega + \epsilon
\Omega_0 \big) \w T = \int_X \Omega \w T + \epsilon \int_X \Omega_0
\w T >0, \] where $T$ is a non-zero $d$-closed positive
$(1,1)$-current. Hence, the class $\big[ \Omega \big]_{\mathrm{A}} +
\epsilon \big[ \Omega_0\big]_{\mathrm{A}} \in \mathcal{G}_X$ by
Lemma \ref{G_class}. Therefore, the class
$\big[\Omega\big]_{\mathrm{A}} \in \overline{\mathcal{G}}_X$, which
implies $\big( {\mathcal{E}}_X \big)^{\mathrm{v}_{\!c}} =
\overline{\mathcal{G}}_X$.

Now, let $\big[ \omega \big]_{\mathrm{BC}} \in
H^{1,1}_{\mathrm{BC}}(X,\mathbb{R})$, which does not  live in the
pseudo-effective cone $\mathcal{E}_X$. The point $\big[ \omega
\big]_{\mathrm{BC}}$ and $\mathcal{E}_X$ are a compact convex
subspace and a closed convex one, respectively, in the locally
convex space $H^{1,1}_{\mathrm{BC}}(X,\mathbb{R})$. From Hahn-Banach
separation theorem, there exists a continuous linear functional,
denoted by $\big[ \widetilde{\Omega} \big]_{\mathrm{A}}$, a class in
$H^{n-1,n-1}_{\mathrm{A}}(X,\mathbb{R})$, such that it evaluates
non-negatively on $\mathcal{E}_X$ and takes a negative value on the
point $\big[ \omega \big]_{\mathrm{BC}}$. Thus, the class $\big[
\widetilde{\Omega} \big]_{\mathrm{A}} \in \overline{\mathcal{G}}_X$,
from the equality $\big( {\mathcal{E}}_X\big)^{\mathrm{v}_{\!c}} =
\overline{\mathcal{G}}_X$. And the inequality $\int_X
\widetilde{\Omega} \w \omega <0$ indicates the inclusion \[
H^{1,1}_{\mathrm{BC}}(X,\mathbb{R}) \setminus \mathcal{E}_X \subseteq
H^{1,1}_{\mathrm{BC}}(X,\mathbb{R}) \setminus \big(
\overline{\mathcal{G}}_X\big)^{\mathrm{v}_{\!c}}, \] which implies
that $\mathcal{E}_X = \big(
\overline{\mathcal{G}}_X\big)^{\mathrm{v}_{\!c}}$.
\end{proof}

\begin{remark}\label{open_dual_add}\rm
Proposition \ref{closed-dual} enhances the result in Theorem
\ref{open_dual}. In fact, any class in
$H^{n-1,n-1}_{\mathrm{A}}(X,\mathbb{R}) \setminus
\overline{\mathcal{G}}_X$ must take a negative value on some class
of $\mathcal{E}_X$, and evaluates negatively on some class in the
interior $\mathcal{E}^{\circ}_X$ when $X$ is K\"{a}hler. Thus, each
class in $H^{n-1,n-1}_{\mathrm{A}}(X,\mathbb{R}) \setminus
\overline{\mathcal{G}}_X$ does not  live in $\big(
\mathcal{E}_X^{\circ} \big)^{\mathrm{v}_{\!o}}$. Therefore,
$\overline{\mathcal{G}}_X \setminus \big[0\big]_{\mathrm{A}} =
\big(\mathcal{E}^{\circ}_X \big)^{\mathrm{v}_{\!o}}$.
\end{remark}

Recall that a compact complex manifold is \emph{balanced} if it
admits a balanced metric and the \emph{closure of its balanced cone}
is defined similarly to the one of Gauduchon cone \eqref{clo-G}.
\begin{proposition}\label{biject} For a compact balanced manifold
$X$, the convex cone $\mathcal{E}_{\p\db} \subseteq
H^{1,1}_{\mathrm{A}}(X,\mathbb{R})$, generated by Aeppli classes
represented by $\p\db$-closed positive $(1,1)$-currents, is closed.
And when $X$ also satisfies the $\p\db$-lemma, the following three
statements are equivalent:
\begin{enumerate}
\item The mapping $\mathbb{J}: \mathcal{B}_X \rightarrow \mathcal{G}_X$
is bijective;
\item The mapping $\overline{\mathbb{J}}:
\overline{\mathcal{B}}_X \rightarrow \overline{\mathcal{G}}_X$ is
bijective;
\item The mapping $\mathbbm{j}: \mathcal{E}_{X} \rightarrow
\mathcal{E}_{\p\db}$ is bijective,
\end{enumerate}
where the mapping $\mathbbm{j}$ is the restriction of the natural
isomorphism $\mathscr{L}: H^{1,1}_{\mathrm{BC}}(X,\mathbb{R}) \>
H^{1,1}_{\mathrm{A}}(X,\mathbb{R})$, induced by the identity map, to
the pseudo-effective cone $\mathcal{E}_X$.
\end{proposition}

\begin{proof}Fix a
balanced metric $\omega_X$ on $X$. Let $\left
\{\big[T_k\big]_{\mathrm{A}} \right\}_{k \in \mathbb{N}^+}$ be a
sequence in the cone $\mathcal{E}_{\p\db}$, where $T_k$ are
$\p\db$-closed positive $(1,1)$-currents. And the sequence converges
to an Aeppli class $\big[ \alpha \big]_{\mathrm{A}}$ in
$H^{1,1}_{\mathrm{A}}(X,\mathbb{R})$. It is clear that
\[\lim\limits_{k\>+\infty} \int_X T_k \w \omega_X^{n-1}
= \int_X \alpha \w \omega_X^{n-1}.\] Thus, the sequence
$\big\{T_k\big\}_{k\in\mathbb{N}^+}$ is bounded in mass, and
therefore weakly compact. Denote the limit of a weakly convergent
subsequence $\big\{T_{k_i}\big\}$ by $T$. It is easy to check that
$T$ is a $\p\db$-closed positive $(1,1)$-current and
$\big[T\big]_{\mathrm{A}} = \big[ \alpha \big]_{\mathrm{A}}$. Hence,
$\big[ \alpha \big]_{\mathrm{A}} \in \mathcal{E}_{\p\db}$, which
implies that the convex cone $\mathcal{E}_{\p\db}$ is closed.

It is obvious that the three mappings
$\mathbb{J},\overline{\mathbb{J}}$ and $\mathbbm{j}$ are injective,
since $\mathscr{J}$ and $\mathscr{L}$ are isomorphisms as long as
the complex manifold $X$ satisfies the $\p\db$-lemma.

$(1) \Rightarrow (2)$~: We need to show that the inverse
$\mathscr{J}^{\!-1}$ of the mapping $\mathscr{J}$ maps the closure
$\overline{\mathcal{G}}_X$ into the one $\overline{\mathcal{B}}_X$.
To see this, let $\big[ \Psi \big]_{\mathrm{A}} \in
\overline{\mathcal{G}}_X$. Denote the inverse image
$\mathscr{J}^{\!-1}(\big[\Psi\big]_{\mathrm{A}})$ of $\big[\Psi
\big]_{\mathrm{A}}$ under the mapping $\mathscr{J}$ by
$\big[\Omega\big]_{\mathrm{BC}}$. For any $\epsilon>0$,
\[ \mathscr{J}^{\!-1}(\big[\Psi\big]_{\mathrm{A}} + \epsilon
\big[\omega^{n-1}_X\big]_{\mathrm{A}}) =
\big[\Omega\big]_{\mathrm{BC}} + \epsilon
\big[\omega^{n-1}_X\big]_{\mathrm{BC}} \in \mathcal{B}_X,\] since
$\mathbb{J}$ is bijective and thus
$\mathscr{J}^{\!-1}(\mathcal{G}_X) \subseteq \mathcal{B}_X$. This
implies that $\big[\Omega \big]_{\mathrm{BC}} \in
\overline{\mathcal{B}}_X$. Then
$\mathscr{J}^{\!-1}(\overline{\mathcal{G}}_X) \subseteq
\overline{\mathcal{B}}_X$, namely, the mapping $\mathscr{J}^{\!-1}:
\overline{\mathcal{G}}_X \> \overline{\mathcal{B}}_X$ is
well-defined. Hence, $\mathscr{J}^{\!-1}$ is the inverse of the
mapping $\overline{\mathbb{J}}$ and thus $\overline{\mathbb{J}}$ is
bijective.

$(2)\Rightarrow (3)$~: $\overline{\mathcal{G}}_X$ and
$\mathcal{E}_X$ are dual cones by Proposition \ref{closed-dual}.
$\overline{\mathcal{B}}_X$ and $\mathcal{E}_{\p\db}$ are also dual
cones by \cite[Lemma 3.3 and Remark 3.4]{FX}. Hence, the mapping
$\mathbbm{j}$ is bijective due to the bijectivity of
$\overline{\mathbb{J}}$.

$(3) \Rightarrow (1)$~: It has to be shown that $\mathbb{J}$ is
surjective. Let $\big[\Omega\big]_{\mathrm{BC}}$ be a class in
$H^{n-1,n-1}_{\mathrm{BC}}(X,\mathbb{R})$, which is mapped into
$\mathcal{G}_X$ by $\mathscr{J}$. Then there exists a $\p\db$-closed
positive $(n-1,n-1)$-form $\Psi$ and an $(n-2,n-1)$-form $\Theta$,
such that
\[ \Omega = \Psi + \p \Theta + \overline{\p \Theta}.
\] Let $\widetilde{T}$ be any fixed nonzero $\p\db$-closed positive
$(1,1)$-current. From the bijectivity of $\mathbbm{j}$, there exists
a $d$-closed positive $(1,1)$-current $T$ and a $(0,1)$-current $S$,
such that
\[ \widetilde{T} = T + \p S + \overline{\p S}. \]
The current $T$ can not  be zero current. If not, $\widetilde{T} = \p S
+ \overline{\p S}$, which implies that the integral $\int_X
\omega_X^{n-1} \w \widetilde{T}$ will be larger than $0$ and also
equal to $0$. This is a contradiction. Hence,
\[ \int_X \Omega \w \widetilde{T} = \int_X \Omega \w (T+\p S + \overline{\p S})
= \int_X \Omega \w T = \int_X (\Psi + \p \Theta + \overline{\p
\Theta}) \w T = \int_X \Psi \w T >0 .\] Therefore, the class
$\big[\Omega\big]_{\mathrm{BC}}$ lies in the balanced cone
$\mathcal{B}_X$ by \cite[Lemma 3.3]{FX} and thus the mapping
$\mathbb{J}$ is surjective.
\end{proof}

\begin{definition}[{\cite[Definition
$1.3.(ii)$]{BDPP}}]\label{movable}\rm Movable cone $\mathcal{M}_X$

Define the \emph{movable cone} $\mathcal{M}_X \subseteq
H^{n-1,n-1}_{\mathrm{BC}}(X,\mathbb{R})$ to be the closure of the
convex cone generated by classes of currents in the type
\[ \mu_{\ast}(\widetilde{\omega}_1 \w \cdots \w \widetilde{\omega}_{n-1}) \]
where $\mu: \widetilde{X} \> X$ is an arbitrary modification and
$\widetilde{\omega}_j$ are K\"{a}hler forms on $\widetilde{X}$ for
$1 \leq j \leq n-1$. Here, $X$ is an $n$-dimensional compact
K\"{a}hler manifold.
\end{definition}

We restate a lemma hidden in \cite[Appendix]{FX} and \cite{Toma}.

\begin{lemma}\label{two-inclusions}
Let $X$ be a compact K\"{a}hler manifold. There exist the following
inclusions:
\[ \mathcal{E}_X \subseteq \mathscr{L}^{-1}(\mathcal{E}_{\p\db})
\subseteq \big(\mathcal{M}_X\big)^{\mathrm{v}_c}, \] where
$\mathscr{L}^{-1}(\mathcal{E}_{\p\db})$ denotes the inverse image of
the cone $\mathcal{E}_{\p\db}$ under the isomorphism $\mathscr{L}$.
Note that $H^{1,1}_{\mathrm{BC}}(X,\mathbb{R})$ and
$H^{n-1,n-1}_{\mathrm{BC}}(X,\mathbb{R})$ are dual vector spaces in
the K\"{a}hler case.
\end{lemma}

\begin{proof}
It is clear that the mapping $\mathscr{L}$ is an isomorphism from
$H^{1,1}_{\mathrm{BC}}(X,\mathbb{R})$ to
$H^{1,1}_{\mathrm{A}}(X,\mathbb{R})$ and $\mathbbm{j}$ is injective
in the K\"{a}hler case. Thus, $\mathcal{E}_X \subseteq
\mathscr{L}^{-1}(\mathcal{E}_{\p\db})$. Let
$\big[\alpha\big]_{\mathrm{BC}}$ be a class in the cone
$\mathscr{L}^{\!-1}(\mathcal{E}_{\p\db})$ with $\alpha$ a smooth
representative, which implies that $\big[\alpha\big]_{\mathrm{A}}$
contains a $\p\db$-closed positive $(1,1)$-current $\widetilde{T}$.

To see $\mathscr{L}^{-1}(\mathcal{E}_{\p\db}) \subseteq
\big(\mathcal{M}_X\big)^{\mathrm{v}_c}$, we need to show that
$\int_X \alpha \w \mu_{\ast}(\widetilde{\omega}_1 \w \cdots \w
\widetilde{\omega}_{n-1}) \geq 0$ for arbitrary modification $\mu:
\widetilde{X} \> X$ and K\"{a}hler forms $\widetilde{\omega}_j$ on
$\widetilde{X}$. A result in \cite{AB95} states that for arbitrary
modification $\mu: \widetilde{X} \> X$ and any $\p\db$-closed
positive $(1,1)$-current $\widetilde{T}$ on $X$, there exists a
unique $\p\db$-closed positive $(1,1)$-current $T'$ on
$\widetilde{X}$ such that $\mu_{\ast}T' = \widetilde{T}$ and $T' \in
\mu^{\ast}(\big[\widetilde{T}\big]_{\mathrm{A}})$. Here, we choose
$\tilde{T}$ to be the one in the Aeppli class $\big[ \alpha
\big]_{\mathrm{A}}$. Then, one has
\[ \int_X \alpha \w \mu_{\ast}(\widetilde{\omega}_1 \w \cdots \w
\widetilde{\omega}_{n-1}) = \int_{\widetilde{X}} \mu^{\ast}\alpha \w
\widetilde{\omega}_1 \w \cdots \w \widetilde{\omega}_{n-1} =
\int_{\widetilde{X}} T' \w \widetilde{\omega}_1 \w \cdots \w
\widetilde{\omega}_{n-1} \geq 0, \] where $T'$ and $\mu^{\ast}
\alpha$ belong to the same Aeppli class on $\widetilde{X}$.
\end{proof}

\begin{corollary}[{\cite[Section 6]{P4}}] If Conjecture \ref{conj-mov} is assumed to hold true, then for a complex manifold $X$
in the Fujiki class $\mathcal{C}$,
\begin{equation}\label{db-i}
\mathscr{J}^{-1}(\mathcal{G}_X)=\mathcal{B}_X
\end{equation}
and thus Conjecture \ref{ddbar-b} is true in this case.
\end{corollary}
\begin{proof} The argument is a bit different from that in \cite[Section 6]{P4} (or \cite[Section
2]{crs}) and we claim no originality here. That $X$ is balanced is
obviously a result of \eqref{db-i} since the Gauduchon cone of a
compact complex manifold is never empty and $\mathscr{J}$ is an
isomorphism from the $\p\db$-lemma. Now let us prove \eqref{db-i}
under the assumption of Conjecture \ref{conj-mov}. Without loss of
generality, we can assume that $X$ is K\"{a}hler and thus this
equality is a direct corollary of Lemma \ref{two-inclusions} and
Proposition \ref{biject}.
\end{proof}
Boucksom-Demailly-Paun-Peternell have proved in \cite[Theorem 10.12,
Corollary 10.13]{BDPP} that Conjecture \ref{conj-mov} is true, when
$X$ is a compact hyperk\"{a}hler manifold or a compact K\"{a}hler
manifold which is a limit by deformation of projective manifolds
with Picard number $\rho = h^{1,1}$. It follows that $\mathbb{J}$ is
bijective in these two cases. The qualitative part of Transcendental
Morse Inequalities Conjecture for differences of two nef classes
\cite[Conjecture 10.1.(ii)]{BDPP} has been proved by Popovici
\cite{P5} and Xiao \cite{X1}. And a partial answer to the
quantitative part is given by \cite{P4}, with the case of
nef $T^{1,0}_X$ obtained in \cite[Proposition 3.2]{X2}.

The following theorem may provide some evidence for the assertion of
Question \ref{que} whether the mapping $\mathbb{J}$ is bijective
from the balanced cone $\mathcal{B}_X$ to the Gauduchon cone
$\mathcal{G}_X$ on the K\"{a}hler manifold $X$.

Let us recall several important results from \cite{Yst,BEGZ} on
solving complex Monge-Amp\`{e}re equations on a compact K\"{a}hler
manifold $X$.

Fix a K\"{a}hler metric $\omega_X$, a nef and big class
$\big[\alpha\big]_{\mathrm{BC}}$ and a volume form $\eta$ on $X$. By
Yau's celebrated results in \cite{Yst}, for $0 < t \leq 1$, there
exists a unique smooth function $u_t$, satisfying that $\sup_X u_t
=0$, such that $\alpha + t \omega + \sqrt{-1}\p\db u_t$ is a
K\"{a}hler metric and
\[\big(\alpha + t\omega_X + \sqrt{-1}\p\db u_t \big)^{n} = c_t \eta,\]
where $c_t = \frac{\int_X (\alpha + t\omega_X)^{n}}{\int_X \eta}$.
As in \cite[Theorems B and C]{BEGZ}, when $t$ is equal to $0$, there
exists a unique $\alpha$-psh $u$, satisfying that $\sup_X u =0$,
such that \[ \big\langle (\alpha + \sqrt{-1}\p\db u)^{n} \big\rangle
= c \eta, \] where $c= \frac{\int_X \alpha^n}{\int_X \eta}$ and the
bracket $\langle \cdot \rangle$ denotes the non-pluripolar product
of positive currents. Moreover, $u$ has minimal singularities and is
smooth on $\mathrm{Amp}(\alpha)$, which is a Zariski open set on $X$
and only depends on the class $\big[ \alpha\big]_{\mathrm{BC}}$.

These results above can be viewed in the following manner as stated in
\cite[the part after Lemma 2.3]{FX}. The family of solutions $u_t$
is compact in $L^1(X)$-topology. Then there exists a sequence
$u_{t_k}$ such that
\[ \alpha + t_k \omega_X + \sqrt{-1}\p\db u_{t_k}
\rightarrow \alpha + \sqrt{-1} \p \db u \] in the sense of currents
on $X$ with $t_{k}\>0$. Meanwhile, $u_t$ is compact in
$C^{\infty}_{loc}(\mathrm{Amp}(\alpha))$, which means uniform
convergence on any compact subset of $\mathrm{Amp}(\alpha)$.
Therefore, there exists a subsequence of $u_{t_k}$, still denoted by
$u_{t_k}$, such that
\[ \alpha + t_k \omega_X + \sqrt{-1}\p\db u_{t_k} \rightarrow
\alpha + \sqrt{-1}\p\db u \] in the sense of
$C^{\infty}_{loc}(\mathrm{Amp}(\alpha))$. Hence $u$ is smooth on
$\mathrm{Amp}(\alpha)$ and $\alpha + \sqrt{-1}\p\db u$ is a
K\"{a}hler metric on $\mathrm{Amp}(\alpha)$, since $\eta$ is a
volume form.

\begin{theorem}\label{nef-b-g}
Let $X$ be a compact K\"{a}hler manifold and
$\big[\alpha\big]_{\mathrm{BC}}$ a nef class. Then
$\big[\alpha^{n-1}\big]_{\mathrm{A}} \in \mathcal{G}_X$ implies that
$\big[ \alpha^{n-1}\big]_{\mathrm{BC}} \in \mathcal{B}_X$. Hence,
$\overline{\mathbb{I}}(\overline{\mathcal{K}}_X) \bigcap
\mathcal{B}_X$ and $\overline{\mathbb{K}}(\overline{\mathcal{K}}_X)
\bigcap \mathcal{G}_X$ can be identified by the mapping
$\mathbb{J}$.
\end{theorem}

\begin{proof}
Assume that $\big[ \alpha^{n-1}\big]_{\mathrm{A}}$ belongs to
$\mathcal{G}_X$, where $\big[\alpha \big]_{\mathrm{BC}}$ is a nef
class. From Lemma \ref{G_class}, for any nonzero $d$-closed positive
$(1,1)$-current $T$, the integral $\int_X \alpha^{n-1} \w T>0$.
Since the nef cone $\overline{\mathcal{K}}_X $ is contained in the
pseudo-effective cone $\mathcal{E}_X$, the nef class
$\big[\alpha\big]_{\mathrm{BC}}$ contains a $d$-closed positive
$(1,1)$-current $S$, which can not  be the zero current. Otherwise,
$\big[ 0\big]_{\mathrm{A}} \in \mathcal{G}_X$, which contradicts
with Lemma \ref{Gcone}. Then, the integral $\int_X \alpha^{n} =
\int_X \alpha^{n-1} \w S >0$, which implies that the class
$\big[\alpha\big]_{\mathrm{BC}}$ is nef and big, by \cite[Theorem
0.5]{DP}.

Let $Q$ be any fixed $\p\db$-closed positive $(1,1)$-current on $X$.
From the discussion before this theorem, it is clear that the
sequence of positive measures \[\big\{ (\alpha + t_k \omega_X +
\sqrt{-1}\p\db u_{t_k})^{n-1} \w Q \big\}_{k \in \mathbb{N}^{+}}\]
has bounded mass, for example
\[\int_X (\alpha
+ t_k \omega_X + \sqrt{-1}\p\db u_{t_k})^{n-1} \w Q \leq \int_X
(\alpha + \omega_X)^{n-1} \w Q.\] Therefore, there exists a
subsequence, still denoted by
\[\big\{ (\alpha + t_k \omega_X + \sqrt{-1}\p\db u_{t_k})^{n-1} \w Q
\big\}_{k \in \mathbb{N}^{+}},\] weakly convergent to a positive
measure on $X$, denoted by $\mu$. It follows that \[\int_X \mu =
\int_X \alpha^{n-1} \w Q,\] since the equalities hold \[ \int_X \mu
= \lim_{k \> +\infty} \int_X (\alpha + t_k \omega_X + \sqrt{-1}\p\db
u_{t_k})^{n-1} \w Q = \lim_{k \> +\infty} \int_X (\alpha + t_k
\omega_X)^{n-1} \w Q = \int_X \alpha^{n-1} \w Q.\] Note that \[
\left. (\alpha + \sqrt{-1}\p\db u)^{n-1} \w Q
\right|_{\mathrm{Amp}(\alpha)}\] is a well-defined positive measure
on $\mathrm{Amp}(\alpha)$, since $\alpha + \sqrt{-1}\p\db u$ is a
K\"{a}hler metric on $\mathrm{Amp}(\alpha)$. Moreover, $\mu$ is
equal to \[ \left. (\alpha + \sqrt{-1}\p\db u)^{n-1} \w Q
\right|_{\mathrm{Amp}(\alpha)}\] on $\mathrm{Amp}(\alpha)$.
Actually, for any smooth function $f$ with $\mathrm{Supp}(f) \subseteq
\mathrm{Amp}(\alpha)$, one has
\begin{align}
\int_{\mathrm{Amp}(\alpha)} f \mu &= \int_X f \mu   \notag \\
&= \lim_{k\>+\infty} \int_X f(\alpha + t_k \omega_X + \sqrt{-1}\p\db
u_{t_k})^{n-1} \w Q  \notag \\
&= \int_X f(\alpha + \sqrt{-1}\p\db u)^{n-1} \w Q \label{4.4}\\
&= \int_{\mathrm{Amp}(\alpha)} f(\alpha + \sqrt{-1}\p\db u)^{n-1} \w
Q \notag\\
&= \int_{\mathrm{Amp}(\alpha)} f \left( \left.(\alpha +
\sqrt{-1}\p\db u)^{n-1} \w Q \right|_{\mathrm{Amp}(\alpha)} \right),
\notag
\end{align}
where the equality $\eqref{4.4}$ results from that the sequence $f(\alpha
+ t_k \omega_X + \sqrt{-1}\p\db u_{t_k})^{n-1}$ converges to
$f(\alpha + \sqrt{-1}\p\db u)^{n-1}$ in the sense of smooth
$(n-1,n-1)$-forms on $X$ due to the convergence result stated before
this theorem, with all their supports always contained in
$\mathrm{Amp}(\alpha)$.

It is obvious that the integral $\int_X \alpha^{n-1} \w Q \geq 0$
for $\big[ \alpha \big]_{\mathrm{BC}}$ nef. Now suppose that $\int_X
\alpha^{n-1} \w Q =0$. Then we have $\int_X \mu = \int_X
\alpha^{n-1} \w Q =0$. And $\mu$ is equal to $ \left.(\alpha +
\sqrt{-1}\p\db u)^{n-1} \w Q \right|_{\mathrm{Amp}(\alpha)}$ on
$\mathrm{Amp}(\alpha)$ with $(\alpha + \sqrt{-1}\p\db u)^{n-1}$ a
positive $(n-1,n-1)$-form on $\mathrm{Amp}(\alpha)$. Then
$\mathrm{Supp}(Q) \subseteq X \setminus \mathrm{Amp}(\alpha)$, which
is an analytic subvariety $V$ on $X$ with $\dim V \leq n-1$.

Denote the irreducible components with dimension $n-1$ of $V$ by $\{
V_i \}_{i=1}^m$. By \cite[Theorem 1.5]{AB92} and \cite[Lemma
3.5]{FX}, there exist constants $c_i \geq 0$ for $1 \leq i \leq m$
such that
\[ Q - \sum_{i=1}^m c_i [V_i] =0, \]
since $V$ has no irreducible component of dimension larger than
$n-1$. And we have $\int_X \alpha^{n-1} \w [V_i]>0$, where $[V_i]$
are nonzero $d$-closed positive $(1,1)$-currents for $1 \leq i \leq
m$. Then $\int_X \alpha^{n-1} \w Q =0$ forces that the constants
$c_i$ are all equal to $0$, namely $Q$ a zero current. Hence,
$\big[\alpha^{n-1}\big]_{\mathrm{BC}} \in \mathcal{B}_X$ from
\cite[Lemma 3.3]{FX}.

It is clear that the restricted mapping $\mathbb{J}$, from
$\overline{\mathbb{I}}(\overline{\mathcal{K}}_X) \bigcap
\mathcal{B}_X$ to $\overline{\mathbb{K}}(\overline{\mathcal{K}}_X)
\bigcap \mathcal{G}_X$, is injective. And the proof above shows that
it is also surjective. Hence the restricted mapping $\mathbb{J}$ is
bijective.
\end{proof}

We will describe the degeneration of balanced cones on compact
complex manifolds, similar to the case of Gauduchon cones in
\cite[Proposition 5.4]{P1}.
\begin{lemma}\label{degenerate_B_cone}
Let $X$ be a compact complex manifold. Then the balanced cone
$\mathcal{B}_X$ degenerates if and only if there exists no non-zero
$\p\db$-closed positive $(1,1)$-current $T$ on $X$.
\end{lemma}
\begin{proof}
Assume that $\mathcal{B}_X =
H^{n-1,n-1}_{\mathrm{BC}}(X,\mathbb{R})$. In particular, there
exists a Hermitian metric $\widetilde{\omega}$ on $X$, such that
$\widetilde{\omega}^{n-1}$ is $\p\db$-exact. If $T$ is a non-zero
$\p\db$-closed positive $(1,1)$-current on $X$, the integral $\int_X
\widetilde{\omega}^{n-1} \w T $ has to be larger than $0$ for the
form $\widetilde{\omega}^{n-1}$ being positive and simultaneously
equal to zero as $\widetilde{\omega}^{n-1}$ is $\p\db$-exact. This
contradiction leads to non-existence of such current $T$.

Conversely, assume that there exists no non-zero $\p\db$-closed
positive $(1,1)$-current $T$ on $X$. Let
$\mathfrak{D'}_{\mathbb{R}}^{1,1}$ be the set of real
$(1,1)$-currents on $X$ with the weak topology. Fix a Hermitian
metric $\omega_{X}$ on $X$. Then apply the Hahn-Banach separation
theorem.

Let us set \[ \mathfrak{D}_1 = \left\{ T \in
\mathfrak{D'}_{\mathbb{R}}^{1,1}\ \Big|\ \p \db T=0 \right\},\]
\[ \mathfrak{D}_2 = \left\{ T \in
\mathfrak{D'}_{\mathbb{R}}^{1,1}\ \Big|\ \int_{X} \omega_X^{n-1} \w
T =1\ \textrm{and} \ T \geq 0 \right\}. \] It is easy to see that
$\mathfrak{D}_1$ is a closed linear subspace of the locally convex
space $\mathfrak{D'}_{\mathbb{R}}^{1,1}$, while $\mathfrak{D}_2$ is
a compact convex one in $\mathfrak{D'}_{\mathbb{R}}^{1,1}$. And
$\mathfrak{D}_1 \bigcap \mathfrak{D}_2 = \emptyset$ from the
assumption. Then there exists a continuous linear functional on
$\mathfrak{D'}_{\mathbb{R}}^{1,1}$, denoted by $\Omega$, a real
$(n-1,n-1)$-form, such that it vanishes on $\mathfrak{D}_1$ and
evaluates positively on $\mathfrak{D}_2$. Hence, $\Omega$ has to be
a $\p\db$-exact positive $(n-1,n-1)$-form. It follows that the class
$\big[ \Omega\big]_{\mathrm{BC}}$ is the zero class in
$H^{n-1,n-1}_{\mathrm{BC}}(X,\mathbb{R})$, which also lives in the
balanced cone $\mathcal{B}_X$, which implies that the balanced cone
$\mathcal{B}_X$ degenerates.
\end{proof}

\begin{remark}\label{G_dgnt_B_dgnt_1}\rm \cite[Proposition 5.4]{P1}
tells us the Gauduchon cone of a compact complex manifold $X$
degenerates if and only if there exists no non-zero $d$-closed
positive $(1,1)$-current on $X$, and, together with Proposition
\ref{degenerate_B_cone}, implies that the Gauduchon cone of a
compact balanced manifold will degenerate when its balanced cone
does.
\end{remark}

\begin{question} \emph{Fu-Li-Yau \cite{FLY} constructed a balanced threefold,
which is a connected sum of $k$-copies of $S^3 \times S^3$ $(k \geq
2)$ and whose balanced cone degenerates (cf. \cite{FX}). Is it
possible to find a balanced manifold such that its Gauduchon cone
degenerates while its balanced cone does not ?}
\end{question}

\subsection{Deformation results related with
$\mathcal{G}_X$}\label{Gcone-deform} In this subsection, we will
discuss several deformation results related with $\mathcal{G}_X$ in
Theorems \ref{inclusion_1} and \ref{inclusion_2}.

Firstly, let us review Demailly's regularization theorem \cite{D1},
whose different versions have been used by various authors in the
literature. Recall that a real $(1,1)$-current $T$ is said to be
\emph{almost positive} if $T \geq \gamma$ for some real smooth
$(1,1)$-form, and each $d$-closed almost positive $(1,1)$-current
$T$ on a compact complex manifold can be written as $\theta +
\sqrt{-1} \p \db f$, where $\theta$ is a $d$-closed smooth
$(1,1)$-form with $f$ almost plurisubharmonic (shortly almost psh)
function (cf. \cite[Section 2.1]{Bou_1} and \cite[Section 3]{DP}).
We say that a $d$-closed almost positive $(1,1)$-current $T$ has
\emph{analytic (or algebraic) singularities} along the analytic
subvariety $Y$, if $f$ does, i.e., $f$ can be locally written as
\[ \frac{c}{2} \log (|g_1|^2+|g_2|^2+ \cdots+|g_N|^2) + h, \]
where $c>0$ (or $c \in \mathbb{Q}^+$), $\{g_i\}_{i=1}^N$ are local
generators of the ideal sheaf of $Y$ and $h$ is some smooth
function. It is clear that $T$ is smooth outside the singularity
$Y$. Then the following formulation of \emph{Regularization Theorem}
will be applied:
\begin{theorem}[{\cite[Theorem 3.2]{DP}; \cite[Theorem 2.4]{Bou_1};  \cite[Theorem 2.1]{Bou_2}}]
\label{regularization} Let $T = \theta + \sqrt{-1} \p \db f$ be a
$d$-closed almost positive $(1,1)$-current on a compact complex
manifold $X$, satisfying that $T \geq \gamma$ for some real smooth
$(1,1)$-form. Then there exists a sequence of functions $f_k$ with
analytic singularities $Y_k$ converging to $f$, such that, if we set
$T_k = \theta + \sqrt{-1} \p \db f_k$, it follows that
\begin{enumerate}
\item $T_k$ weakly converges to $T$;
\item $T_{k} \geq \gamma - \epsilon_k \omega$, where $\lim\limits_{k \> +\infty} \downarrow \epsilon_k =0$
and $\omega$ is some fixed Hermitian metric;
\item The Lelong numbers $\nu(T_k,x)$ increase to $\nu(T,x)$ uniformly with respect to $x\in X$;
\item The analytic singularities increase with respect to $k$, i.e., $Y_k \subseteq Y_{k+1}$.
\end{enumerate}
\end{theorem}
Denote the blow up of $X$ along the singularity $Y_k$ by $\mu_k:
\tilde{X}_k \rightarrow X$, and we will see that $\mu_k^*(T_k)$
still acquires the analytic singularity $\mu_k^{-1}(Y_k)$, without
irreducible components of complex codimensions at least $2$, for
each $k$. According to \cite[Section 2.5]{Bou_2}, the Siu's
decomposition \cite{Siu} for $\mu_k^*(T_k)$ writes
\begin{equation}\label{Siu_up}
\mu_k^*(T_k) = \tilde{R}_k +  \sum_j \nu_{kj} \big[ \tilde{Y}_{kj}
\big],
\end{equation}
where $\tilde{R}_k$ is a $d$-closed smooth $(1,1)$-form, satisfying
that $\tilde{R}_k \geq \mu_k^* ( \gamma - \epsilon_k \omega)$,
$\tilde{Y}_{kj}$ are irreducible components of complex codimension
one of $\mu_k^{-1}(Y_k)$ for all $j$, and $\nu_{kj}$ are all
positive numbers. It is obvious that the degree of $\mu_k$ is equal
to $1$ for each $k$. It follows that, after the push forward,
\begin{equation}\label{Siu_down}
T_k = \mu_{k*} \big( \mu_k^*(T_k) \big) = \mu_{k*}(\tilde{R}_k) +
\sum_j \nu_{kj} \big[ Y_{kj} \big],
\end{equation}
which is exactly the Siu's decomposition for $T_k$. Here,
$\mu_{k*}(\tilde{R}_k)$ is a $d$-closed positive $(1,1)$-current,
which is smooth outside irreducible components of complex
codimension at least $2$ of $Y_k$ and satisfies that
$\mu_{k*}(\tilde{R}_k) \geq \gamma - \epsilon_k \omega$. The symbols
$Y_{kj}$ stand for the irreducible components of complex codimension
one of $Y_k$, since the following equalities hold
\[ \mu_{k*} \Big( \big[ \tilde{Y}_{kj} \big] \Big)=
\begin{cases} \big[ \mu_k(\tilde{Y}_{kj}) \big], &\quad \text{when}\ \dim \mu_k(\tilde{Y}_{kj}) = n-1; \\
                         \qquad0,  &\quad \text{when}\ \dim \mu_k(\tilde{Y}_{kj}) < n-1.            \\
\end{cases}\]

Meanwhile, Barlet's theory \cite{Bar} of cycle spaces comes into
play and let us follow the statements in Demailly-Paun's paper
\cite[Section 5]{DP}. Let $\pi: \mathcal{X} \rightarrow
\Delta_{\epsilon}$ be a holomorphic family of K\"ahler fibers of
complex dimension $n$. Then there is a canonical holomorphic
projection
\[ \pi_p: C^{p}(\mathcal{X} / \Delta_{\epsilon}) \rightarrow \Delta_{\epsilon},\]
where $C^{p}(\mathcal{X} / \Delta_{\epsilon})$ denotes the relative
analytic cycle space of complex dimension $p$, i.e., all cycles
contained in the fibers of the family $\pi: \mathcal{X} \rightarrow
\Delta_{\epsilon}$. And it is known that the restriction of $\pi_p$
to the connected components of
$C^{p}(\mathcal{X}/\Delta_{\epsilon})$ are proper maps by the
K\"ahler property of the fibers. Also, there is a cohomology class
map, commuting with the projection to $\Delta_{\epsilon}$, defined
by
\[ \begin{array}{cccc}
\iota_p : &  C^{p}(\mathcal{X} / \Delta_{\epsilon}) & \rightarrow &
\mathrm{R}^{2(n-p)}\pi_{*} \big( \mathbb{Z_{\mathcal{X}}} \big) \\
          &    Z   & \mapsto  & \big[ Z \big], \\
\end{array} \]
which associates to every analytic cycle $Z$ in $X_t$ its cohomology
class $\big[Z\big] \in H^{2(n-p)}(X_t,\mathbb{Z})$. Again by the
K\"ahlerness, the mapping $\iota_{p}$ is proper.

Denote the images in $\Delta_{\epsilon}$ of those connected
components of $C^{p}(\mathcal{X} / \Delta_{\epsilon})$ which do not
project onto $\Delta_{\epsilon}$ under the mapping $\pi_{p}$ by
$\bigcup S_{\nu}$, namely a countable union of analytic subvarieties
$S_{\nu}$ of $\Delta_{\epsilon}$, from the properness of the mapping
$\pi_{p}$ restricted to each component of $C^{p}(\mathcal{X} /
\Delta_{\epsilon})$ for $1 \leq p \leq n-1$ (cf. \cite[proof of
Theorem 0.8]{DP}). Clearly, each $S_{\nu} \subsetneq
\Delta_{\epsilon}$. And thus, for $t \in  \Delta_{\epsilon}
\setminus \bigcup S_{\nu}$, every irreducible analytic subvariety of
complex codimension $n-p$ in $X_{t}$ can be extended into any other
fiber in the family $\pi: \mathcal{X} \rightarrow \Delta_{\epsilon}$
with the invariance of its cohomology class.

Now, let us go back to the deformation of Gauduchon cone. An
\emph{\textbf{sGG} manifold} is a compact complex manifold,
satisfying that each Gauduchon metric on it is strongly Gauduchon
from the definition in \cite[Lemma 1.2]{PU}. And the \textbf{sGG}
property is open under small holomorphic deformations from \cite[the
remark after Theorem 1.5]{PU}. Thus, let us call the holomorphic
family $\pi:\mathcal{X} \> \Delta_{\epsilon}$ with the central fiber
$X_0$ being an \textbf{sGG} manifold an \emph{\textbf{sGG} family}.
Moreover, Popovici and Ugarte proved that the following inclusion
holds
\[\mathcal{G}_{X_0} \subseteq \lim\limits_{t\>0} \mathcal{G}_{X_t}\]
when the family $\pi: \mathcal{X} \rightarrow \Delta_{\epsilon}$ is
an $\textbf{sGG}$ family in \cite[Definition 5.6, Theorem 5.7]{PU}.
The definition of $\lim\limits_{t\>0} \mathcal{G}_{X_t}$ is given by
\[ \lim\limits_{t\>0}
\mathcal{G}_{X_t} = \Big\{ \big[ \Omega \big]_{\mathrm{A}} \in
H^{n-1,n-1}_{\mathrm{A}}(X_0,\mathbb{R})\ \Big|\ \mathrm{P}_t \circ
\mathrm{Q}_{\,0} \Big(\big[ \Omega \big]_{\mathrm{A}}\Big) \in
\mathcal{G}_{X_t}\ \textrm{for}\ \textrm{sufficiently}\
\textrm{small}\ t\Big\},\] where the canonical mappings
$$\mathrm{P}_t: H^{2n-2}_{\mathrm{DR}}(X_t,\mathbb{R}) \>
H^{n-1,n-1}_{\mathrm{A}}(X_t,\mathbb{R})$$ send the De Rham class
$\big[ \Theta\big]_{\mathrm{DR}}$ to the Aeppli class $\big[
\Theta^{n-1,n-1}\big]_{\mathrm{A}}$, represented by the
$(n-1,n-1)$-component of $\Theta$ on $X_t$, and the mapping
$$\mathrm{Q}_{\,0}:  H^{n-1,n-1}_{\mathrm{A}}(X_0,\mathbb{R}) \>
H^{2n-2}_{\mathrm{DR}}(X_t,\mathbb{R}),$$ depends on a fixed
Hermitian metric $\omega_0$ on $X_0$ according to \cite[Definition
5.3]{PU}. By \cite[Proposition 5.1, Lemma 5.4]{PU}, the canonical
mappings $\mathrm{P}_t$ are surjective and the mapping
$\mathrm{Q}_{\,0}$ is injective, satisfying that \[\mathrm{P}_0
\circ \mathrm{Q}_{\,0} = \mathrm{id}_{H^{n-1,n-1}_{\mathrm{A}}
(X,\mathbb{R})}.\]

The following theorem gives a bound from the other side.
\begin{theorem}\label{inclusion_1}
Let $\pi: \mathcal{X} \rightarrow \Delta_{\epsilon}$ be a
holomorphic family with a K\"ahlerian central fiber. Then we have \[
\lim_{t \> \tau} \mathcal{G}_{X_t} \subseteq
\mathcal{N}_{X_{\tau}}\quad \textrm{for each}\ \ \tau \in
\Delta_{\epsilon},\] where $\mathcal{N}_{X_{\tau}}$ is the convex
cone generated by Aeppli classes of $\p_{\tau}\db_{\tau}$-closed
positive $(n-1,n-1)$-currents on $X_{\tau}$. Moreover, the following
inclusion holds,
\[  \lim_{t \> \tau} \mathcal{G}_{X_t} \subseteq \overline{\mathcal{G}}_{X_{\tau}} \quad
\textrm{for each}\ \ \tau \in \Delta_{\epsilon} \setminus \bigcup
S_{\nu},\] where $\bigcup S_{\nu}$ is explained above in this
section.
\end{theorem}

\begin{proof}
It is clear that we can assume that each fiber of the family $\pi:
\mathcal{X} \rightarrow \Delta_{\epsilon}$ is K\"ahler (apparently
an \textbf{sGG} family) and $\{\omega_t\}_{ t\in \Delta_{\epsilon}
}$ is a family of K\"ahler metrics of the fibers, varying smoothly
with respect to $t$, by use of the stability theorem of K\"{a}hler
structures \cite{KS}, after shrinking the disk $\Delta_{\epsilon}$.

For $\tau \in \Delta_{\epsilon}$, let
$\big[\Omega\big]_{\mathrm{A}}$ be an element of $\lim\limits
_{t\>\tau}{\mathcal{G}}_{X_t}$, $\Omega$ its smooth representative,
which indicates \[ \mathrm{P}_t \circ \mathrm{Q}_{\,\tau} \Big(
\big[\Omega\big]_{\mathrm{A}} \Big) \in \mathcal{G}_{X_t}\quad
\textrm{for}\ \ 0<|t-\tau|<\delta_{[\Omega]_{\mathrm{A}}}\] by
definition. Set the positive representative of $\mathrm{P}_t \circ
\mathrm{Q}_{\,\tau} ( \big[\Omega\big]_{\mathrm{A}} )$ as
$\Omega_t$. It is obvious that the following equality holds:
\[ \lim_{t\> \tau} \int_{X_t} \Omega_t \wedge \omega_t = \int_{X_{\tau}} \Omega \wedge \omega_{\tau}, \]
since the integral just depends on the Aeppli class of $\Omega_t$.
This implies that \[ \left\{ \Omega_t \right \}_{
0<|t-\tau|<\delta_{[\Omega]_{\mathrm{A}}}}\] have bounded mass, and
thus the weak limit of a subsequence is a
$\p_{\tau}\db_{\tau}$-closed positive $(n-1,n-1)$-current, which
lies in the Aeppli class $\big[ \Omega \big]_{\mathrm{A}}$ on
$X_{\tau}$. Hence, this shows \[ \lim_{t \> \tau} \mathcal{G}_{X_t}
\subseteq \mathcal{N}_{X_{\tau}}.\]

As to the second inclusion, let us fix $\tau \in \Delta_{\epsilon}
\setminus \bigcup S_{\nu}$. Then the following integral should be
considered
\[ \int_{X_{\tau}} \Omega \wedge T, \]
where $T$ is any fixed $d$-closed positive $(1,1)$-current on
$X_{\tau}$. Apply Theorem \ref{regularization} to $T$ and we have a
sequence of currents $T_k$ with analytic singularities, denoted by
$Y_k$, such that $T_k$ always lies in the Bott-Chern class
$\big[T\big]_{\mathrm{BC}}$ and $T_k \geq -\epsilon_k
\omega_{\tau}$. From the very definition of $\bigcup S_{\nu}$, the
singularity $Y_k$ on $X_{\tau}$, with possibly high codimensional
irreducible components, can be extended into the other fibers of the
family $\pi:\mathcal{X} \rightarrow \Delta_{\epsilon}$, for each
$k$. The extension of $Y_k$ is denoted by $\mathcal{Y}_k$, which is
a relative analytic subvariety of the total space $\mathcal{X}$ of
the family $\pi:\mathcal{X} \rightarrow \Delta_{\epsilon}$. Blow up
$\mathcal{X}$ along $\mathcal{Y}_k$, and then we will obtain
\[\tilde{\mathcal{X}}_k \stackrel{\mu_k}{\longrightarrow} \mathcal{X} \stackrel{\pi}{\longrightarrow} \Delta_{\epsilon}.\]
The restriction of $\mu_k$ to the $t$-fiber is exactly the blow up
$\mu_k(t): \tilde{X}_k(t) \rightarrow X_t$ of $X_t$ along $Y_k(t)$,
with the exceptional divisor denoted by $\tilde{Y}_k(t)$, where
$Y_k(t) = \mathcal{Y}_k \cap X_t$. Then we can apply Equalities
\eqref{Siu_up} and \eqref{Siu_down} to $T_k$:
\begin{equation}\label{estimate}
\begin{aligned}
\int_{X_{\tau}} \Omega \wedge T &= \int_{X_{\tau}} \Omega \wedge T_k \\
& =  \int_{X_{\tau}} \Omega \wedge \left( \mu_k(\tau)_{*} \big( \tilde{R}_k \big) + \sum_j \nu_{kj} \big[ Y_{kj} \big] \right) \\
& =  \int_{\tilde{X}_k(\tau)} \Big( \mu_k(\tau)^*\Omega \Big) \wedge \tilde{R}_k + \sum_j \nu_{kj} \int_{X_{\tau}} \Omega \wedge \big[ Y_{kj} \big],\\
\end{aligned}
\end{equation}
where $\tilde{R}_k \geq - \epsilon_k \mu_k(\tau)^* \omega_{\tau}$,
$Y_{kj}$ are irreducible components of complex codimension one of
$Y_k$ and $\nu_{kj}$ are positive numbers for all $j$.

We claim the following two statements:
\begin{enumerate}
\item\label{fst} $\int_{\tilde{X}_k(\tau)} \Big( \mu_k(\tau)^*\Omega \Big) \wedge \tilde{R}_k \geq
-\epsilon_k \int_{X_{\tau}} \Omega \wedge \omega_{\tau}$;
\item\label{scd} $ \int_{X_{\tau}} \Omega \wedge \big[ Y_{kj} \big] \geq 0 $.
\end{enumerate}
For the statement \eqref{fst}, we consider that
\[ \begin{aligned}
& \int_{\tilde{X}_k(\tau)} \Big( \mu_k(\tau)^*\Omega \Big) \wedge \tilde{R}_k \\
=& \int_{\tilde{X}_k(\tau)} \Big( \mu_k(\tau)^*\Omega \Big) \wedge
\Big(\tilde{R}_k + 2 \epsilon_k \mu_k(\tau)^* \omega_{\tau} \Big)
- 2 \epsilon_k \int_{\tilde{X}_k(\tau)} \Big( \mu_k(\tau)^*\Omega \Big) \wedge \Big( \mu_k(\tau)^*\omega_{\tau} \Big) \\
=& \int_{\tilde{X}_k(\tau)} \Big( \mu_k(\tau)^*\Omega \Big) \wedge \Big(\tilde{R}_k + 2 \epsilon_k \mu_k(\tau)^* \omega_{\tau} \Big) - 2 \epsilon_k \int_{X_{\tau}} \Omega \wedge \omega_{\tau}.\\
\end{aligned} \]
It should be noted that $\mu_k(\tau)^* \omega_{\tau}$ is a
semi-positive $(1,1)$-form on $\tilde{X}_k(\tau)$ for each $k$. And
thus, we can choose a sequence of positive numbers $\{\lambda_k\}_{k
\in \mathbb{N}^+}$, converging to $0$, such that
$\mu_k(\tau)^* \omega_{\tau} - \lambda_k u_k$ is positive for each
$k$, where $u_k$ is some smooth form in the Bott-Chern cohomology
class of $\big[ \tilde{Y}_k(\tau) \big]$ (cf. \cite[Lemma 3.5]{DP}).
Hence, the integral above amounts to the following equalities:
\[ \begin{aligned}
& \int_{\tilde{X}_k(\tau)} \Big( \mu_k(\tau)^*\Omega \Big) \wedge \tilde{R}_k \\
=& \int_{\tilde{X}_k(\tau)} \Big( \mu_k(\tau)^*\Omega \Big) \wedge
\Big(\tilde{R}_k + 2 \epsilon_k \mu_k(\tau)^* \omega_{\tau}
- \epsilon_k \lambda_k u_k \Big) \\
&+ \epsilon_k \lambda_k \int_{\tilde{X}_k(\tau)} \Big(
\mu_k(\tau)^*\Omega \Big) \wedge u_k
- 2 \epsilon_k \int_{X_{\tau}} \Omega \wedge \omega_{\tau}\\
=& \int_{\tilde{X}_k(\tau)} \Big( \mu_k(\tau)^*\Omega \Big) \wedge
\Big(\tilde{R}_k + 2 \epsilon_k \mu_k(\tau)^* \omega_{\tau}
- \epsilon_k \lambda_k u_k \Big) \\
&+ \epsilon_k \lambda_k \int_{\tilde{X}_k(\tau)} \Big(
\mu_k(\tau)^*\Omega \Big) \wedge \big[ \tilde{Y}_k(\tau) \big]
- 2 \epsilon_k \int_{X_{\tau}} \Omega \wedge \omega_{\tau}.\\
\end{aligned} \]
It is clear that \[ \Big(\tilde{R}_k + 2 \epsilon_k \mu_k(\tau)^*
\omega_{\tau} - \epsilon_k \lambda_k u_k \Big) = \Big(\tilde{R}_k +
\epsilon_k \mu_k(\tau)^* \omega_{\tau} \Big) + \epsilon_k \Big(
\mu_k(\tau)^* \omega_{\tau} - \lambda_k u_k \Big) \] is a K\"ahler
metric on $\tilde{X}_k(\tau)$ for each $k$. Then it follows that
\[ \int_{\tilde{X}_k(\tau)} \Big( \mu_k(\tau)^*\Omega \Big) \wedge \Big(\tilde{R}_k
+ 2 \epsilon_k \mu_k(\tau)^* \omega_{\tau} - \epsilon_k \lambda_k
u_k \Big) = \lim\limits_{t \> \tau} \int_{\tilde{X}_k(t)} \Big(
\mu_k(t)^*\Omega_t \Big) \wedge \tilde{\omega}_k(t) \geq 0,\] where
$\tilde{\omega}_k(t)$ is a family of K\"ahler metrics on
$\tilde{X}_k(t)$, starting with $$\Big(\tilde{R}_k + 2 \epsilon_k
\mu_k(\tau)^* \omega_{\tau} - \epsilon_k \lambda_k u_k \Big)$$ and
varying smoothly with respect to $t$, from the stability theorem of
K\"{a}hler structures \cite{KS}. Moreover, the integral
$\int_{\tilde{X}_k(t)} \Big( \mu_k(t)^*\Omega_t \Big) \wedge
\tilde{\omega}_k(t)$ only depends on the Aeppli class of $\mu_k(t)^*
\Omega_t$ and $\big[\mu_k(t)^* \Omega_t \big]_{\mathrm{A}}$
converges to $\big[\mu_k(\tau)^* \Omega \big]_{\mathrm{A}}$ when $t
\> \tau$. Similarly, we can get that
\[ \epsilon_k \lambda_k \int_{\tilde{X}_k(\tau)} \Big( \mu_k(\tau)^*\Omega \Big) \wedge \big[ \tilde{Y}_k(\tau) \big]
= \epsilon_k \lambda_k \lim\limits_{t \> \tau} \int_{\tilde{X}_k(t)}
\Big( \mu_k(\tau)^*\Omega_t \Big) \wedge \big[ \tilde{Y}_k(t) \big]
\geq 0, \] where $\tilde{Y}_k(t)$ is the extension of
$\tilde{Y}_k(\tau)$ to the $t$-fiber $\tilde{X}_k(t)$ of the total
space $\tilde{\mathcal{X}}_k$. Based on these two inequalities above,
one has
\[ \int_{\tilde{X}_k(\tau)} \Big( \mu_k(\tau)^*\Omega \Big) \wedge \tilde{R}_k
\geq -\epsilon_k \int_{X_{\tau}} \Omega \wedge \omega_{\tau}. \]
Therefore, the statement \eqref{fst} is proved.

For the statement \eqref{scd}, let us recall that every analytic
irreducible subvariety of complex codimension $n-p$ in $X_{\tau}$
can be extended into any other fiber in the family $\pi: \mathcal{X}
\rightarrow \Delta_{\epsilon}$ with the invariance of its cohomology
class, from Barlet's theory of analytic cycle discussed above.
Especially, the irreducible components $Y_{kj}$ of complex
codimension one of $Y_k$ on $X_{\tau}$ can be extended to the ones
$Y_{kj}(t)$ on the $t$-fiber $X_t$, which are contained in $Y_k(t)$.
Then it is easy to see that
\[ \int_{X_{\tau}} \Omega \wedge \big[ Y_{kj} \big]
= \lim_{t\>\tau} \int_{X_t} \Omega_t \wedge \big[ Y_{kj}(t)\big]
\geq 0. \] The statement \eqref{scd} is also proved.

Together with these two statements and \eqref{estimate}, it is clear
that
\[  \int_{X_{\tau}} \Omega \wedge T \geq -\epsilon_k \int_{X_{\tau}} \Omega \wedge \omega_{\tau}, \]
for each $k$. Then it follows that
\[ \int_{X_{\tau}} \Omega \wedge T \geq 0, \]
where $T$ is any fixed $d$-closed positive $(1,1)$-current on
$X_{\tau}$. Proposition \ref{closed-dual} assures the inclusion: for
$\tau \in \Delta_{\epsilon} \setminus \bigcup S_{\nu}$,
\[  \lim_{t \> \tau} \mathcal{G}_{X_t} \subseteq \overline{\mathcal{G}}_{X_{\tau}}. \]
\end{proof}

\begin{theorem}\label{inclusion_2}
Let $\pi: \mathcal{X} \rightarrow \Delta_{\epsilon}$ be a
holomorphic family with fibers all K\"{a}hler manifolds. For some
$\tau \in \Delta_{\epsilon}$, the fiber $X_{\tau}$ admits the
equality $\overline{\mathcal{K}}_{X_{\tau}}=\mathcal{E}_{X_{\tau}}$.
Then the inclusion holds:
\[ \lim_{t \> \tau} \mathcal{G}_{X_t}
\subseteq \overline{\mathcal{G}}_{X_{\tau}}. \] In particular, the
fiber $X_{\tau}$ with nef holomorphic tangent bundle
$T^{1,0}_{X_{\tau}}$ satisfies the inclusion above.
\end{theorem}

\begin{proof}
The condition
$\overline{\mathcal{K}}_{X_{\tau}}=\mathcal{E}_{X_{\tau}}$ implies
that, for any $d$-closed positive $(1,1)$-current $T$ and arbitrary
$\delta >0$, there exists a smooth $(1,1)$-form $\alpha_{\delta}$,
which lies in the Bott-Chern class $\big[T\big]_{\mathrm{BC}}$, such
that $$\alpha_{\delta} \geq - \delta \omega_{\tau},$$ where
$\omega_{\tau}$ is the fixed K\"ahler metric of $X_{\tau}$.

Fix an element $\big[\Omega\big]_{\mathrm{A}}$ of $\lim\limits
_{t\>{\tau}}{\mathcal{G}}_{X_t}$, which means that
\[\mathrm{P}_t \circ \mathrm{Q}_{\, \tau} \Big( \big[\Omega\big]_{\mathrm{A}} \Big) \in
\mathcal{G}_{X_t} \quad \textrm{for}\ \
0<|t-\tau|<\delta_{[\Omega]_{\mathrm{A}}}. \] Then for any
$d$-closed positive $(1,1)$-current $T$,
\[ \begin{aligned}
\int_{X_{\tau}} \Omega \wedge T &= \int_{X_{\tau}} \Omega \wedge \alpha_{\delta} \\
& = \int_{X_{\tau}} \Omega \wedge \big( \alpha_{\delta} + 2 \delta
\omega_{\tau}\big) - 2 \delta \int_{X_{\tau}} \Omega \wedge
\omega_{\tau}. \\ \end{aligned}\]

It is clear that $ \alpha_{\tau} + 2 \tau \omega_{\tau}$ is a
K\"ahler metric on $X_{\tau}$, and thus, from the stability theorem
of K\"{a}hler structures \cite{KS}, there exists a family of
K\"ahler metrics $\tilde{\alpha}_{\delta}(t)$ on $X_t$, starting
with $ \alpha_{\tau} + 2 \delta \omega_{\tau}$ and varying smoothly
with respect to $t$. It follows that
\[ \int_{X_{\tau}} \Omega \wedge \big( \alpha_{\delta} + 2 \delta \omega_{\tau} \big)
= \lim_{t\>{\tau}} \int_{X_t} \Omega_t \wedge
\tilde{\alpha}_{\delta}(t) \geq 0,\] since the integral also depends
on the Aeppli class of $\Omega_t$ and $\Omega_t$ is the positive
representative in $\mathrm{P}_t \circ \mathrm{Q}_{\,\tau} \Big(
\big[\Omega\big]_{\mathrm{A}} \Big)$ for each $t \neq \tau$. As
$\delta$ can be arbitrarily small, we have
\[ \int_{X_{\tau}} \Omega \wedge T \geq 0, \]
which assures that $\big[\Omega \big]_{\mathrm{A}} \in
\overline{\mathcal{G}}_{X_{\tau}}$ by Proposition \ref{closed-dual}.
If a compact complex manifold has nef holomorphic tangent bundle,
the nef cone and the pseudo-effective cone coincide by
\cite[Corollary 1.5]{D1}. Therefore, the proofs are completed.
\end{proof}

\end{document}